\documentclass[a4paper,fullpage]{scrartcl}

\usepackage{graphicx}
\usepackage[utf8]{inputenc}
\usepackage[english]{babel}

\usepackage{amsmath,amssymb,amsthm,enumerate,url,appendix,tabularx,stmaryrd,mathrsfs}
\usepackage{todonotes,comment,afterpage}
\usepackage{subcaption}
\usepackage{framed}
\usepackage{mathtools}
\usepackage{hyperref}

\theoremstyle{plain}
\newtheorem{theorem}{Theorem}[section]
\newtheorem{lemma}[theorem]{Lemma}
\newtheorem{corollary}[theorem]{Corollary}
\newtheorem{proposition}[theorem]{Proposition}

\newtheorem{assumption}[theorem]{Assumption}

\newtheorem{definition}[theorem]{Definition}
\newtheorem{example}[theorem]{Example}

\newtheorem{remark}[theorem]{Remark}

\numberwithin{equation}{section}
\numberwithin{figure}{section}
\numberwithin{table}{section}

\newcommand{\R}{\mathbb{R}}
\newcommand{\C}{\mathbb{C}}
\newcommand{\N}{\mathbb{N}}

\newcommand{\bigO}{\mathcal{O}}

\newcommand{\microspace}{\mspace{0.5mu}}

\newcommand{\abs}[1]{\left|#1\right|}
\newcommand{\norm}[1]{\left\|#1\right\|}
\newcommand{\triplenorm}[1]{
  \lvert\!\microspace\lvert\!\microspace\lvert #1 
  \rvert\!\microspace\rvert\!\microspace\rvert}

\newcommand{\dualproduct}[3][\Gamma]{\left<#2,#3 \right>_{#1}}
\newcommand{\colvec}[1]{\begin{pmatrix}#1\end{pmatrix}}

\newcommand{\laplace}{\Delta}

\def\TT{\mathcal{T}}

\def\ii{i}

\newcommand{\hone}[1][\R^d]{H^{1}(#1)}
\newcommand{\hdiv}[1][\R^d]{H(\mathrm{div},#1)}
\newcommand{\HH}{\mathcal H}
\newcommand{\HHdiv}{\mathcal H^{\mathrm{div}}}
\newcommand{\XX}{\mathcal X}
\newcommand{\YY}{\mathcal Y}
\newcommand{\hH}{\mathbb H}
\newcommand{\vV}{\mathbb V}
\newcommand{\mM}{\mathbb M}
\newcommand{\LL}{\mathcal L^2}
\newcommand{\bsLL}{\bs L^2}

\newcommand{\hppw}[1]{H^{#1}_{\operatorname{pw}}}

\newcommand{\prodtrace}{\boldsymbol{\gamma}}
\newcommand{\prodnormaltrace}{\boldsymbol\gamma_\nu}

\newcommand{\traceint}[1]{{\gamma}_{#1}^{\mathrm{int}}}
\newcommand{\normaltraceint}[1]{\gamma_{\nu,#1}^{\mathrm{int}}}
\newcommand{\traceext}[1]{{\gamma}_{#1}^{\mathrm{ext}}}
\newcommand{\normaltraceext}[1]{\gamma_{\nu,#1}^{\mathrm{ext}}}

\newcommand{\prodtraceint}{\boldsymbol{\gamma}^{\mathrm{int}}}
\newcommand{\prodnormaltraceint}{\boldsymbol{\gamma}_{\nu}^{\mathrm{int}}}
\newcommand{\prodtraceext}{\boldsymbol{\gamma}^{\mathrm{ext}}}
\newcommand{\prodnormaltraceext}{\boldsymbol{\gamma}_{\nu}^{\mathrm{ext}}}

\newcommand{\tracejump}[2]{\llbracket \gamma_{#1} #2 \rrbracket}
\newcommand{\normaltracejump}[2]{\llbracket \gamma_{\nu,#1} #2\rrbracket}
\newcommand{\tracemean}[2]{\{\!\{ \gamma_#1 #2\} \!\}}
\newcommand{\normaltracemean}[2]{\{\!\{ \gamma_{\nu,#1} #2 \} \!\}}
\newcommand{\prodtracejump}[1]{\llbracket \prodtrace #1  \rrbracket}
\newcommand{\prodnormaltracejump}[1]{\llbracket \prodnormaltrace #1  \rrbracket}
\newcommand{\prodtracemean}[1]{\{\!\{ \prodtrace #1 \}\!\}}
\newcommand{\prodnormaltracemean}[1]{\{\!\{ \prodnormaltrace #1 \}\!\}}

\newcommand{\prodnormaldermean}[1]{\{\!\{ \boldsymbol\partial_\nu #1\}\!\}}
\newcommand{\prodnormalderjump}[1]{\llbracket\boldsymbol\partial_\nu #1\rrbracket}
\newcommand{\prodnormalderext}[1]{\boldsymbol{\partial}_{\nu}^{\mathrm{ext}}}

\newcommand{\bs}[1]{\boldsymbol #1}

\newcommand{\ones}{\mathbf{1}}

\newcommand{\dom}{\operatorname{dom}}

\def\rkA{\mathcal{Q}}
\def\rkb{\mathbf{b}}
\def\rkc{\mathbf{c}}
\def\lifting{\mathscr{E}}

\def\dd{\partial_k}

\def\postproc{\mathbb{P}}

\newcommand{\Cpspace}[2][\mathcal{X}_{\mu}]{\mathcal{C}^{#2}\left(\left[0,T\right],#1\right)}

\newcommand{\fdX}{{\bs X}^{h,k}}

\newcommand{\fdx}{\bs x^{h,k}}
\newcommand{\fdu}{ \bs u^{h,k}}

\newcommand{\fdw}{\bs w^{h,k}}

\newcommand{\fdl}{ \bs \phi^{h,k}}
\newcommand{\fdp}{ { \bs \psi^{h,k}}}

\def\AA{\bs A}
\def\AAstar{\AA_{\star}}
\def\BB{\bs B}

\def\uinc{u^{\text{inc}}}

\def\id{I}

\newcommand{\prodOp}[1]{ \breve{#1}}
\newcommand{\prodSpace}[1]{\left[#1\right]^m}

\newcommand{\fdiv}{\nabla \cdot}

\newcommand{\eremk}{\hbox{}\hfill\rule{0.8ex}{0.8ex}}

\newcommand{\Rbar}{\underline{R}_k}
\newcommand{\opT}{\mathcal{T}}
\newcommand{\opC}{\mathcal{C}}

\iffalse 
\usepackage{tikz}
\usetikzlibrary{shapes}
\usepackage{pgfplots}
\usetikzlibrary{backgrounds}
\pgfplotsset{compat=1.15}

\newcommand{\includeTikzOrEps}[1]{\tikzexternalenable \tikzsetnextfilename{#1}  {\include{figures/#1}} \tikzexternaldisable}

\usetikzlibrary{external}
\tikzexternaldisable

\else
\newcommand{\includeTikzOrEps}[1]{\includegraphics{figures_pdf/#1}}
\fi

\usetikzlibrary{arrows.meta}

\title{Time domain boundary integral equations and convolution quadrature for scattering by composite media
(extended preprint)}
\author{Alexander Rieder
  \thanks{Institut f\"ur Analysis und Scientific Computing, TU Wien, 1040 Vienna, Austria\newline E-mail: {\textrm{ alexander.rieder@tuwien.ac.at}}},
  Francisco--Javier Sayas \thanks{Department of Mathematical Science, University of Delaware, passed away April 2, 2019
  }, Jens Markus Melenk \thanks{Institut f\"ur Analysis und Scientific Computing, TU Wien, 1040 Vienna, Austria\newline E-mail: {\textrm{ melenk@tuwien.ac.at}}}}

\date{\today}

\begin{document}

\maketitle
\begin{abstract}
  We consider acoustic scattering in heterogeneous media with piecewise constant wave number. The discretization is carried out
  using a Galerkin boundary element method in space and Runge-Kutta convolution quadrature in time.
  We prove well-posedness of the scheme and provide \textsl{a priori} estimates for the convergence in
  space and time.
\end{abstract}

\section{Introduction}
\label{sect:introduction}
A basic problem in wave propagation is that of scattering in heterogeneous media. A prominent example 
is the classical inverse problem of seismic analysis, where one aims at understanding the structure of a 
medium from the scattered fields of impinging waves. Such an analysis requires efficient methods for the so-called 
forward problem, in which the heterogeneous medium is assumed given and the scattering field of impinging waves is 
computed. In this setting, an important problem class, which is considered in the present work, is that of piecewise homogeneous media. 

When considering piecewise constant material parameters, time domain boundary integral equations (TDBIE) can be applied since 
fundamental solutions for the wave equation are available. A particular strength of boundary integral techniques 
is that they allow for a convenient treatment of unbounded domains, which appear frequently in scattering problems. 

In order to treat the scattering from heterogeneous media embedded in an unbounded homogeneous medium, there are in fact 
several possibilities. One of the more common approaches is to combine the boundary element method in the exterior with 
a finite element method for a bounded domain. This approach was taken in~\cite{bls_fembem} and~\cite{sayas_hassel_fembem}.
For the Schr\"odinger equation, a similar approach was investigated in~\cite{schroedinger}.
\cite{AJRT11} combines a discontinuous Galerkin method with a boundary element method.
Another approach that is suitable for the case of piecewise constant material properties is to use boundary integral equations 
for each subdomain and suitably couple them.  In the context of time-harmonic scattering this approach has been pioneered 
by Costabel and Stephan~\cite{costabel_stephan_trasmission} in the case of a single scatterer and by von~Petersdorff 
for multiple scatterers \cite{von_petersdorff}. These approaches have later been extended 
in~\cite{claeys_hiptmair_2013,claeys2011single,hiptmair_hanckes}.  The case of time-dependent scattering has thus far 
seen less attention, although the case of a single scatterer embedded in a homogeneous medium was treated 
in \cite{tianyu_sayas_costabel}.

For the discretization of the time variable, a variety of approaches have been developed in the past. The oldest one 
is based on a space-time formulation involving the retarded potentials, \cite{BamH,BamH2,GMOSS18}. Another common approach
is based on Lubich's convolution quadrature~\cite{lubich_cq1,lubich_cq2} (CQ) and its Runge-Kutta 
variation (RK-CQ), introduced in~\cite{lubich_ostermann_rk_cq}. The present work takes this route and analyzes 
an RK-CQ. 

In the treatment of TDBIEs using CQ methods the case of scattering by a single impenetrable obstacle (using different boundary conditions to account 
for different material behaviors) has garnered a lot of attention, and the available numerical methods can be considered 
well developed, see \cite{banjai_sauter_rapid_wave,banjai_laliena_sayas_kirchhoff_formulas,dominguez_sayas_2013} 
and the comprehensive treatment in \cite{sayas_book}. 
Recently, these methods have even been extended to a  class of nonlinear scattering problems,
see \cite{banjai_and_me,banjai_lubic_nonlinear_wave_rkcq}.

In this paper, we present and analyze a fully discrete formulation of the multiple-subdomain acoustic scattering problem based on a Galerkin boundary element method
and RK-CQ for the time discretization. Our analysis is based on a pure time-domain point of view, 
combining ideas by \cite{banjai_laliena_sayas_kirchhoff_formulas} and~\cite{sayas_new_analysis} with the theory of Runge-Kutta approximations
of abstract semigroups, as laid out in \cite{mallo_palencia_optimal_orders_rk} and recently extended in~\cite{semigroups}.
A main contribution of the present work is that our analysis covers scattering problems by piecewise constant materials with a very 
    general layout of subdomains.
  Most notably, in comparison to \cite{tianyu_sayas_costabel} 
    we allow for more than one subdomain and permit cross points where more than
    two subdomains touch. Our approach therefore generalizes the
    results of~\cite[Chapters 3 and 4]{qiu_thesis}, which only allows certain nested geometries (see Section~\ref{sect:particular_configurations}).
    Compared to
    other works, e.g., \cite{tianyu_sayas_costabel,qiu_thesis} we also
    consider RK-CQ using the novel time-domain analysis developed in~\cite{semigroups}, whereas previous analyses concentrated
    on multistep methods, whose order, however, is limited to $2$ if A-stability is required. 
    These higher order RK-methods suffer from some reduction of order phenomenon in that the convergence order
    falls somewhere between the stage- and classical order of the RK-method. By
    careful analysis of the regularity of certain lifting problems, we are able to establish 
    an improved convergence by $k^{1/2}$ compared to a more straight-forward analysis,
    as long as the mild assumption is made that the incident wave is in $L^2(\partial \Omega_0)$.

  We analyze several RK-CQ formulations for a scattering problem. For a slightly non-standard formulation based on 
  differentiating the Dirichlet data we show that a higher order of convergence can be achieved than for the more standard
  one based on using same order of differentiation of the Dirichlet and Neumann data.

Related to our approach is the recent \cite{EFHS19}, which studies, on the continuous level, well-posedness of certain TDBIEs for 
acoustic scattering problems.  \cite{EFHS19} considers the case of two subdomains (plus the exterior) endowed with suitable 
transmission conditions and general boundary conditions.  Their approach relies on frequency domain estimates.  
  In contrast, our analysis includes a fully discrete convergence analysis for more complicated geometric situations of arbitrary
number of subdomains. In order to do so, we use the novel ``time-domain only approach'' developed
along side this paper and presented recently in~\cite{semigroups}, showcasing
that this approach is feasible for complex model problems that go beyond rather simple ones.
Compared to the Laplace domain approach, the pure time domain theory offers several advantages.
  Firstly, at least for the spatial discretization, it leads to sharper estimates with lower
  regularity requirements on the input data. It then makes sense to stay in the time-domain
  also for the CQ-analysis.  
  Secondly, it allows for sharper control on how the estimates degenerate as time grows.
  Thirdly it, to some degree, allows for a wider set of Runge-Kutta methods. Namely,
  the estimates on the $H^1$-norm and the Dirichlet-trace of the post-processed solution
  also hold, for example, for the Gauss-methods.
  Finally, the theory better emphasizes the dynamical nature of the underlying problem,
  whereas this remains opaque when only analyzing transfer functions  for the
  boundary integral operators.

The paper is structured as follows. In Section~\ref{sect:model_problem} we present the details of the model problem under consideration.
We reformulate the problem in the language of $C_0$-semigroups and prove well-posedness. (In order to streamline the presentation,
all this is done in a semidiscrete setting that takes into account the Galerkin discretization in space).
Section~\ref{sect:integral_equations} presents a boundary integral formulation and establishes equivalence in the fully continuous and semidiscrete settings.
Section~\ref{sect:rkcq} deals with the discretization of the time-variable using Runge-Kutta based convolution quadrature and gives the final 
fully discrete scheme for the scattering problem. We give explicit error bounds for the convergence in space and time.
Section~\ref{sect:particular_configurations} relates our results to the existing literature by showing equivalences in certain simpler geometric situations.
In Section~\ref{sect:numerics} we give  numerical examples in 2D.

Compared to the published version, this extended preprint contains the additional Appendix
\ref{sect:local_higher_order},
which details how to prove full classical convergence rates away from the boundary
and in a pointwise setting; see also Theorem~\ref{thm:local_and_pointwise_convergence}.

We close with a remark on notation. Throughout this article we will encounter collections of functions on different levels.
Functions defined on a single subdomain  will be denoted by regular lowercase characters. For collections
of such functions for multiple subdomains, we will use bold characters. When discretizing in time using an $m$-stage Runge-Kutta method,
we will add the superscript $k$ to all quantities.
For each scalar quantity, we obtain a stage vector of $m$ functions. These will be denoted by uppercase letters, the corresponding (scalar) approximations
at the time-steps will then again be the same lowercase letter.
For example, starting from scalar functions $u_\ell$ on $\Omega_{\ell}$, collecting them gives $\bs u:=(u_{\ell})_{\ell=0}^{L}$.
The stage vector of their RK-approximation is then $\bs U^{k}$ and the scalar approximation will be $\bs u^{k}$.
The same rules will be applied to functions defined on the boundary of subdomains, except that we will use the Greek alphabet.

\section{Model problem and notation}
\label{sect:model_problem}

We consider the scattering of waves from one or multiple objects, with possibly adjacent parts and different material properties. 
We are given mutually disjoint bounded Lipschitz domains $\Omega_\ell \,\subseteq \R^d$, $\ell=1,\dots,L$, and we use
\[
 \Omega_0:=\R^{d}\setminus \bigcup_{\ell=1}^{L}{\overline{\Omega_\ell}},
\qquad
 \Gamma:=\bigcup_{\ell=1}^{L} \partial \Omega_{\ell},
\]
to respectively denote the unbounded exterior domain (which might be disconnected) and the union of the boundaries of all the domains. In physical terms, the scatterer occupies the closed set $\cup_{\ell=1}^L \overline{\Omega_\ell}$, while $\Omega_0$ is the surrounding medium. The set $\Gamma$ will be called the skeleton of the partition of the scatterer. The acoustic behavior of the surrounding domain and the scatterer is described with two piecewise constant positive functions $\kappa, c:\mathbb R^d\to (0,\infty)$ given by
\[
\kappa|_{\Omega_\ell}\equiv\kappa_\ell > 0,
\qquad
c|_{\Omega_\ell}\equiv c_\ell > 0, \qquad \ell=0,\ldots,L.
\]

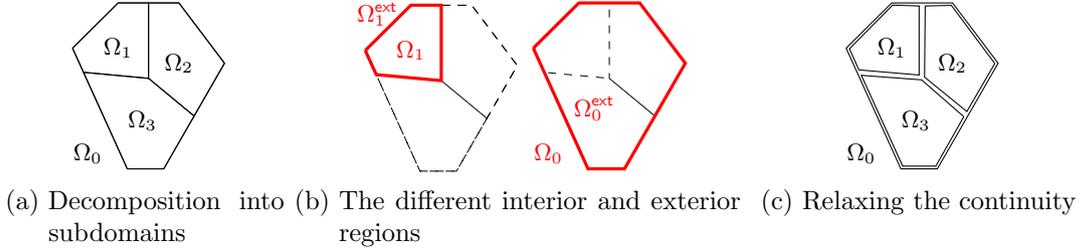
\begin{figure}[htb]  
  \begin{subfigure}[t]{0.25\textwidth}
    \begin{center}
      \footnotesize
    \begin{tikzpicture}[scale=0.4]
\draw (0,4) -- (1.5,5.5) -- (2.5,5.5) -- (3.5,5.5)  -- 
(5,3.5) -- (4,1.75) -- (3,0) -- (1.8,0) -- (0.36,3.2) -- (0,4);
\draw (0,4) -- (1.5,5.5) -- (2.5,5.5)  -- (2.5,3) -- (0.36,3.2) -- (0,4);
\draw (2.5,5.5) -- (3.5,5.5)  -- (5,3.5) -- (4,1.75)  -- (2.5,3) -- (2.5, 5.5) ;
\draw (4,1.75) -- (3,0) -- (1.8,0) -- (0.36,3.2) -- (2.5,3) -- (4,1.75);
\draw (1.5,4) node {$\Omega_1$};
\draw (3.5,3.5) node {$\Omega_2$};
\draw (2.3,1.6) node {$\Omega_3$};
\draw (0.5,0.5) node {$\Omega_0$};
\end{tikzpicture}
\caption{Decomposition into subdomains}
\end{center}
\end{subfigure}
\begin{subfigure}[t]{0.4\textwidth}
  \begin{center}
    \footnotesize
\begin{tikzpicture}[scale=0.4]
\draw[very thin, dashed] (0,4) -- (1.5,5.5) -- (2.5,5.5) -- (3.5,5.5)  -- 
(5,3.5) -- (4,1.75) -- (3,0) -- (1.8,0) -- (0.36,3.2) -- (0,4);
\draw[very thin, dashed] (2.5,5.5) -- (3.5,5.5)  -- (5,3.5) -- (4,1.75)  -- (2.5,3) -- (2.5, 5.5) ;
\draw[very thin, dashed] (4,1.75) -- (3,0) -- (1.8,0) -- (0.36,3.2) -- (2.5,3) -- (4,1.75);
\draw[red, very thick] (0,4) -- (1.5,5.5) -- (2.5,5.5)  -- (2.5,3) -- (0.36,3.2) -- (0,4);
\draw[red] (1.5,4) node {$\Omega_1$};
\draw[red] (0.4,5.2) node {$\Omega_1^{\mathsf{ext}}$};
\end{tikzpicture}
\begin{tikzpicture}[scale=0.4]
\draw[very thin, dashed] (2.5,5.5) -- (3.5,5.5)  -- (5,3.5) -- (4,1.75)  -- (2.5,3) -- (2.5, 5.5) ;
\draw[very thin, dashed] (4,1.75) -- (3,0) -- (1.8,0) -- (0.36,3.2) -- (2.5,3) -- (4,1.75);
\draw[red, very thick] (0,4) -- (1.5,5.5) -- (2.5,5.5) -- (3.5,5.5)  -- 
(5,3.5) -- (4,1.75) -- (3,0) -- (1.8,0) -- (0.36,3.2) -- (0,4);
\draw[red] (0.5,0.5) node {$\Omega_0$};
\draw[red] (2,2) node {$\Omega_0^{\mathsf{ext}}$};
\end{tikzpicture}
\caption{The different interior and exterior regions}
\end{center}
\end{subfigure}
\begin{subfigure}[t]{0.3\textwidth}
  \begin{center}
\footnotesize
\begin{tikzpicture}[scale=0.4]
\draw (0,4) -- (1.5,5.5) -- (2.5,5.5) -- (3.5,5.5)  -- 
(5,3.5) -- (4,1.75) -- (3,0) -- (1.8,0) -- (0.36,3.2) -- (0,4);
\draw (0.11,3.98) -- (1.53,5.43) -- (2.40,5.40)  -- (2.40,3.11) -- (0.42,3.28) -- (0.11,3.98);
\draw (2.60,5.40) -- (3.46,5.42)  -- (4.89,3.49) -- (3.97,1.90)  -- (2.58,3.04) -- (2.60, 5.40) ;
\draw (3.87,1.73) -- (2.95,0.09) -- (1.86,0.09) -- (0.50,3.10) -- (2.48,2.95) -- (3.87,1.73);
\draw (1.5,4) node {$\Omega_1$};
\draw (3.5,3.5) node {$\Omega_2$};
\draw (2.3,1.6) node {$\Omega_3$};
\draw (0.5,0.5) node {$\Omega_0$};
\end{tikzpicture}
\caption{Relaxing the continuity}
\label{fig:broken_interfaces}
\end{center}
\end{subfigure}
\caption{Geometry of the problem and notation for subdomains}
\end{figure}

\subsection{The scattering problem}
We now give the formal definition of the model problem. We will later on encounter
  other, equivalent, formulations.  We start with an incident wave.
Generally thinking of transient plane waves, we assume that we are given a sufficiently smooth function $u^{\mathrm{inc}}:\mathbb R\to H^1_{\mathrm{loc}}(\Omega_0)$ such that
\begin{subequations}
\begin{alignat}{6}
\label{eq:2.0}
c_0^{-2} \ddot{u}^{\mathrm{inc}}(t) =\kappa_0 \laplace u^{\mathrm{inc}}(t) & \qquad & \text{in $\Omega_0$ }
  \forall t\in \mathbb R,\\
\mathrm{supp}\, u^{\mathrm{inc}}(t) \subseteq \Omega_0 & & \forall t\le 0.
\end{alignat}
\end{subequations}
(For cylindrical or spherical  incident waves, a source term, supported strictly in $\Omega_0$, has to be added in \eqref{eq:2.0}.)
The total wave field is then a function $u^{\text{tot}}:\mathbb R\to H^1_{\mathrm{loc}}(\R^d)$ satisfying
\begin{alignat*}{6}
  c^{-2}\ddot u^{\text{tot}}(t)=\nabla\cdot (\kappa \nabla u^{\text{tot}})(t)  &\qquad &
  \text{in $\R^d \setminus \Gamma$ }\forall t\in \mathbb R,\\
u^{\text{tot}}(t)=u^{\mathrm{inc}}(t) & & \mbox{in $\Omega_0$} \quad \forall t\le 0,\\
\mathrm{supp}\, (u^{\text{tot}}-u^{\mathrm{inc}})(t) \mbox{ is bounded} & & \forall t\in \mathbb R.
\end{alignat*}
The condition $u^{\text{tot}}(t)\in H^1_{\mathrm{loc}}(\R^d)$ implies that the traces of $u^{\text{tot}}(t)$ do not jump across $\Gamma$. The fact that we are applying the divergence operator to $\kappa\nabla u^{\text{tot}}(t)$ in $\mathbb R^d$ implies that the normal components of $\kappa\nabla u^{\text{tot}}(t)$ do not jump across $\Gamma$.

We will write the problem in terms of the scattered wave $u:=u^{\text{tot}}-u^{\mathrm{inc}}$ and restricted to the time interval $[0,\infty)$. The vanishing values of $u$ for negative times will make a reappearance once the retarded potentials are introduced. To introduce this formulation, while avoiding to deal with the possibly complicated forms for the intersections of the boundaries of the subdomains $\Omega_\ell$, we proceed as follows. We first extend $u^{\mathrm{inc}}$ to the equally named function $u^{\mathrm{inc}}:[0,\infty)\to L^2(\R^d)$ by setting $u^{\mathrm{inc}}(t)\equiv 0$ in $\R^d\setminus\Omega_0$ for all $t$.
This is mostly for notational convenience, as it allows us to write 
the equations for $u^{\mathrm{tot}}=u+u^{\mathrm{inc}}$ in a concise way such as in \eqref{eq:1} below. To that end, 
 we introduce an arbitrary open ball $B$ that contains the skeleton $\Gamma$. The scattered field is then a function $u:[0,\infty)\to H^1(\R^d\setminus\Gamma)$ satisfying
\begin{subequations}\label{eq:1}
\begin{alignat}{6}
c^{-2} \ddot u(t)= \nabla\cdot(\kappa \nabla u)(t) & \qquad & \mbox{in $\R^d\setminus\Gamma$} \quad & \forall t\ge 0,\\
\label{eq:1b}
u(t)+u^{\mathrm{inc}}(t)\in H^1(B)  & & &  \forall t\ge 0, \\
\label{eq:1c}
\kappa\nabla (u(t)+u^{\mathrm{inc}}(t))\in H(\mathrm{div},B)
&  & &  \forall t\ge 0, 
\end{alignat} 
with vanishing initial conditions
\begin{equation}
u(0)=0, \qquad \dot u(0)=0.
\end{equation}
\end{subequations}

\subsection{A multiply overlapped wave problem}

In this section we will formulate a generalization of problem \eqref{eq:1} including some sort of partial observation of transmission conditions on the skeleton. We will end up having $L+1$ fields $u_0,\ldots, u_L$, where $u_\ell|_{\Omega_\ell}$ will be, in a sense to be made precise later, an approximation of $u|_{\Omega_\ell}$. The transmission conditions implicit in equations \eqref{eq:1b} and \eqref{eq:1c} will be relaxed and, at the same time, given a trace operator-based formulation. A rigorous formulation of the problem will use a considerable collection of spaces and operators, which we now introduce:
\begin{enumerate}
\item Trace operators
\[
\traceint{\ell}, \traceext{\ell}, \tracejump{\ell}{\,\cdot}, \tracemean{\ell}{\,\cdot}:
\hone[\R^d\setminus\partial\Omega_\ell]\to H^{1/2}(\partial\Omega_\ell),
\]
where the interior and exterior traces are self-explanatory (note that for $\partial\Omega_0$, the interior trace is taken from the unbounded domain $\Omega_0$) and
\[
\tracejump\ell u:=\traceint\ell u-\traceext\ell u,
\qquad
\tracemean\ell u:=\tfrac12(\traceint\ell u+\traceext\ell u).
\]
\item Weak normal trace operators
\[
\normaltraceint\ell, \normaltraceext\ell, \normaltracejump\ell{\,\cdot},
\normaltracemean\ell{\,\cdot}:
\hdiv[\R^d \setminus \partial \Omega_{\ell}] \to H^{-1/2}(\partial \Omega_{\ell}),
\]
defined similarly with, e.g.,  $\normaltraceint\ell u = \nu_\ell \cdot u_\ell|_{\Omega_\ell}$ for sufficiently smooth $u$ and 
noting that the normal is always taken to point out of the corresponding domain $\Omega_{\ell}$.
\item Four spaces collecting  $L+1$ fields (scalar or vector-valued) on $\R^d\setminus\partial\Omega_\ell$,
\begin{align*}  
  \HHdiv&:=\prod_{\ell=0}^{L}{\hdiv[\R^d \setminus \partial \Omega_{\ell}]}, &
  \HH&:=\prod_{\ell=0}^{L}{\hone[\R^d \setminus \partial \Omega_{\ell}]},  \\
  \HH^{-1/2}_\Gamma&:=\prod_{\ell=0}^{L}{H^{-1/2}\left(\partial \Omega_\ell\right)}, &
  \HH^{1/2}_\Gamma&:=\prod_{\ell=0}^{L}{H^{1/2}\left(\partial \Omega_\ell\right)},
\end{align*}
endowed with the product norms. The $\HH^{-1/2}_\Gamma \times \HH^{1/2}_\Gamma$ duality will be denoted 
$\langle\cdot,\cdot\rangle_\Gamma$. It extends the usual $L^2$ inner product, i.e., for $\bs u=(u_{\ell})_{\ell=0}^{L}$ and $\bs v=(v_{\ell})_{\ell=0}^{L}$
in $\prod_{\ell=0}^{L}{L^2(\partial \Omega_\ell)}$, it is given by
$$
\langle\bs u, \bs v\rangle_{\Gamma}
:=\sum_{\ell=0}^{L}{\int_{\partial \Omega_\ell} u_{\ell} {v_{\ell}}}.
$$
\item Diagonal operators
\begin{align*}
\prodtraceint, \prodtraceext, \prodtracejump{\,\cdot},
\prodtracemean{\,\cdot} & : \HH\to \HH^{1/2}_\Gamma,\\
\prodnormaltraceint, \prodnormaltraceext, \prodnormaltracejump{\,\cdot},
\prodnormaltracemean{\,\cdot}& : \HHdiv\to \HH^{-1/2}_\Gamma
\end{align*}
\item Single-trace spaces
\begin{subequations}
\begin{align}
\YY :&=\{ (\traceint\ell u)_{\ell=0}^L \,:\, u\in \hone\} \\
\nonumber
& = \{ \bs\psi \in \HH^{1/2}_\Gamma: \; \exists  u \in \hone, \,  
\bs\psi=\prodtraceint  u\},  \label{eq:def_YY}\\
  \XX:&=\{ (\normaltraceint\ell \bs v)_{\ell=0}^L\,:\, \bs v \in \hdiv\}\\
  \nonumber
 &= \{ \bs\phi \in \HH^{-1/2}_\Gamma: \; \exists \bs v \in \hdiv, \,  \bs\phi=\prodnormaltraceint \bs v\},
\end{align}
\end{subequations}
which are closed subspaces of $\HH^{1/2}_\Gamma$ and $\HH^{-1/2}_\Gamma$ respectively.
\end{enumerate}

While the problem is posed on the ``broken'' spaces $\HH$ and $\HHdiv$, 
the spaces $\XX$ and $\YY$ are introduced to enforce continuity conditions across interfaces $\partial \Omega_{\ell} \cap \partial \Omega_k$;
see Figure~\ref{fig:broken_interfaces}.
This is done following the ideas of~\cite{von_petersdorff}, but using notation analogous to \cite{claeys_hiptmair_2013}.
Note the slight abuse of notation in the second definition of both spaces, where the diagonal trace operators are used on a single function, which is assumed to be copied $L+1$ times.

\begin{lemma}[Restricting and gluing]
  \label{lemma:restrict_and_glue}
If $\bs u=(u_\ell)_{\ell=0}^L \in \HH$ satisfies $\prodtracejump{\bs u}\in \YY$ and $\prodtraceext{\bs u}\in \YY$, then the function $u:\R^d\to \R$ defined by $u|_{\Omega_\ell}:=u_\ell |_{\Omega_\ell}$ satisfies $u\in \hone$. Similarly, if $\bs v=(v_\ell)_{\ell=0}^L\in \HHdiv$ satisfies $\prodnormaltracejump{\bs v}\in \XX$ and $\prodnormaltraceext{\bs v}\in \XX$, then the function $v:\R^d\to\R^d$ defined by $v|_{\Omega_\ell}:=v_\ell|_{\Omega_\ell}$ satisfies $v\in \hdiv$.  
\end{lemma}

\begin{proof} The conditions imply  $\prodtraceint \bs u\in \YY$ and $\prodnormaltrace{\bs v}\in \XX$, from where the result follows using the definition of $\YY$ and $\XX$.
\end{proof}

We now take two closed subspaces $\XX_h \subseteq \XX$ and $\YY_h \subseteq \YY$. When $\XX_h$ and $\YY_h$ are finite dimensional they will play the role of approximation spaces for a Galerkin semidiscretization in space of an equivalent time domain boundary integral formulation.
In order to succinctly write down Galerkin orthogonalities, we define the polar sets 
\begin{align*}
  \XX_h^{\circ}&:=\{ \bs\psi \in \HH^{1/2}_\Gamma: \;\; \dualproduct{\bs\mu}{\bs\psi}=0 \;\forall \bs\mu \in \XX_h \}, \\
  \YY_h^{\circ}&:=\{ \bs\phi \in \HH^{-1/2}_\Gamma: \dualproduct{\bs\phi}{\bs\eta}=0 \;\forall \bs\eta \in \YY_h \}.
\end{align*}
When $\XX_h=\XX$ and $\YY_h=\YY$ it can be proved (see \cite[Prop.~{2.1}]{claeys2011single})
that $\XX^\circ=\YY$ and $\YY^\circ=\XX.$ In particular
\begin{equation}\label{XYpolar}
\dualproduct{\bs\phi}{\bs\psi}=0 \qquad \forall \bs\psi\in \YY, \quad \bs\phi\in \XX.
\end{equation}
This also implies that $\XX_h\subseteq \XX=\YY^\circ\subseteq \YY_h^\circ$ and likewise $\YY_h\subseteq \XX_h^\circ$.

We are finally ready to introduce the multiply overlapped transmission problem that is the object of the first part of this work. The data are functions $\bs\xi^0:[0,\infty)\to \HH^{1/2}_\Gamma$ and $\bs\xi^1:[0,\infty)\to \HH^{-1/2}_\Gamma$ and we look for  $\bs u^h=(u^h_\ell)_{\ell=0}^L:[0,\infty)\to \HH$ and $\bs w^h=(w^h_\ell)_{\ell=0}^L:[0,\infty)\to \HHdiv$ satisfying the first order system
\begin{subequations}\label{eq:probMain}
\begin{equation}
\label{eq:probMain:a}
\dot u^h_\ell(t)=c_\ell^2 \nabla\cdot w^h_\ell(t),
\qquad
\dot{ w}^h_\ell(t)=\kappa_\ell\nabla u^h_\ell(t),
\qquad \forall t>0, \qquad \ell=0,\ldots,L
\end{equation}
(the differential operators in space are distributional derivatives in $\R^d\setminus\partial\Omega_\ell$), four transmission conditions for all $t\ge 0$
\begin{alignat}{6}
\prodtracejump{\bs u^h}(t)+\bs\xi^0(t)\in \YY_h, & \qquad &
	\prodnormaltracejump{\bs w^h}(t)+\bs\xi^1(t)\in \XX_h, \label{eq:probMain:b}\\
\prodtraceext{\bs u^h}(t)\in \XX_h^\circ, & & 
\prodnormaltraceext{\bs w^h}(t) \in \YY_h^\circ,  \label{eq:probMain:c}
\end{alignat}
and vanishing initial conditions
\begin{equation}
\bs u^h(0)=0, \qquad \bs w^h(0)=0.
\end{equation}
\end{subequations}
The following result clarifies the relation between \eqref{eq:probMain} and \eqref{eq:1}. In what follows we will write
\[
(\partial_t^{-1} f)(t):=\int_0^t f(\tau)\mathrm d\tau,
\]
with integration in the sense of Bochner in the space where $f$ takes values. The characteristic function of the domain $\Omega_\ell$ will be denoted $\chi_{\Omega_\ell}$. Before we state the result, and foreseeing possible confusion with notation in existing literature, let us emphasize that the interior traces from the unbounded domain $\Omega_0$ are coming from inside this domain and the normal vector points towards the scatterer in this case.

\begin{proposition}
\label{prop:2.2}
Let 
\[
\bs\xi^0:=(\traceint0 u^{\mathrm{inc}},0,\ldots,0),
	\qquad
\bs\xi^1:=(\kappa_0 \normaltraceint0\nabla \partial_t^{-1} u^{\mathrm{inc}},0,\ldots,0).
\]
If $(\bs u^h,\bs w^h)$ is a solution to \eqref{eq:probMain} for the choice $\XX_h=\XX$, $\YY_h=\YY$, then $u:[0,\infty)\to H^1(\R^d\setminus\Gamma)$ defined by $u(t)|_{\Omega_\ell}:=u^h_\ell(t)|_{\Omega_\ell}$ for all $t\ge 0$ and $\ell \in \{0,\ldots,L\}$ is a solution to \eqref{eq:1}. Reciprocally, if $u$ solves \eqref{eq:1} and we define
\[
u^h_\ell:=\chi_{\Omega_\ell} u, \qquad  w^h_\ell:=\kappa_\ell \nabla \partial_t^{-1} u^h_\ell, \qquad \ell=0,\ldots,L,
\]
then $((u^h_\ell)_{\ell=0}^L,(w^h_\ell)_{\ell=0}^L)$ is a solution to \eqref{eq:probMain}.
\end{proposition}
\begin{proof}
  We only show that if $(\bs u^h,\bs w^h)$ is a solution to \eqref{eq:probMain},
    then we can reconstruct a solution $u$ to \eqref{eq:1} by setting $u|_{\Omega_\ell} = u_\ell^h$.
    On each subdomain $\Omega_\ell$, differentiating the first equality in~\eqref{eq:probMain:a}
    and inserting the second one gives:
    $$
    \ddot{u}^h_{\ell}=c_{\ell}^2 \, \fdiv  \dot{w}^h_{\ell}
    =c_{\ell}^2 \fdiv \big( \kappa_{\ell} \nabla u_{\ell}^{h}\big).
    $$
    Thus, the reconstructed $u$ solves the PDE in $\R^d \setminus \Gamma$. To see the
    jump condition~\eqref{eq:1b}, we consider
    the function $\widetilde{u}^h_{0}:=u^h_0 + u^{\mathrm{inc}} \chi_{B}$, 
    where $\chi_{B}$ is a cutoff function with compact support which is equal to $1$ in the ball of \eqref{eq:1b}, 
    and the functions $\widetilde{u}^h_{\ell}:=u_{\ell}^h$ for $\ell \geq 1$.
    By~\eqref{eq:probMain:b} and \eqref{eq:probMain:c}, 
    we can apply Lemma~\ref{lemma:restrict_and_glue} to see that the reconstructed function $\widetilde{u}$ 
    defined by $\widetilde{u}|_{\Omega_\ell}:= u_\ell^h$ is in $H^1(\R^d)$.
    Since $u + u^{\mathrm{inc}}$ and $\widetilde{u}$ coincide on the ball $B$ the jump condition~\eqref{eq:1b} follows. 
An analogous computation shows \eqref{eq:1c}.
\end{proof}

\subsection{A particular construction of the approximation spaces}
\label{sect:construction}
Consider the two or three dimensional case, i.e., $d=2$ or $d=3$.
Let us assume that all domains $\Omega_\ell$ are Lipschitz polygons in $\R^2$
or polyhedra in $\R^3$. We can thus separate $\Gamma$ into a finite collection of relatively open flat surfaces (resp. segments) $\Gamma_1,\ldots,\Gamma_M$ so that for all $\ell$, there exists an index set $\mathcal I(\ell)\subseteq \{1,\ldots,M\}$ such that
\[
\partial\Omega_\ell=\cup\{ \overline{\Gamma_i}\,:\, i\in \mathcal I(\ell)\}.
\]
We consider a conforming triangulation of $\Gamma$ that respects the subdivision of $\Gamma$ into the subsets $\Gamma_i$. For instance, we can start with a regular {\em \`a la Ciarlet} partition of the interior of $\overline{\Omega_1}\cup\ldots\cup\overline{\Omega_L}$ into open tetrahedra such that no tetrahedral element intersects $\Gamma$. In particular, this means that $\{ T\in \mathcal T_h\,:\, T\subseteq \Omega_\ell\}$ provides a partition of $\Omega_\ell$ for $\ell\ge 1$. Let then $\Gamma_h$ be the triangulation of $\Gamma$ induced by $\mathcal T_h$. We now consider the following finite dimensional spaces of functions defined on the skeleton:
\begin{alignat*}{6}
\mathcal P_h &:=\{ \phi_h:\Gamma \to \mathbb R\,:\, \phi_h|_e\in
\mathcal P_{\rho}(e) \quad \forall e\in \Gamma_h\},\\
\mathcal Q_h &:=\{\psi_h \in \mathcal C(\Gamma)\,:\, \psi_h|_e\in \mathcal P_{\rho+1}(e)\quad\forall e\in \Gamma_h\},
\end{alignat*}
where $\mathcal P_\rho(e)$ is the space of polynomials of degree up to $\rho$ defined on (tangential coordinates of) $e$. We can easily define
\begin{subequations}
  \label{eq:particular_construction_of_spaces}
\begin{align}  
\YY_h:=\{ (\psi_h|_{\partial\Omega_\ell})_{\ell=0}^L\,:\, \psi_h\in \mathcal Q_h\}\subseteq \YY,
\end{align}
which is isomorphic to $\mathcal Q_h$. To define $\XX_h$, we introduce sign functions handling orientation. For $\ell\ge 0$,  $s_\ell:\Gamma \to \{-1,0,1\}$  is constant on each $\Gamma_i$, $s_\ell\equiv 0$ outside $\partial\Omega_\ell$, and $|s_\ell|\equiv 1$ on $\partial\Omega_\ell$. We then assume that common faces have opposite signs, i.e.,
for $\ell \neq j$:
\[
s_\ell+s_j \equiv 0 \qquad \mbox{on $\Gamma_i$ for all $i\in \mathcal{I}(\ell)\cap \mathcal{I}(j)$}.
\]
These sign functions are easy to construct as follows: we assign a normal vector to each $\Gamma_i$ and then write $s_\ell|_{\Gamma_i}=1$ if the assigned normal is exterior to $\partial\Omega_\ell$ and $s_\ell|_{\Gamma_i}=-1$ otherwise. With these sign functions we can finally define
\begin{align}
\XX_h:=\{ (s_\ell\phi_h|_{\partial\Omega_\ell})_{\ell=0}^L\,:\,
\phi_h\in \mathcal P_h\},
\end{align}
\end{subequations}
which is isomorphic to $\mathcal P_h$. The approximation properties of these spaces are inherited from the approximations of $\mathcal{Q}_h$ and $\mathcal{P}_h$.
The details are in the following proposition:
\begin{proposition}
  Let all $\Omega_{\ell}$, $\ell=1,\ldots,L$, be Lipschitz polygons or polyhedrons in 2d or 3d.
The spaces $\XX_h$ and $\YY_h$ defined in~\eqref{eq:particular_construction_of_spaces} have the following approximation property for every integer $0\le r\le \rho$:
  \begin{subequations}
    \begin{align}
      \inf_{\bs \phi^h \in \XX_h}{\norm{\bs \phi - \bs \phi^h}_{\HH^{-1/2}}}
      &\leq C h^{r+3/2} \sum_{\ell=0}^{L}{\norm{\phi_{\ell}}_{\hppw{r+1}\left(\partial \Omega_{\ell}\right)}}, \\
      \inf_{\bs \psi^h \in \YY_h}{\norm{\bs \psi - \bs \psi^h}_{\HH^{1/2}}}
      &\leq C h^{r+3/2} \sum_{\ell=0}^{L}{\norm{\psi_{\ell}}_{\hppw{r+2}\left(\partial \Omega_{\ell}\right)}}
    \end{align}
  \end{subequations}
  for all $\bs \phi=\left(\phi_{\ell}\right)_{\ell=0}^{L} \in \XX$ with $\phi_{\ell}\in \hppw{r+1}\left(\partial \Omega_{\ell}\right)$
  and all $\bs \psi=\left(\psi_{\ell}\right)_{\ell=0}^{L} \in \YY$ with $\psi_{\ell}\in \hppw{r+2}\left(\partial \Omega_{\ell}\right)$
  and with the additional restriction that the lifting $u \in H^1(\R^d)$ from \eqref{eq:def_YY} is a continuous function on $\Gamma$.
  Here $\hppw{\ell}$ denotes the space of piecewise $H^{\ell}$-functions on each face (see
  \cite[Definition 4.8.48]{book_sauter_schwab} for details).
\end{proposition}
\begin{proof}
  We start with the estimate for $\bs \phi$.  
  For each $i \in \{0,\dots,M\}$, we pick a subdomain $\Omega_{\ell_{i}}$, 
  such that $i \in \mathcal{I}\left(\ell_i\right)$, and set $\phi^h|_{\Gamma_i}:=s_{\ell_i}\Pi_{i}\phi_{\ell}|_{\Gamma_i}$,
  where $\Pi_{i} \phi_{\ell}$
  is the orthogonal projection with respect to the $L^2$-product on $\Gamma_i$ onto the set of discontinuous piecewise polynomials. Since 
  $\mathcal{P}_h$ is only required to be $L^2$-conforming, this defines a function in $\XX_h$
  via $\bs \phi^h:=(\phi^h_{\ell}):=\left(s_{\ell} \phi^{h}|_{\partial \Omega_{\ell}}\right)_{\ell=0}^{L}$.
  It follows from standard estimates (see, e.g., \cite[Thm.~{4.3.20}]{book_sauter_schwab}),
  $$
  \norm{\phi_{\ell_i} - \phi^h_{\ell}}_{L^2(\Gamma_i)}\lesssim h^{r+1} \norm{\phi_{\ell_i}}_{H^{r+1}(\Gamma_i)}.
  $$
  Since the functions $\phi \in \XX$ agree on shared interfaces  up to the changed sign,
  i.e. $\phi_{\ell_i}=s_{\ell} s_{\ell_i} \phi_{\ell}$, it is easy to see that for arbitrary $\ell=0,\dots,L$ 
  \begin{align*}    
    \norm{\phi_{\ell} - \phi^h_{\ell}}_{L^2(\partial \Omega_{\ell})}^2
    &\lesssim \sum_{i=1}^{M}{ \norm{ \phi_{\ell_i} - \Pi_i \phi_{\ell_i}}_{L^2(\Gamma_i)}^2}
    \lesssim h^{r+1} \sum_{i=1}^{M}{\norm{\phi_{\ell_i}}_{H^{r+1}(\Gamma_i)}^2}.
  \end{align*}


  To get an estimate in the $\HH^{-1/2}$-norm, we can use  a standard duality argument, (see \cite[Thm.~{4.3.20}]{book_sauter_schwab}),
  using the fact that $\phi_{\ell} - \phi^h_{\ell}$ is orthogonal to the piecewise polynomials on each face gaining
  an extra factor $\sqrt{h}$ in the process.
  
  For estimating $\bs \psi$, we note that our assumptions on the lifting $u$ implies  that the functions $\psi_{\ell}$ are continuous
  on $\partial \Omega_{\ell}$, most notably at the boundary of the facets.
  Therefore, we may employ a nodal interpolation operator 
 $I_{\ell}$.  It is well known that  if $d\leq 3$
  $$\norm{\psi_{\ell} - I_{\ell} \psi_{\ell}}_{H^{1/2}(\partial \Omega_{\ell})}\lesssim h^{r+3/2} \norm{\psi_{\ell}}_{\hppw{r+2}(\partial \Omega_{\ell})},$$
  see~\cite[Thm.~{4.3.22}]{book_sauter_schwab}.  
  Since the functions $\psi_{\ell}$, $\psi_{k}$ are assumed to be traces of a continuous function $u$,
  they must coincide on $\partial \Omega_{\ell} \cap \partial \Omega_k$.
  This means that the interpolated functions $I_{\ell} \psi_{\ell}$ also coincide on $\partial \Omega_{\ell} \cap \partial \Omega_k$
  or $(I_{\ell} \psi_{\ell})_{\ell=0}^{L} \in \XX_h$.
\end{proof}

\subsection{Towards an analyzable form}
\label{sect:analyzable_form}
For the sake of analysis (also of the forthcoming time discretization), we  find it advantageous to introduce some further notation. In order to not
overwhelm notation, the norms and inner products of $L^2(\Omega)$ and $[L^2(\Omega)]^d$ will be equally denoted $\|\cdot\|_{L^2(\Omega)}$ and $(\cdot,\cdot)_{L^2(\Omega)}$ respectively. We now consider:
\begin{enumerate}[(a)]
\item The product spaces $\LL:=L^2(\R^d)^{L+1}$ and $\bsLL:=(L^2(\R^d)^d)^{L+1}$.
\item The natural componentwise differential operators
$\nabla: \HH \to \bsLL$ and $\nabla\cdot :\HHdiv\to \LL$. 
\item The diagonal scaling operators $\bs T_{c^2}:\LL \to \LL$ and $\bs T_\kappa:\bsLL \to \bsLL $ given by
\[
\bs T_{c^2}(u_\ell)_{\ell=0}^L= \big(c_\ell^2\, u_\ell\big)_{\ell=0}^L,
	\qquad
\bs T_\kappa (w_\ell)_{\ell=0}^L =\big(\kappa_\ell w_\ell\big)_{\ell=0}^L.
\]
\item A second layer of product spaces given by
$\hH:=\LL \times \bsLL$, $\mathbb B:=\HH^{1/2}_\Gamma\times \HH^{-1/2}_\Gamma$, and  $\vV:=\HH\times \HHdiv$, where $\vV$ and $\mathbb B$ are endowed with the product norm, while in $\hH$ we consider the weighted norm
\[
\| (\bs u,\bs w)\|_{\hH}^2
=\|\left((u_\ell)_{\ell=0}^L, (w_\ell)_{\ell=0}^L\right)\|_{\hH}^2:=
\sum_{\ell=0}^L c_\ell^{-2} \| u_\ell\|_{L^2(\R^d)}^2
+\sum_{\ell=0}^L \kappa_\ell^{-1} \| w_\ell\|_{L^2(\R^d)}^2,
\]
  with the associated inner product given by
  \begin{align*}
    \big\langle (\bs u,\bs w),(\bs v,\bs p)\big\rangle_{\hH}
    &=\big\langle\left((u_\ell)_{\ell=0}^L, (w_\ell)_{\ell=0}^L\right),
    \left((v_\ell)_{\ell=0}^L, (p_\ell)_{\ell=0}^L\right)
      \big\rangle_{\hH}\\
    &:=
    \sum_{\ell=0}^L c_\ell^{-2} ( u_\ell,v_\ell)_{L^2(\R^d)}
    +\sum_{\ell=0}^L \kappa_\ell^{-1} ( w_\ell,p_\ell)_{L^2(\R^d)}. 
  \end{align*}

\item The operator $\bs A_\star:\vV\to \hH$ (see the right-hand side of \eqref{eq:probMain:a}) given by
\[
\bs A_\star(\bs u,\bs w):=(\bs T_{c^2} \nabla\cdot \bs w,\, \bs T_\kappa \nabla \bs u).
\]
\item The space $\mM:= (\YY_h^\circ)'\times (\XX_h^\circ)'\times \XX_h'\times \YY_h'$, endowed with the product dual norm, where in $\XX_h$ and $\YY_h^\circ$ we use the $\HH^{-1/2}_\Gamma$ norm and in $\YY_h$ and $\XX_h^\circ$ we use the $\HH^{1/2}_\Gamma$ norm.

    Note that while the spaces $(\YY_h^\circ)'$  and $(\XX_h^\circ)'$ have quite complicated structure, we will only
    directly use the subset of functionals of the following form: given $\bs \psi \in \HH^{1/2}_{\Gamma}
    $ and $\bs \phi \in \HH^{-1/2}_{\Gamma}$:
    $$
    \YY_h^\circ\ni\bs\xi\mapsto \langle\bs\xi,\bs \psi\rangle_\Gamma \qquad 
  \text{ or } \qquad 
  \XX_h^\circ\ni\bs\eta\mapsto \langle\bs \phi,\bs \eta\rangle_\Gamma,
$$
   i.e., restrictions of standard $\HH_\Gamma^{-1/2}$ or $\HH_\Gamma^{1/2}$ functionals to
   the smaller subspace of functions satisfying an additional orthogonality condition.
  
\item The boundary operator $\bs B_h:\vV\to \mM$ given by
\[
\bs B_h(\bs u,\bs w):=
(\prodtracejump{\bs u}|_{\YY_h^\circ},
	\prodnormaltracejump{\bs w}|_{\XX_h^\circ},
	\prodtraceext{u}|_{\XX_h},
	\prodnormaltraceext{\bs w}|_{\YY_h}
).
\]
The operator $\bs N_h: \mathbb B \to \mM$ given by $\bs N_h\bs \xi=\bs N_h(\bs\xi^0,\bs\xi^1):=-(\bs\xi^0|_{\mathcal Y_h^\circ},\bs\xi^1|_{\mathcal X_h^\circ},0,0)$.
\end{enumerate}
In the definitions of $\bs B_h$ and $\bs N_h$ we use the same notational convention as in~\cite{sayas_new_analysis} that we explain with the first component of $\bs B_h$: since $\prodtracejump{\bs u}\in \HH^{1/2}_\Gamma=(\HH^{-1/2}_\Gamma)'$ and $\YY_h^\circ\subseteq \HH^{-1/2}_\Gamma$, we can consider $\prodtracejump{\bs u}|_{\YY_h^\circ}\in (\YY_h^\circ)'$ as the functional 
$\YY_h^\circ\ni\bs\xi\mapsto \langle\bs\xi,\prodtracejump{\bs u}\rangle_\Gamma$.
Note that since
  identifying functionals with their Riesz-representant,
  restricting functionals, and padding with zeros  all have operator norm $1$ it holds that
\begin{equation}\label{eq:26}
\|\bs N_h\bs\xi(t)\|_{\mM}\le \| \bs\xi(t)\|_{\mathbb B}
\end{equation}
and that the operator norm of $\bs B_h$ can be bounded independently of the choice of $\XX_h$ and $\YY_h$
as it can be written as a combination of trace operators, restriction maps and the Riesz-isomorphism. 
We can then write problem \eqref{eq:probMain} in the following condensed form. We look for $(\bs u^h,\bs w^h):[0,\infty)\to \vV$ satisfying
\begin{subequations}\label{eq:25}
\begin{alignat}{6}
(\dot{\bs u}^h(t),\dot{\bs w}^h(t))&=\bs A_\star (\bs u^h(t),\bs w^h(t)) &\qquad & \forall t\ge 0, \label{eq:25a}\\
\bs B_h (\bs u^h(t),\bs w^h(t))&=\bs N_h\bs\xi(t) & & \forall t > 0,\\
(\bs u^h(0),\bs w^h(0))&=0.
\end{alignat}
\end{subequations}
Occasionally, we will write $\AA:=\AAstar|_{\ker(\BB_h)}$ for the operator endowed with homogeneous
boundary conditions, i.e., with $\dom(\AA)=\ker(\BB_h) \subseteq \vV$.

\subsection{Analysis}

The next three lemmas will verify the hypotheses of the general framework of \cite[Appendix~A]{brown2018}, which had streamlined the hypotheses of \cite[Sect.~{3}]{sayas_new_analysis}.

\begin{lemma}\label{lemma:1}
The following equality holds:
\[
(\bs A_\star(\bs u,\bs w),(\bs u,\bs w))_\hH=0 \qquad \forall (\bs u,\bs w)\in \ker \bs  B_h.
\]
\end{lemma}

\begin{proof}
A simple computation, using integration by parts on each subdomain, shows that
\begin{alignat}{6}
\nonumber
 (\bs A_\star(\bs u,\bs w),(\bs u,\bs w))_\hH
& = \sum_{\ell=0}^L 
\langle \normaltraceint\ell  w_\ell, \traceint\ell u_\ell \rangle_{\partial\Omega_\ell}-
\langle \normaltraceext\ell  w_\ell, \traceext\ell u_\ell \rangle_{\partial\Omega_\ell}\\
& = \langle \prodnormaltracejump{\bs w}, \prodtraceint{\bs u}\rangle_\Gamma
+\langle \prodnormaltraceext{\bs w},\prodtracejump{\bs u}\rangle_\Gamma
\qquad\forall (\bs u,\bs w)\in \hH.
\label{eq:22.8}
\end{alignat}
Note that $(\bs u,\bs w)\in \ker \bs B_h$ if and only if
\begin{equation}\label{eq:22.10}
\prodtracejump{\bs u}\in \YY_h,
	\qquad
\prodnormaltracejump{\bs w}\in \XX_h,
	\qquad
\prodtraceext{\bs u}\in \XX_h^\circ,
	\qquad
\prodnormaltraceext{\bs w}\in \YY_h^\circ,
\end{equation}
given that
$\dualproduct{\prodnormaltracejump{\bs w}}{\bs \eta}=0 \; \forall \bs \eta \in \XX_h^{\circ}$ implies
$\prodnormaltracejump{\bs w} \in \XX_h$ (in short $(\XX_h^\circ)^\circ=\XX_h$;
this is a simple consequence of the closedness of $\XX_h$ and the Hahn-Banach theorem).
Similarly, $(\YY_h^\circ)^\circ=\YY_h$. 
Since  $\YY_h\subseteq \XX_h^\circ$, the first and third conditions 
in \eqref{eq:22.10} imply $\prodtraceint{\bs u}\in \XX_h^\circ$. 
The result is then a consequence of \eqref{eq:22.8}.
\end{proof}

\begin{lemma}\label{lemma:2}
For all $(\bs f,\bs g)\in \hH$ and $\bs \zeta\in \mM$, there exists a unique $(\bs u,\bs w)\in \vV$ such that
\[
(\bs u,\bs w)=\bs A_\star(\bs u,\bs w)+(\bs f,\bs g), \qquad \bs B_h(\bs u,\bs w)=\bs \zeta,
\]
and there exists a constant $C>0$, depending only on the geometry and the physical parameters (and thus independent of the choice of $\XX_h$ and $\YY_h$) such that
\[
\| (\bs u,\bs w)\|_\hH\le C \big(\|(\bs f,\bs g)\|_\hH +\|\bs \zeta\|_\mM\big).
\]
\end{lemma}

\begin{proof}
We write $\bs f=(f_\ell)_{\ell=0}^L$, $\bs g=(\bs g_\ell)_{\ell=0}^L$, and
$
\bs \zeta=(\zeta_1,\zeta_2,\zeta_3,\zeta_4).
$
Consider the space
\begin{alignat*}{6}
\mathcal W:= & \{ \bs v \in \HH\,:\, \prodtracejump{\bs v}\in \YY_h, \prodtraceext{\bs v}\in \XX_h^\circ\}\\
	\stackrel{\YY_h \subset \XX_h^\circ}{=}& \{ \bs v\in \HH\,:\, \prodtracejump{\bs v}\in \YY_h, \prodtraceint{\bs v}\in \XX_h^\circ\}
	\stackrel{(\ref{eq:22.10})}{=} \{ \bs v\in \HH\,:\, (\bs v,0)\in \ker \bs B_h\},
\end{alignat*}
and  
\begin{alignat*}{6}
a(\bs u,\bs v) :=& \sum_{\ell=0}^L c_\ell^{-2} (u_\ell, v_\ell)_{L^2(\R^d)}
		+\sum_{\ell=0}^L 
		\kappa_\ell (\nabla u_\ell, \nabla v_\ell)_{L^2(\R^d\setminus\partial\Omega_\ell)},\\
b(\bs v) :=& \sum_{\ell=0}^L c_\ell^{-2} (f_\ell,v_\ell)_{L^2(\R^d)}
		- \sum_{\ell=0}^L (g_\ell,\nabla v_\ell)_{L^2(\R^d\setminus\partial\Omega_\ell)}
		 + \langle \zeta_2,\prodtraceint \bs v\rangle_{(\XX_h^\circ)'\times \XX_h^\circ}
			+\langle \zeta_4,\prodtracejump{\bs v} \rangle_{\YY_h'\times \YY_h}.
\end{alignat*}
On the closed subspace ${\mathcal W} \subset \HH$ the bilinear form $a$ is bounded and coercive with constants depending only on the physical coefficients. The linear functional $b$ is bounded and
\[
\| b\|_{\mathcal W'}\le C (\| (\bs f,\bs g)\|_{\hH}+\|\bs \zeta\|_{\mM}).
\]
We now look for $\bs u\in \HH$ satisfying
\begin{subequations}\label{eq:VP}
\begin{alignat}{6}
\label{eq:VP1}
& \prodtracejump{\bs u}|_{\YY_h^\circ}=\zeta_1, \quad \prodtraceext{\bs u}|_{\XX_h}=\zeta_3,\\
& a(\bs u ,\bs v)=b(\bs v) \quad\forall \bs v\in \mathcal W. \label{eq:VP2}
\end{alignat}
\end{subequations}
To that end, we observe that the map
\[
\HH \ni \bs u \longmapsto (\prodtracejump{\bs u},\prodtraceext{\bs u})\in 
\HH^{1/2}_\Gamma\times \HH^{1/2}_\Gamma
\]
admits a bounded right-inverse and that the restriction map
\[
\HH^{1/2}_\Gamma\times \HH^{1/2}_\Gamma
=(\HH^{-1/2}_\Gamma)'\times(\HH^{-1/2}_\Gamma)'
\to (\YY_h^\circ)'\times \XX_h'
\]
admits a norm preserving right-inverse by the Hahn-Banach theorem.
Hence, the linear map that imposes the essential transmission conditions in \eqref{eq:VP1} admits a bounded right-inverse with bound independent of the choice of $\XX_h$ and $\YY_h$. 
  Existence of the solution $\bs u$ of \eqref{eq:VP} is therefore ensured by first lifting the essential transmission conditions
  to get a function $\widetilde{\bs u}$ and then  solving~\eqref{eq:VP2} 
  with homogeneous transmission conditions and a modified right-hand side to obtain a function $\bs u_0 \in \mathcal{W}$.
  (Recall that on this space, the bilinear form $a$ is coercive.)
  Setting $\bs u=\bs u_0+\widetilde{\bs u}$, we note that the transmission conditions~\eqref{eq:VP1}
  still hold, since
  by the definition of polar sets
  $\prodtracejump{\bs u_0} \in \YY_h$ implies $\prodtracejump{\bs u_0}|_{\YY_h^\circ}=0$
  and $\prodtraceext{\bs u_0} \in \XX_{h}^{\circ}$
  gives $\prodtraceext{\bs u_0}|_{\XX_h}=0$.  

  With a solution $\bs u$ to~\eqref{eq:VP} in hand, we
define $\bs w:=\bs T_\kappa \nabla \bs{u}+\bs g$. It is simple to prove that $\bs T_{c^2}\nabla\cdot\bs w+\bs f=\bs u$ (therefore $\bs u\in \HHdiv$), 
where all the operators are applied in a component-wise way and separately on $\Omega_{\ell}$ and
$\R^d \setminus \overline{\Omega_{\ell}}$. Hence, 
$(\bs u,\bs w)=\bs A_\star(\bs u,\bs w)+(\bs f,\bs g). $
It is also easy  to check that 
\begin{align}
  \label{eq:natural_transmission_conditions}
\langle \prodnormaltracejump{\bs w},\prodtraceint{\bs v}\rangle_\Gamma
	+\langle \prodnormaltraceext{\bs w},\prodtracejump {\bs v}\rangle_\Gamma
	=\langle \zeta_2,\prodtraceint{\bs v}\rangle_{(\XX_h^\circ)'\times \XX_h^\circ}
			+\langle\zeta_4,\prodtracejump{\bs v}\rangle_{\YY_h'\times\YY_h}\; \forall \bs v\in \mathcal W.
\end{align}
\eqref{eq:natural_transmission_conditions} implies $\prodnormaltracejump{\bs w}|_{\XX_h^\circ} = \zeta_2$
and $\prodnormaltraceext{\bs w}|_{\YY_h} = \zeta_4$ in view of the surjectivity of the map 
\[
\mathcal W \ni \bs v \longmapsto 
(\prodtraceint{\bs v},\prodtracejump{\bs v})\in{\XX_h^\circ}\times \YY_h. 
\]
Together with \eqref{eq:VP1}, we see that $\bs B_h(\bs u,\bs w) = \bs \zeta$. 
The norm bound follows by the construction.
\end{proof}

\begin{lemma}\label{lemma:3}
The sign flipping operator $\Phi(\bs u,\bs w)=(\bs u,-\bs w)$ is an isometric involution in $\hH$ that preserves $\ker\bs B$ and satisfies $\Phi\bs A_\star=-\bs A_\star\Phi$.
\end{lemma}
\begin{proof} Straightforward.\end{proof}

Following the arguments in \cite[Appendix A]{brown2018},  
Lemmas~\ref{lemma:1}---\ref{lemma:3} prove that the unbounded operator $\bs A_\star|_{\ker\bs B_h}$ is the infinitesimal generator of a group of isometries in $\hH$.

\begin{theorem}\label{thm:26}
If $\bs\xi \in \mathcal C^2([0,\infty);\mathbb B)$ satisfies
$\bs\xi(0)=\dot{\bs\xi}(0)=0$, then the unique solution of \eqref{eq:25} satisfies
\begin{equation}\label{eq:22.12}
\| \big(\bs u^h(t),\bs w^h(t)\big)\|_\vV \le C t \big( 
	\max_{0\le \tau\le t}\|\bs\xi(\tau)\|_{\mathbb B}
	+\max_{0\le \tau\le t}\|\dot{\bs\xi}(\tau)\|_{\mathbb B}
	+\max_{0\le \tau\le t}\|\ddot{\bs\xi}(\tau)\|_{\mathbb B}\big).
  \end{equation}
  Moreover, for $\ell \in \N$,
  if in addition $\bs\xi \in \mathcal C^{\ell+2}([0,\infty);\mathbb B)$ and ${\bs\xi}^{(j)}(0)=0$ for $j\leq\ell+1$, we can also estimate
\begin{equation}\label{eq:28}
\big\|\frac{d^\ell}{dt^\ell}\big(\bs u^h(t), \bs w^h(t)\big)\big\|_{\vV} \le 
C t \sum_{j=\ell}^{\ell+2}\big( 
	\max_{0\le \tau\le t}\|{\bs\xi}^{(j)}(\tau)\|_{\mathbb B}
	\big).
  \end{equation}
\end{theorem}

\begin{proof}
  Let $\bs \xi\in \mathcal C^2([0,\infty);\mM)$ satisfy $\bs \xi(0)=\dot{\bs \xi}(0)=0$. Using~\cite{brown2018} (a slight simplification of \cite{sayas_new_analysis}),
  we can prove that equation \eqref{eq:25} with boundary condition $\bs B_h(\bs u^h,\bs w^h)=\bs \zeta:=\bs N_h \bs \xi$ has a unique classical solution satisfying
\begin{subequations}
\label{eq:thm26-10}
\begin{alignat}{3}
\label{eq:thm26-10-a}
\| (\bs u^h(t),\bs w^h(t)\|_{\hH} \le & C t \max_{0\le \tau\le t} (\| \bs \zeta(\tau)\|_\mM+\|\dot {\bs \zeta}(\tau)\|_\mM), \\
\label{eq:thm26-10-b}
\| (\dot{\bs u}(t),\dot{\bs w}(t)\|_{\hH} \le & 
C t \max_{0\le \tau\le t} (\| \bs{\zeta}(\tau)\|_\mM+\| \dot{\bs \zeta}(\tau)\|_\mM+\|\ddot{\bs \zeta}(\tau)\|_\mM).
\end{alignat}
\end{subequations}
To obtain~\eqref{eq:22.12} from (\ref{eq:thm26-10}), we use~\eqref{eq:26} and~\eqref{eq:25a} 
and the fact that 
$\|(\bs{u},\bs{w})\|_{\vV} \sim 
\|\bs{A}_\star (\bs{u},\bs{w})\|_{\hH} + 
\|(\bs{u},\bs{w})\|_{\hH}$. The  estimate 
\eqref{eq:28}  follows from a simple shifting argument, i.e., by differentiating the equation.
\end{proof}

\section{A system of semidiscrete TDBIE}
\label{sect:integral_equations}

In this section we relate \eqref{eq:probMain} with a system of semidiscrete-in-space time-domain boundary integral equations (TDBIE). Some concepts of TDBIE are needed for the sequel. Full details, in the same language, but with a slightly different notation (we use here Lubich's operational notation) can be found in \cite{sayas_book}.

The retarded potentials for the acoustic wave equation can be introduced through their Laplace transforms and all associated boundary integral operators will be derived using the standard rules of the Calder\'on calculus \cite[Chap.~{1}]{sayas_book}. 
For $s \in \C_+:=\{z \in  \C: \;\Re(z) > 0\}$, we denote the fundamental solution for the differential operator $\laplace - s^2$ 
by
  \begin{align*}
    \Phi(z;s):=\begin{cases}
      \frac{\ii}{4} H_0^{(1)}\left(\ii s \abs{z}\right), & \text{ for $d=2$}, \\
      \frac{e^{-s\abs{z}}}{4 \pi \abs{z}}, & \text{ for $d = 3$},
    \end{cases}
  \end{align*}
where $H_0^{(1)}$ denotes the Hankel function of the first kind and order zero.
We then define the single and double layer potentials for the Laplace resolvent equation
\begin{align*}
    \left(\mathrm S_{\ell}(s) \phi \right)(\mathbf x)
    	&:=\int_{\partial \Omega_{\ell}}{\Phi(\mathbf x-\mathbf y;s) \phi(\mathbf y) 
    	\;d\sigma(\mathbf y)}, \\
    \left(\mathrm D_{\ell}(s) \psi \right)(\mathbf x)
    	& :=\int_{\partial {\Omega_{\ell}}}{\partial_{\nu(\mathbf y)}
    	\Phi(\mathbf x-\mathbf y;s) \psi(\mathbf y) \;d\sigma(\mathbf y)},
\end{align*}
where $d\sigma$ is the arc/area element on $\partial\Omega_\ell$. We will use the symbol for normal derivatives $\bs\partial_\nu:=\bs\gamma_\nu\nabla$ in expressions for interior/exterior traces, jumps, and averages.

\begin{enumerate}
\item On each boundary $\partial\Omega_\ell$, we define the single and double layer retarded potentials
\begin{alignat*}{6}
  \mathrm S_{\ell}(\partial_t)\phi &
  :=\mathscr{L}^{-1}
  \big\{ \mathrm S(\cdot/m_\ell) \mathscr{L} \{\phi\} \big\} \;\qquad\text{and}\;\qquad
  \mathrm D_{\ell}(\partial_t)\psi
  :=\mathscr{L}^{-1}\big\{ \mathrm D(\cdot/m_\ell) \mathscr{L} \{\psi\} \big\},
\end{alignat*}
where $m_\ell:=c_\ell\sqrt{\kappa_\ell}$ and $\mathscr L$ is the distributional Laplace transform.
\item The subdomain potentials are collected in diagonal operators
\begin{alignat*}{6}
{\mathrm{ \bs S}}(\partial_t)\bs\phi
	& = \mathrm{\bs S}(\partial_t)(\phi_\ell)_{\ell=0}^L
	& & := (\mathrm S_\ell(\partial_t)\phi_\ell)_{\ell=0}^L,\\
\mathrm{\bs D}(\partial_t)\bs\psi
	& = \mathrm{\bs D}(\partial_t)(\psi_\ell)_{\ell=0}^L
	& & := (\mathrm D_\ell(\partial_t)\psi_\ell)_{\ell=0}^L,
\end{alignat*}
and we also introduce 
\[
\mathrm{\bs  G}(\partial_t):=
\left[\begin{array}{cc} 
-\mathrm{\bs  D}(\partial_t) & \mathrm{\bs S}(\partial_t)
\end{array}\right],
\]
which satisfies
\begin{equation}\label{eq:jump1}
\left[\begin{array}{c} 
	\prodtracejump{\,\cdot}\\
	\prodnormalderjump{\,\cdot} 					 
\end{array}\right]
\mathrm{\bs G}(\partial_t)=\mathrm I.
\end{equation}
\item The matrix with the time domain boundary integral operators (from all $L+1$ boundaries and using different wave speeds) is defined by
\begin{equation}\label{eq:jump2}
\mathrm{\bs C}(\partial_t):=
\left[\begin{array}{c} 
	\prodtracemean{\,\cdot} \\
	\prodnormaldermean{\,\cdot}
\end{array}\right]
\mathrm{\bs G}(\partial_t)
=\left[\begin{array}{cc}
-\mathrm{\bs K}(\partial_t) & \mathrm{\bs V}(\partial_t) \\
\mathrm{\bs W}(\partial_t) & \mathrm {\bs K^t}(\partial_t)
\end{array}\right].
\end{equation}
(The latter matrix of operators is given for ease of comparison with the literature.) Note that by \eqref{eq:jump1} and \eqref{eq:jump2}, we have
\begin{equation}\label{eq:jump3}
\left[\begin{array}{c} 
	\prodtraceext \\
	\prodnormalderext{} 				
\end{array}\right]
\mathrm{\bs G}(\partial_t)
=\mathrm{\bs C}(\partial_t)-\tfrac12\mathrm I. 
\end{equation}
\item We introduce the diagonal scaling operator
$\mathrm Q_\kappa 
(\bs\psi,\bs\phi)^\top:
=(\bs\psi,(\kappa_\ell\phi_\ell)_{\ell=0}^L)^\top,$
and the partial anti-differentiation operator $\mathrm{\bs J}(\partial_t)(\bs\xi^0,\bs\xi^1):=(\partial_t^{-1}\bs\xi^0,\bs\xi^1)$.
\end{enumerate}

Kirchhoff's formula (see \cite[Proposition~3.5.1]{sayas_book})
shows that if 
\[
\dot{\bs u} =\bs T_{c^2}\nabla\cdot \bs w, \qquad \dot{\bs w}= \bs T_\kappa\nabla \bs u,
\]
(with some very mild distributional regularity conditions and with time-differentiation understood in the sense of vector-valued distributions), then
\begin{equation}\label{eq:3.2}
\bs u=\mathrm{\bs S}(\partial_t)\prodnormaltracejump{\bs T_\kappa^{-1} \dot{\bs w}}
	-\mathrm{\bs D}(\partial_t)\prodtracejump{\bs u}=\mathrm{\bs G}(\partial_t) \mathrm Q_\kappa^{-1}
	(\prodtracejump{\bs u},\prodnormaltracejump{\dot{\bs w}})^\top.
\end{equation}
A precise statement of a theorem relating a system of semidiscrete TDBIE with a distributional version of \eqref{eq:probMain} would use the language of Laplace transformable causal distributions that we will avoid.

To make notation more compact and compatible with the definition of $\mathrm {\bs G}(\partial_t)$, we will collect the $\mathbb B$-valued densities $(\bs\psi,\bs\phi)$ in the column vector $\bs\lambda:=(\bs\psi,\bs\phi)^\top$. The data will appear in the somewhat peculiar form $\mathrm {\bs J}(\partial_t)\dot{\bs\xi}=(\bs\xi^0,\dot{\bs\xi}^1)$. The operator $\mathrm {\bs J}(\partial_t)$ will be part of what we will discretize in time, while we will work with $\dot{\bs\xi}$ as data, which means that we will use
$\traceint0 \dot u^{\mathrm{inc}}$ and $\normaltraceint0 \nabla u^{\mathrm{inc}}$ as data for the numerical method expressed with TDBIE (see Proposition \ref{prop:2.2}), i.e., we either differentiate the incident wave in space or in time. When the incident wave is a plane wave $u^{\mathrm{inc}}(t)(\mathbf x)=g(\mathbf x\cdot\mathbf d-t)$ (for $\mathbf d\in \R^d$ with $|\mathbf d|=1$), we will only need to evaluate $\dot g(\mathbf x\cdot\mathbf d-t)$ for points $\mathbf x\in \partial\Omega_0$.

As in the previous section, we treat the continuous-  and the semidiscrete problem
  simultaneously. The continuous solutions can always be recovered by taking
  $\XX_h=\XX$ and $\YY_h=\YY$. We denote the continuous field by removing the
  index $h$, i.e., writing $\bs u$, $\bs w$, $\bs \lambda$, etc.

\begin{theorem}\label{thm:32}
If $(\bs u^h,\bs w^h)$ is a $\vV$-valued causal distribution satisfying
\begin{equation}\label{eq:33.5}
(\dot{\bs u}^h,\dot{\bs w}^h)=\bs A_\star(\bs u^h,\bs w^h),
\qquad
\bs B_h(\bs u^h,\bs w^h)=\bs N_h\bs\xi,
\end{equation}
then $\bs \lambda^h:=(
	\prodtracejump{\bs u^h}+\bs\xi^0,
	\prodnormaltracejump{\dot{\bs w}^h}+\dot{\bs\xi}^1)^\top$
is the unique $\mathbb B$-valued causal distribution satisfying
\begin{subequations}\label{eq:36}
\begin{alignat}{6}
\label{eq:36a}
& \bs\lambda^h \in \YY_h\times \XX_h,\\
\label{eq:36b}
& \langle \mathrm Q_\kappa \mathrm{\bs C}(\partial_t)\mathrm Q_\kappa^{-1}\bs\lambda^h,
	\bs\varpi\rangle = 
	\langle \mathrm Q_\kappa (\mathrm {\bs C}(\partial_t)-\tfrac12\mathrm I)
	\mathrm Q_\kappa^{-1}\mathrm{\bs J}(\partial_t)\dot{\bs\xi},\bs\varpi\rangle 
	\qquad\forall \bs\varpi\in \XX_h\times \YY_h,
\end{alignat}
where the angled bracket is the $\mathbb B \times \mathbb B'$ duality product. 
\end{subequations}
Reciprocally, if  $\bs\lambda^h$ is the solution of \eqref{eq:36} and we let
\begin{equation}\label{eq:377}
\bs u^h:=\mathrm {\bs G}(\partial_t)\mathrm Q_\kappa^{-1}
(\bs\lambda^h-\mathrm{\bs J}(\partial_t)\dot{\bs\xi}),
\qquad
\bs w^h=\bs T_\kappa \nabla\partial_t^{-1} \bs u^h,
\end{equation}
then $(\bs u^h,\bs v^h)$ satisfies \eqref{eq:33.5}.
\end{theorem}

\begin{proof} 
First of all the boundary conditions in \eqref{eq:33.5} are equivalent to
\begin{subequations}\label{eq:35}
\begin{alignat}{6}
\label{eq:35b}
& \prodtracejump{\bs u^h}+\bs\xi^0\in \YY_h, & \quad 
	& \prodnormaltracejump{\dot{\bs w}^h}+\dot{\bs\xi}^1 \in \XX_h, \\
\label{eq:35c}
& \prodtraceext \bs u^h\in \XX_h^\circ, & 
	& \prodnormaltraceext{\dot{\bs w}^h}\in \YY_h^\circ.
\end{alignat}
\end{subequations}
(Compare with \eqref{eq:probMain} and note that we have differentiated the conditions related to $\bs w^h$ for later convenience.)
Given a solution to \eqref{eq:33.5}, we can use \eqref{eq:3.2} to write the pair $(\bs u^h,\bs w^h)$  in the form \eqref{eq:377}. The condition \eqref{eq:35b} is equivalent to $\bs\lambda^h\in \YY_h\times \XX_h$, while \eqref{eq:35c} is equivalent (using \eqref{eq:jump3}) to 
\begin{equation}\label{eq:37}
\mathrm Q_\kappa (\mathrm {\bs C}(\partial_t)-\tfrac12\mathrm I)
	\mathrm Q_\kappa^{-1}(\bs\lambda^h-\mathrm{\bs J}(\partial_t)\dot{\bs\xi})
	\in \XX_h^\circ\times \YY_h^\circ.
\end{equation}
However, since $\YY_h\times\XX_h\subseteq \XX_h^\circ\times \YY_h^\circ$, \eqref{eq:36a} and \eqref{eq:37} are equivalent to \eqref{eq:36a} and
\begin{equation}\label{eq:38}
\mathrm Q_\kappa \mathrm{\bs C}(\partial_t)\mathrm Q_\kappa^{-1}\bs\lambda^h
-\mathrm Q_\kappa (\mathrm{\bs C}(\partial_t)-\tfrac12\mathrm I)
	\mathrm Q_\kappa^{-1}\mathrm{\bs J}(\partial_t)\dot{\bs\xi}
	\in \XX_h^\circ\times \YY_h^\circ.
\end{equation}
But \eqref{eq:38} is just a short hand version of \eqref{eq:36b}. The proof of the reciprocal statement is very similar.
\end{proof}

The estimates of Theorem \ref{thm:26} hold for the solution of \eqref{eq:36} if we prove (which can be easily done using the techniques of 
\cite[Sect.~{3}]{sayas_new_analysis}) that the strong solution of \eqref{eq:probMain}, extended by zero to negative times, is the distributional solution of \eqref{eq:35}. 

\begin{theorem}
\label{thm:approximation_estimate_in_space}
Assume that
$\bs \xi \in \mathcal{C}^{6}([0,\infty);\mathbb{B})$ with
$\bs\xi^{(\ell)}(0)=0$ for $\ell=0,\dots,5$.
Let $\bs\lambda=(\bs\psi,\bs\phi)^\top$ be the solution of
\begin{subequations}\label{eq:39}
\begin{alignat}{6}
& \bs\lambda \in \YY\times \XX \\
& \langle \mathrm Q_\kappa \mathrm{\bs C}(\partial_t)\mathrm Q_\kappa^{-1}\bs\lambda,
	\bs\varpi\rangle = 
	\langle \mathrm Q_\kappa (\mathrm{\bs C}(\partial_t)-\tfrac12\mathrm I)
	\mathrm Q_\kappa^{-1}\mathrm{\bs J}(\partial_t)\dot{\bs\xi},\bs\varpi\rangle 
	\qquad\forall \bs\varpi\in \XX\times \YY,
\end{alignat}
\end{subequations}
 and let $\bs\lambda^h=(\bs\psi^h,\bs\phi^h)^\top$ be
the solution of \eqref{eq:36}. Consider the associated potentials
\[
\bs u=\mathrm{\bs G}(\partial_t) \mathrm Q_\kappa^{-1}
	(\bs\lambda-\mathrm{\bs J}(\partial_t)\dot{\bs\xi}), \qquad
\bs u^h=\mathrm{\bs G}(\partial_t) \mathrm Q_\kappa^{-1}
	(\bs\lambda^h-\mathrm J(\partial_t)\dot{\bs\xi}).
\]
Then: 
\begin{alignat*}{6}
\| \bs u(t)-\bs u^h(t)\|_{\vV} &\le C t 
\sum_{j=0}^2 \max_{0\le \tau\le t}  
	\|(\bs\psi^{(j+1)},\bs\phi^{(j)})(\tau)
		-\Pi(\bs\psi^{(j+1)},\bs\phi^{(j)})(\tau)\|_{\mathbb B},\\
\| \bs\psi(t)-\bs\psi^h(t)\|_{\HH^{1/2}} &\le C t  
\sum_{j=0}^2 \max_{0\le \tau\le t}
	\|(\bs\psi^{(j+1)},\bs\phi^{(j)})(\tau)
		-\Pi(\bs\psi^{(j+1)},\bs\phi^{(j)})(\tau)\|_{\mathbb B},\\
\| \bs\phi(t)-\bs\phi^h(t)\|_{\HH^{-1/2}} &\le C t  
\sum_{j=1}^3 \max_{0\le \tau\le t}
	\|(\bs\psi^{(j+1)},\bs\phi^{(j)})(\tau)
		-\Pi(\bs\psi^{(j+1)},\bs\phi^{(j)})(\tau)\|_{\mathbb B}.
\end{alignat*}
where $\Pi:\mathbb B\to \YY_h\times \XX_h$ is the best approximation operator onto $\YY_h\times \XX_h$.
\end{theorem}

\begin{proof}
    We consider the difference $\bs e:=(\bs e_1,\bs e_2):= (\bs u,\bs w) -  ( \bs u^h,\bs w^h)$.
  This function solves the differential equation $\dot{\bs e}=\AA_{\star} \bs e$, and the transmission conditions satisfied by $\bs u^h$ give
  the following transmission conditions for $\bs e$:
  \begin{alignat*}{3}
    \prodtracejump{\bs{e}_1}(t) &- \prodtracejump{u}(t) - \bs \xi^0(t) \in \YY_h,&
    \qquad  \prodnormaltracejump{\bs e_{2}}(t) &- \prodnormaltracejump{\bs v}(t) - \bs \xi^1(t) \in \XX_h, \\
    \prodtraceext{\bs{e}_1}(t) &\in \XX_h^{\circ},& \qquad \prodnormaltraceext{\bs e_2}(t) &\in \YY_h^{\circ}
  \end{alignat*}
  for all $t\geq 0$.
  Secondly, we notice that these conditions are invariant under subtracting discrete functions, i.e., for 
  $\bs\chi_h(t) \in \YY_h$, $\bs\mu_h(t) \in \XX_h$, they are equivalent to the  following conditions:  
    \begin{alignat*}{3}
    \prodtracejump{\bs{e}_1}(t) &- \bs \psi(t) + \bs \chi_h(t) \in \YY_h,&
    \qquad  \prodnormaltracejump{\bs e_2}(t) &- \partial_t^{-1}\bs \phi(t) + \bs \mu_h(t) \in \XX_h, \\
    \prodtraceext{\bs{e}_1}(t) &\in \XX_h^{\circ},& \qquad \prodnormaltraceext{\bs e_2}(t) &\in \YY_h^{\circ},
  \end{alignat*}
  where we  also inserted the definitions of $\bs \psi$ and $\bs \phi$ to shorten notation.
  This is structurally the same as~\eqref{eq:probMain}.
  Using the best approximation operator $\Pi$, setting $(\bs \chi_h(t),\bs \mu_h(t)):=\Pi \bs \lambda(t)$
  and applying the stability estimate of Theorem~\ref{thm:26} gives the
  estimate for $\bs{u}-\bs{u}^h$.
  The bound for $\|\bs\psi(t)-\bs\psi^h(t)\|_{\HH^{1/2}}=\| \prodtracejump{\bigl(\bs u(t) -\bs u^h(t)\bigr)}\|_{\HH^{1/2}}$ follows from the trace theorem. Finally, the bound for
\[
\|\bs\phi(t)-\bs\phi^h(t)\|_{\HH^{-1/2}}=\|\mathrm Q_\kappa
\prodnormaltracejump{\bigl(\dot{\bs w}(t)-\dot{\bs w}^h(t)\bigr)}\|_{\HH^{-1/2}}
\le C \|  \dot{\bs w}(t)-\dot{\bs w}^h(t)\|_{\HHdiv}
\]
requires \eqref{eq:28}.
The requirements on $\bs \xi$ are such that the exact traces $\bs \psi$ and $\bs\phi$
  have the required regularity by Theorem~\ref{thm:26}.
\end{proof}

  \begin{remark}
    We point out out that the formulation~\eqref{eq:36} is just a time-domain version
    of the formulation introduced by von~Petersdorff in~\cite{von_petersdorff}. The only
    minor difference compared to a straight-forward adaptation
    is that we use $\dot{\bs \xi}$ instead of $\bs \xi$ as the data; see also~\eqref{eq:the_bad_method}.
  \end{remark}

\section{Time discretization - Runge Kutta convolution quadrature}
\label{sect:rkcq}

An implicit Runge-Kutta method with $m$ stages is given by a matrix $\rkA\in \R^{m\times m}$ and two vectors $\rkb,\rkc\in \R^m$. Its stability function is the rational function $r(z):=1+z\rkb^\top (I-z\rkA)^{-1}\ones$, where $\ones:=(1,\ldots,1)^\top$. In everything that follows, we will always assume that $\rkA$ is invertible, which is a necessary condition to be in the framework of RK-based convolution quadrature methods. Therefore, the limit $r(\infty)=\lim_{z\to \infty} r(z)=1-\rkb^\top \rkA^{-1}\ones$ exists. We say that the RK method is:
\begin{enumerate}
\item[(a)] A-stable when $|r(\imath t)|\le 1$ for all $t\in \R$,
\item[(b)] strictly A-stable when $|r(\imath t)|<1$ for $t\in \R\setminus\{0\}$ and $r(\infty)<1$,
\item[(c)] stiffly accurate, when $\rkb^\top\rkA^{-1}=(0,\ldots,0,1)$ and therefore $c_m=1$ and $r(\infty)=0$. 
\end{enumerate}
We will assume that the stage order of the RK method is $q$, while its classical order is $p\ge q$. The methods of the Radau IIa family of RK methods have invertible matrix $\rkA$, are strictly A-stable and stiffly accurate. 
These methods are standard for applications in convolution quadrature, despite their
  damping properties, which are not ideal for wave equations.
  This is in part due to the fact that the standard theory (see, e.g.,~\cite{BanLM})
  makes some assumptions not satisfied by the Gauss methods.
  We also would like to point out that in higher order methods
  the dissipation and dispersion is much better controlled than for the low order
  cousins~\cite[Section~{4.3}]{banjai_schanz}, which is another good reason for utilizing
  Runge-Kutta methods for wave propagation applications.

\subsection{The fully discrete method}

In Section \ref{sect:integral_equations} we have introduced operators $\mathrm H(\partial_t)$ (with $\mathrm H\in \{\mathrm{\bs C},{\mathrm{ \bs G}},{\mathrm{ \bs J}}\}$) such that there exists an analytic function $\mathrm H:\C_+\to \mathcal B(\mathcal Z_1,\mathcal Z_2)$ (here $\C_+:=\{ z\in \C\,:\, \mathrm{Re}\,z>0\}$ and $\mathcal B(\mathcal Z_1, \mathcal Z_2)$ is the space of bounded 
linear operators between two Hilbert spaces) such that 
\[
\mathscr L\{ \mathrm H(\partial_t)\xi\}=\mathrm H\, \mathscr L\{\xi\}.
\]
We can then expand
\[
\mathrm H\left( k^{-1} \left(\rkA-\frac{z}{1-z}\ones\rkb^\top\right)^{-1}\right)
	= \sum_{j=0}^\infty z^j \mathrm H_j,
\]
where evaluating $\mathrm H$ with a matrix as its argument can be done with Riesz-Dunford calculus (see \cite[Chap.~{11}]{golub_van_loan}
or~\cite[Chap.~{VIII.7}]{yosida_fana}), and the series is a Maclaurin expansion of an analytic function with coefficients $\mathrm H_j\in \mathcal B(\mathcal Z_1^m, \mathcal Z_2^m)$. Note that for $\abs{z}<1$ the spectrum of $\rkA-(1-z)^{-1}z \ones\rkb^\top$
is contained in $\C_+$ for every A-stable RK method with invertible $\rkA$ by \cite[Lemma~{3}]{BanLM}.
Given a sequence of vectors $\Xi:=(\Xi_{n})_{n=0}^{\infty} \subseteq \mathcal{Z}_1^m$, the discrete convolution defined by the above sequence of operators 
\[
Y_n:= \sum_{j=0}^n \mathrm H_j \Xi_{n-j}, \qquad n\ge 0,
\]
transforms sequences in $\mathcal Z_1^m$ into sequences in $\mathcal Z_2^m$ and will be denoted $Y= H(\partial_k)\Xi$. Additionally, we can produce a sequence in $\mathcal Z_2$ in the postprocessed form
\begin{equation}\label{eq:4.0}
y_0:=0, \qquad 
y_n:=r(\infty)y_{n-1}+\rkb^\top \rkA^{-1} Y_{n-1}, \quad n\ge 1,
\end{equation}
which in the case of stiffly accurate RK methods just delivers the sequence with the $m$-th components of $\{Y_n\}$, namely, $y_n=(0,\ldots,0,1) Y_{n-1}$. The postprocessing step described in \eqref{eq:4.0} will be denoted $\{ y_n\}=\postproc\,\{Y_n\}$. The computation of $y=\partial_k^{-1} \Xi$ (the RK-CQ method when $\mathrm H(s)=s^{-1} \mathrm I$) can be easily seen to be equivalent to the recurrence 
\begin{equation}\label{eq:4.00}
y_0:=0, 
	\qquad
Y_n=\ones y_n+k\rkA \Xi_n,
	\quad
y_{n+1}=r(\infty) y_n+\rkb^\top \rkA^{-1} Y_n, \quad n\ge 0,
\end{equation}
which computes the `postprocessed' sequence simultaneously. When $\Xi_n=\xi(t_n+\rkc k):=(\xi(t_n+c_1\,k),\ldots,\xi(t_n+c_m\,k))^\top$, this is just 
the application of the RK method to 
\[
\dot y(t)=\xi(t), \quad t\ge 0, \qquad y(0)=0,
\]
which we can write as the operator equation $y=\partial_t^{-1}\xi$. In \eqref{eq:4.0} and \eqref{eq:4.00} we have used the product of scalar matrices by elements of $\mathcal Z_2^m$, which has to be understood as
taking linear combinations of elements of $\mathcal Z_2$ using  the coefficients of the matrix.
We will also use the following instance of Kronecker products: given $R\in \mathcal B(\mathcal Z_1,\mathcal Z_2)$ we denote
\begin{align}
\label{eq:def_breve}
\breve R
	:=I_{m\times m}\otimes R
	:=\begin{bmatrix} R \\ & \ddots \\ & & R\end{bmatrix}
		\in \mathcal B(\mathcal Z_1^m;\mathcal Z_2^m). 
\end{align}

The fully discrete numerical method that we propose and analyze is an RK-CQ discretization of \eqref{eq:36}, followed by the RK-CQ discretization of the potentials \eqref{eq:377}. We start by sampling the data
\begin{subequations}\label{eq:4.1}
\begin{equation}
\dot{\bs\Xi}^k:=\{ \dot{\bs\xi}(t_n+\rkc k)\}_{n=0}^\infty, \qquad t_n:=n\,k.
\end{equation}
Next we compute a sequence
\begin{equation}\label{eq:4.1a}
\bs\Lambda^{h,k}=\{ \bs\Lambda^{h,k}_n\}_{n=0}^\infty, 
	\qquad \bs\Lambda^{h,k}_n\in (\YY_h\times \XX_h)^m,
\end{equation}
satisfying 
\begin{equation}
\label{eq:4.1b}
\langle \breve{\mathrm Q}_\kappa \mathrm{\bs C}(\partial_k)\breve{\mathrm Q}_\kappa^{-1}
	\bs\Lambda^{h,k},
	\bs\varpi\rangle = 
	\langle \breve{\mathrm Q}_\kappa (\mathrm{\bs C}(\partial_k)-\tfrac12\mathrm I)
	\breve{\mathrm Q}^{-1}_\kappa\mathrm{\bs J}(\partial_k)\dot{\bs\Xi}^k,\bs\varpi\rangle 
	\qquad\forall \bs\varpi\in (\XX_h\times \YY_h)^m.
\end{equation}
The expression \eqref{eq:4.1b} represents a discrete convolutional system that yields the different time-values of the sequence $\bs\Lambda^{h,k}$ as a recursion. Each time step requires the solution of a square linear system of equations with $m(\mathrm{dim}\,\YY_h+\mathrm{dim}\,\XX_h)$ unknowns. We finally compute 
\begin{equation}\label{eq:4.1c}
\bs U^{h,k}=\mathrm{\bs G}(\partial_k)\breve{\mathrm Q}_\kappa^{-1}
(\bs\Lambda^{h,k}-\mathrm {\bs J}(\partial_k)\dot{\bs\Xi}^k),
\quad
\bs{W}^{h,k}=\breve{\bs T}_\kappa \breve\nabla\partial_k^{-1} \bs U^{h,k}.
\end{equation}

Corresponding to these stage vectors, we can then define the approximations at the endpoints via
\begin{align}
  \label{eq:def_postprocessed_fd}
  \bs \lambda^{h,k}&:=(\bs \psi^{h,k},\bs \phi^{h,k}):=\postproc{\Lambda^{h,k}}, \quad
  \bs u^{h,k}=\postproc{\bs U^{h,k}}, \quad \text{and } \quad \bs w^{h,k}:=\postproc{\bs W^{h,k}}.
\end{align}
\end{subequations}
(Here we committed the slight abuse of notation and identified  $(\XX_h\times \YY_h)^m$ with $(\XX_h^m\times \YY_h^m)$.)

Our next effort is to relate \eqref{eq:4.1} with a discretization of a certain IBVP related to the pair $\bs x^h:=(\bs u^h,\bs w^h)$, in the same way that Theorem \ref{thm:32} related the semidiscrete system of TDBIE \eqref{eq:36}, postprocessed with the retarded potential expressions \eqref{eq:377} to a weak-in-time version of \eqref{eq:25}. 
In strong form, $\bs x^h:=(\bs u^h,\bs w^h):[0,\infty)\to \vV$ satisfies
\begin{equation}\label{eq:4.4}
\dot{\bs x}^h(t)=\bs A_\star \bs x^h(t), \qquad 
\bs B_h\dot{\bs x}^h(t)=\bs N_h\dot{\bs\xi}(t), \qquad
\bs{x}^h(0)=0,
\end{equation}
which is equivalent to \eqref{eq:25}. The boundary condition can equivalently be written $\bs B {\bs x}^h=\partial_t^{-1} \bs{N}_h\dot{\bs\xi}=\bs N_h\bs\xi$, but, as we have already mentioned, we will use $\dot{\bs\xi}$ as data. An RK-CQ approximation of \eqref{eq:4.4} simply substitutes time derivatives by $\partial_k$: 
\begin{equation}\label{eq:4.5}
\partial_k \bs{X}^{h,k}=\breve{\bs A_\star} \bs{X}^{h,k},\qquad
\breve{\bs B}_h \partial_k \bs{X}^{h,k}=\breve{\bs N}_h \dot{\bs\Xi}^k, \qquad
\bs x^{h,k}=\postproc \bs{X}^{h,k}.
\end{equation}
This can also be written in RK form
\begin{subequations}
  \label{eq:4.5_rk_form}
\begin{alignat}{6}
& \bs{x}^{h,k}_0:=0,\\
& \bs{X}^{h,k}_n=\ones \bs{x}^{h,k}_n+ k\rkA\breve{\bs A_\star} \bs{X}^{h,k}_n, \qquad
 \breve{\bs B}_h \bs{X}^{h,k}_n=\bs\Theta^k_n, \\
& \bs{x}^{h,k}_{n+1}=\bs{x}^{h,k}_n+k\rkb^\top\breve{\bs A_\star}\bs{X}^{h,k}_n
	=r(\infty) \bs{x}^{h,k}_n+ \rkb^\top \rkA^{-1} \bs{X}^{h,k}_n,
\end{alignat}
\end{subequations}
where $\bs\Theta^k=\{\bs\Theta^k_n\}:=\breve{\bs N_h} \partial_k^{-1}\dot{\bs\Xi}^k$.

\begin{proposition}
If $\bs\Lambda^{h,k}, \bs U^{h,k}, \bs W^{h,k}$ solve \eqref{eq:4.1}, then $\bs X^{h,k}:=(\bs U^{h,k},\bs W^{h,k})$ solves \eqref{eq:4.5}. Reciprocally, if $\bs X^{h,k}=(\bs U^{h,k},\bs W^{h,k})$ solves \eqref{eq:4.5}, then the shifted traces 
$\bs\Lambda^{h,k}
	:= (\llbracket \breve{\prodtrace} \bs U^{h,k}\rrbracket,
		\llbracket \breve{\boldsymbol\gamma}_\nu \bs W^{h,k}\rrbracket)^\top 
		+\partial_k^{-1} \dot{\bs\Xi}^k$
satisfy \eqref{eq:4.1a}-\eqref{eq:4.1b} and \eqref{eq:4.1c} holds.
\end{proposition}

\begin{proof}
Taking the Laplace transforms of the corresponding continuous problems (Theorem \ref{thm:32}) and $Z$-transforms of the discrete problems \eqref{eq:4.1} and \eqref{eq:4.5}, we can easily prove the statement. See~\cite{schroedinger} for a detailed analogous computation.
\end{proof}

In the numerical experiments Section \ref{sect:numerics} , we will compare \eqref{eq:4.1} with a method that has $\mathrm{\bs{J}}(\partial_t)\bs\xi$ as data, i.e.,
where $\mathrm{\bs{J}}(\partial_t)$ is not discretized in the time variable. In this method we first sample
\begin{subequations}
  \label{eq:the_bad_method}
\begin{equation}
\bs\Sigma^k:=\{ (\bs\xi^0(t_n+\rkc k),{\dot{\bs\xi}^1}(t_n+\rkc k))\}_{n=0}^\infty,
\end{equation}
next look for
\begin{equation}
\bs\Lambda^{h,k}=\{ \bs\Lambda^{h,k}_n\}_{n=0}^\infty, 
	\qquad \bs\Lambda^{h,k}_n\in (\YY_h\times \XX_h)^m
\end{equation}
satisfying
\begin{equation}
\langle \breve{\mathrm Q}_\kappa \mathrm C(\partial_k)\breve{\mathrm Q}_\kappa^{-1}
	\bs\Lambda^{h,k},
	\bs\varpi\rangle = 
	\langle \breve{\mathrm Q}_\kappa (\mathrm C(\partial_k)-\tfrac12\mathrm I)
	\breve{\mathrm Q}^{-1}_\kappa\bs\Sigma^k,\bs\varpi\rangle 
	\qquad\forall \bs\varpi\in (\XX_h\times \YY_h)^m
\end{equation}
and finally postprocess by setting
\begin{equation}
\bs U^{h,k}=\{\bs U^{h,k}_n\}:=\mathrm G(\partial_k)\breve{\mathrm Q}_\kappa^{-1}
(\bs\Lambda^{h,k}-\bs\Sigma^k),
\quad
\bs W^{h,k}=\{\bs W^{h,k}_n\}:=\breve{\bs T}_\kappa \breve\nabla\partial_k^{-1} \bs U^{h,k}.
\end{equation}
\end{subequations}

\subsection{Some regularity theorems}
In this section we verify that the semidiscrete solution to~\eqref{eq:25} satisfies the assumptions
of the abstract RK-theory in~\cite{mallo_palencia_optimal_orders_rk,semigroups}.
\begin{lemma}\label{lemma:4.1}
The map $\lifting: \mathbb B\to \mathbb V$, given by $\lifting \bs\zeta:=(\bs u,\bs w)$, where
\begin{align}
  (\bs u,\bs w)&=\bs A_\star(\bs u,\bs w), \qquad \bs B_h(\bs u,\bs w)=\bs N_h\bs\zeta,
  \label{eq:def_lifting}
\end{align}
is well defined and bounded independently of the choice of the spaces $\XX_h$ and $\YY_h$.
\end{lemma}

\begin{proof}
It is a direct consequence of Lemma \ref{lemma:2} and \eqref{eq:26}.
\end{proof}

We consider the spaces
\begin{equation}
\label{eq:hHmu}
\hH_\mu:=[\hH,\ker \bs B_h]_{\mu,\infty}, \qquad \mu\in (0,1), \qquad \text{ with } \|\cdot\|_{\ker \bs B_h}:=\|\cdot\|_{\hH} + \|\AA_{\star} \cdot\|_{\hH},
\end{equation}
obtained by the real interpolation method for Banach spaces (see \cite{tartar07,triebel95} or \cite[Appendix~B]{mclean}).
We recall that for two Banach spaces $\mathcal{X}_1 \subseteq \mathcal{X}_0$ with continuous embedding, the norm is given by:
\begin{align}
  \label{eq:definition_interpolation_norm}
  \norm{u}_{[\mathcal{X}_0,\mathcal{X}_1]_{\mu,\infty}}
  &:=\operatorname{ess\,sup}_{t>0}{\Big(t^{-\mu} \inf_{v \in \mathcal{X}_1} \left[\norm{u-v}_{\mathcal{X}_0} + t \norm{v}_{\mathcal{X}_1}\right]\Big)}.
\end{align}

\begin{lemma}
  \label{lemma:more_regularity}
  For $\mu \leq 1/2$,
  the map $\lifting$ of~\eqref{eq:def_lifting} is bounded from $\HH^{1/2}_\Gamma\times \HH^{-1/2+\mu}_\Gamma$ to $\hH_\mu$.
\end{lemma}

\begin{proof}
  For $\mu=0$, the statement follows from Lemma~\ref{lemma:4.1}. We focus on $\mu=1/2$.
  Given $\bs\zeta_0\in \HH^{1/2}_{\Gamma}$, we take $\bs u_0\in \HH$ satisfying
  \[
    -\Delta \bs u_0+\bs u_0=0, \qquad \prodtraceint \bs u_0=\bs\zeta_0,
    \qquad \prodtraceext u_0=0.
  \]
  This is a collection of $L+1$ decoupled interior-exterior Dirichlet problems in $\R^d\setminus\partial\Omega_\ell$ with vanishing exterior components in all cases.
  We claim that each component of $\bs u_0=(u_{0,\ell})_{\ell=0}^{L}$ satisfies
  $$u_{0,\ell} \in  \big[L^2(\R^d),H^1_0(\R^d\setminus\partial\Omega_\ell) \big]_{\frac{1}{2},\infty}.$$
This follows from the observation $u_{0,\ell} \in H^1(\R^d\setminus\partial\Omega_\ell) $, the embeddings 
$H^1(\omega) \subset B^{1/2}_{2,1}(\omega) \subset [L^2(\omega),H^1_0(\omega)]_{1/2,\infty}$ 
(for $\omega \in \{\Omega_\ell, \R^d \setminus \overline{\Omega_\ell}\}$) asserted in \cite[Thm.~{A.1}]{semigroups}. 
Given $\bs\zeta_1=(\zeta_{1,\ell})_{\ell =0}^{L}\in  \prod_{\ell}{L^2(\partial\Omega_\ell)}$,
we use \cite[Thm.~{A.4}]{semigroups} to construct on each subdomain a function
  $w_{0,\ell} \in H(\operatorname{div},\Omega_\ell)$ with
  $$ 
  \gamma^{\mathrm{int}}_{\nu,\ell} w_{0,\ell} = \zeta_{1,\ell}, \quad \text{and} \quad
  \norm{w_{0,\ell}}_{\big[L^2(\Omega_{\ell}),H_0(\operatorname{div},\Omega_{\ell})\big]_{\frac{1}{2},\infty}}\lesssim \norm{\zeta_{1,\ell}}_{L^2(\partial \Omega_\ell)}.
  $$
  Here, $H_0(\operatorname{div}, \Omega_{\ell})$ denotes the functions in $H(\operatorname{div},\Omega_{\ell})$ with vanishing interior normal trace.
  Similarly, we  write $H_0(\operatorname{div}, \R^d \setminus \partial \Omega_{\ell})$ for functions with vanishing interior and exterior normal traces.
  
  Extending these functions $w_{0,\ell}$ by zero outside $\Omega_{\ell}$ and collecting them in $\bs w_0:=(w_{0,\ell})_{\ell=0}^{L}$, we get  $\bs w_0\in \HHdiv$ satisfying
  \[
    \prodnormaltraceint \bs w_0=\bs\zeta_1,
    \qquad \prodnormaltraceext  \bs w_0=0.
  \]

 Since true zero boundary conditions are stronger than those imposed by $\ker(\bs B_h)$ it is easy to see  that
  $$
  \prod_{\ell=0}^{L}{H_0^1(\R^d \setminus \Omega)} \times \prod_{\ell=0}^{L}{H_{0}(\operatorname{div},\R^d \setminus \Omega)}
  \subseteq \ker(\bs B_h)=\dom(\AA).
  $$
  Since interpolation of product spaces corresponds to the product of interpolation spaces
  (see~\cite[Lemma A.5]{semigroups}),
  we get that 
  \begin{align*}
    (\bs u_0,\bs w_0)
    &\in
      \prod_{\ell=0}^{L} [L^2(\R^d),H^1_0(\R^d\setminus\partial\Omega_\ell)]_{\frac{1}{2},\infty}
      \times \prod_{\ell=0}^{L} \big[L^2(\R^d \setminus \partial \Omega_{\ell}),H_0(\operatorname{div},\R^d \setminus \partial \Omega_{\ell})\big]_{\frac{1}{2},\infty} \\
    &= \Big[ \mathcal{L}^2, \prod_{\ell=0}^{L}H_0^1(\R^d \setminus \Omega)\Big]_{\frac{1}{2},\infty}
      \times \Big[ \bs L^2, \prod_{\ell=0}^{L}{H_{0}(\operatorname{div},\R^d \setminus \Omega)} \Big]_{\frac{1}{2},\infty}
    \subseteq  \hH_{1/2}.
  \end{align*} 

Since all these inclusions come with norm estimates, we thus have a bounded operator
  \[
    \mathbb B \ni \bs\zeta\longmapsto (\bs u_0, \bs w_0)\in \vV
  \]
  (this is not the lifting $\lifting$) such that
  \[
    \HH^{1/2}_\Gamma\times \HH^{0}_\Gamma \ni \bs\zeta
    \longmapsto (\bs u_0,\bs w_0) \in \hH_{1/2}
  \]
  is also bounded.
 Therefore,
  for $(\bs u,\bs w):=\lifting\bs \zeta$, we have
  \[
    (\bs u-\bs u_0,\bs w- \bs w_0)\in \ker \bs B_h\subseteq \hH_{1/2}.  
  \]
  Since for elements of $\ker \bs B_h$ the $\hH_{1/2}$ norm can be estimated by the $\vV$ norm, (cf.\ (\ref{eq:hHmu}))
  in which $\lifting$ is bounded, this concludes the  proof for $\mu \in \{0,1/2\}$.
    An interpolation argument and the reiteration theorem~\cite[Theorem 26.3]{tartar07},
      extends this bound to $\mu \in [0,1/2]$.
\end{proof}

To shorten some expressions, we introduce notation for the norm on the right-hand side of~\eqref{eq:regularity_estimate}.
For $m \in \N$ and $\psi \in \mathcal C^{m}([0,T],H^{1/2}(\partial \Omega_0))$, $\phi \in \mathcal C^{m-1}([0,T],H^{-1/2+\mu}(\partial \Omega_0))$,
$\mu \in [0,1/2]$, we write
\begin{align}
  \label{eq:def_regularity_norm}
  \triplenorm{ (\psi,\phi)}_{m,T,\mu}
  :=\sum_{j=0}^{m}{ \,
  \sup_{0 \leq t  \leq T} { \left(\norm{\psi^{(j)}(t)}_{H^{\frac{1}{2}}(\partial \Omega_0)}      
  +\norm{ \phi^{(j-1)}(t)}_{H^{-\frac{1}{2}+\mu}(\partial \Omega_0)}\right)}}.
\end{align}

Lemma~\ref{lemma:more_regularity} then directly gives the following corollary for the semidiscrete solution:
\begin{corollary}
  \label{exact_solution_in_hH_mu}
  For $\mu \in [0,1/2]$  and $m\in \N_0$,
  let $\gamma \uinc \in \Cpspace[H^{1/2}(\partial \Omega_0)]{m+2}$ and  
  $\partial_\nu \uinc \in \Cpspace[H^{-1/2 + \mu}(\partial \Omega_0)]{m+1}$.
  Then the solution $\bs x^h$ to~\eqref{eq:probMain} is in $\Cpspace[\hH_{\mu}]{m}$, and for $\ell \leq m$ it satisfies the bound
  \begin{align}
    \label{eq:regularity_estimate}
    \Big\|{\frac{d^\ell}{dt^\ell}\bs x^h(t)}\Big\|_{\mathbb{H}_\mu}
    &\leq C \, t\,
      \triplenorm{\big(\gamma \uinc,\partial_{\nu} \uinc\big)}_{\ell+2,T,\mu}.
  \end{align}
\end{corollary}
\begin{proof}
  For $\bs \xi^0:=\left(\gamma \uinc,0,\dots, 0\right)$, 
  $\bs \xi^1:=\left(\kappa_{0} \partial_\nu  \partial_t^{-1} \uinc,0, \dots, 0\right)$
  and $\bs \zeta:=\bs N(\bs \xi^0,\bs \xi^1)$,
  Theorem~\ref{thm:26} gives that $\bs x^h:=(\bs u^h,\bs w^h) \in \Cpspace[\vV]{m}$.
  
  We write $\bs x^h= \left(\bs x^h - \lifting \bs \zeta\right) + \lifting \bs \zeta$. Due to the boundary conditions on $\bs x^h$ 
  we get that $\bs x^h(t) - \lifting\bs \zeta(t) \in \dom(\AA)=\dom(\AAstar) \cap \ker(\bs B_h) $ and we can estimate:
  \begin{align*}
    \big\|\bs x^h(t) - \lifting\bs \zeta(t)\big\|_{\hH_{\mu}}
    &\leq \big\|{\bs x^h(t) - \lifting\bs \zeta(t)}\big\|_{\hH} +  \big\|{\AA\big(\bs x^h(t) - \lifting\bs \zeta(t)\big)}\big\|_{\hH}\\
    &\lesssim \big\|{\bs x^h(t)}\big\|_{\vV} + \big\|{\bs \zeta(t)}\big\|_{\mM}.
  \end{align*}
  The term $\norm{\lifting[\bs \zeta(t)]}_{\hH_{\mu}}$ can be estimated by Lemma~\ref{lemma:more_regularity}.
  The triangle inequality and Lemma~\ref{lemma:more_regularity} give 
  \begin{align*} 
  \big\|\bs x^h(t)\big\|_{\hH_{\mu}}
  & \leq \big\|\bs x^h(t) - \lifting\bs \zeta(t)\big\|_{\hH_{\mu}}+ \big\|\lifting\bs \zeta(t)\big\|_{\hH_{\mu}}  \\
&  \lesssim \|\bs x^h(t)\|_{\vV} + \|\bs \zeta(t)\|_{\mM} + \|\bs \zeta(t)\|_{\HH^{1/2}_\Gamma\times \HH^{-1/2+\mu}_\Gamma}.
\end{align*}
  To get to the explicit estimate in terms of the data, we use Theorem~\ref{thm:26}.
  We conclude the proof for $m=0$ with the remark that we can estimate  $\|{\bs \zeta(t)}\|_{\mM}\lesssim t  \|{\dot{\bs \zeta}(t)}\|_{\mM}$
  since $\bs \zeta(0)=0$.  
  A similar argument applied to the differentiated equation gives the result for $m \in \N$.  
\end{proof}

\subsection{Convergence of the time discretization}
We are now in position to prove the main convergence result for the time discretization.

\begin{theorem}  
  \label{thm:rk_convergence}  
  Let $\bs x^{h}:=(\bs u^h,\bs w^h)$ be the solution to Problem~\eqref{eq:33.5}
  and assume $\gamma \uinc \in \mathcal C^{p+4}([0,T],H^{1/2}(\partial \Omega_0))$ and
   $\partial_{\nu} \uinc \in \mathcal C^{p+3}([0,T],H^{-1/2+\mu}(\partial \Omega_0))$ 
  for some $\mu \in [0,1/2]$.

  Assume that the Runge-Kutta method employed is A-stable and that $\rkA$ is invertible.
  Set $\alpha:=1$ if the Runge-Kutta method is strictly A-stable and $\alpha:=0$ otherwise.

  If $\bs x^{h,k}$ is the solution to~\eqref{eq:4.5_rk_form}, then the following error estimates hold for $0< t_n \leq T$:
  \begin{subequations}
  \begin{align}
    \norm{\bs x^{h}(t_n) - \fdx(t_n)}_{\hH}
    & \leq C\, T^2 k^{\min(q+\mu+1+\alpha,p)}
     \triplenorm{(\gamma \uinc,\partial_{\nu} \uinc)}_{p+4,T,\mu},  \label{eq:rk_conv_Hh}\\ 
    \norm{\bs x^{h}(t_n) - \fdx(t_n)}_{\vV}
    & \leq C  \,T^2 k^{\min(q+\mu+\alpha,p)} 
      \triplenorm{(\gamma \uinc,\partial_{\nu} \uinc)}_{p+4,T,\mu},  \label{eq:rk_conv_A}
  \end{align}
  where $p$ and $q$ denote the classical and stage order of the Runge-Kutta method employed.
  For the trace component $\bs \psi^{h,k}$, computed in~\eqref{eq:def_postprocessed_fd},
  the following estimates can be shown:
  \begin{align}
    \norm{\bs \psi^h(t_n) - \bs \psi^{h,k}(t_n)}_{\HH^{1/2}}
    & \leq C \, T^2 k^{\min(q+\mu+\alpha,p)}\,
      \triplenorm{(\gamma \uinc,\partial_\nu \uinc)}_{p+4,T,\mu}. \label{eq:rk_conv_psi}
  \end{align}
  \end{subequations}
  The constants depend on the Runge-Kutta method, $\mu$, and the geometry.
\end{theorem}
\begin{proof}
  \def\fdY{{\bs Y}^{h,k}}
  \def\fdy{\bs y^{h,k}}
  We apply the theory developed in \cite{semigroups}.
  Since we are in the situation of an integrated boundary condition, we apply \cite[Thm.~{3.4}]{semigroups} to get \eqref{eq:rk_conv_Hh},
  using the regularity estimate~\eqref{eq:regularity_estimate}.

  By looking at the $Z$-transforms, it is easy to see that  $\fdY:=\dd \fdX=\AAstar \fdX$ solves the following problem
  \begin{subequations}
    \label{eq:differentiated_eq}
    \begin{align}
      \fdY(t_n) &= \fdy(t_n)\ones + k [\rkA \otimes \AA_{\star}] \fdY(t_n), \\
      \prodOp{\BB_{h}}\fdY(t_n) &= \left(\dd \big(\dd\big)^{-1} \dot{\bs{\Xi}}(t_n), 0, 0\right) = \bs N_{h}\dot{\bs{\Xi}}(t_n)=\bs N_{h}\dot{\bs{\xi}}(t_n+k \rkc), \\
      \fdy(t_{n+1})&=R(\infty)\fdy(t_n) + \rkb^{T}\rkA^{-1} \fdY(t_n),
    \end{align}
  \end{subequations}
  while $\bs y^h:=\AAstar \bs x^h$ solves $\dot{\bs y}^h=\AAstar \bs y^h$ and $\bs \BB_{h} \bs y^h= \bs N_{h} \dot{\bs{\xi}}$. 
  This means we can apply~\cite[Thm.~1 or 2]{mallo_palencia_optimal_orders_rk} to get the 
  following error estimate:
  \begin{align*}
    \norm{\AAstar\big(\bs x^{h}(t_n) - \fdx(t_n)\big)}_{\hH}
    & \lesssim   T k^{\min(q+\mu+\alpha,p)}
      \sum_{j=0}^{p+2}{\sup_{0\leq t \leq T}{\norm{\bs x^h(t)}_{\hH_{\mu}}}}\\
    &\lesssim T^2 k^{\min(q+\mu+\alpha,p)} \triplenorm{(\gamma \uinc, \partial_\nu \uinc)}_{p+4,T,\mu}.
  \end{align*}
  Together with the $\hH$-estimate for $\bs x^{h}(t_n) - \fdx(t_n)$ we can estimate the $\vV$-norm. The trace theorem then immediately gives~\eqref{eq:rk_conv_psi}.  
\end{proof}

\begin{theorem}
  \label{thm:rk_convergence_phi}
  Consider the same setting as in Theorem~\ref{thm:rk_convergence}
  and further assume that  $\gamma \uinc \in \mathcal C^{p+5}([0,T],H^{1/2}(\partial \Omega_0))$ and
   $\partial_{\nu} \uinc \in \mathcal C^{p+4}([0,T],H^{-1/2+\mu}(\partial \Omega_0))$ for some $\mu \in [0,1/2]$.
  
  If, in addition, the Runge-Kutta method is also stiffly accurate, then we can estimate $\bs \phi^h$ as defined in~\eqref{eq:def_postprocessed_fd} by:
  \begin{align}
    \norm{\bs\phi^h(t_n) - \bs \phi^{h,k}(t_n)}_{\HH^{-1/2}}
    & \leq C T^2 k^{r_{\bs \phi}}\triplenorm{(\gamma \uinc,\partial_\nu \uinc)}_{p+5,T,\mu},  \label{eq:rk_conv_lambda} 
  \end{align}
  where the rate $r_{\bs \phi}$ is given by
  \begin{align*}
    r_{\bs \phi}:=\begin{cases}
      q+ \mu +\alpha - 1/2 & \text{for } q+ \alpha < p \\
      q+ \alpha + \frac{\mu-1}{2} & \text{for } q+ \alpha = p \\
      p + \frac{\alpha -1}{2} & \text{for } q+ \alpha >  p.
      \end{cases}
  \end{align*}
\end{theorem}

\begin{proof}
  Reusing the notation from Theorem~\ref{thm:rk_convergence}, write $\bs y^{h}=:\big(\bs v^{h},\bs z^{h}\big)$.
  In order to estimate $\bs \phi^h - \bs \phi^{h,k}$
  we need to control $\fdiv{\bs z^{h}} - \fdiv{\bs z^{h,k}}$.  
  This can be estimated by using \cite[Thm.~{3.5}]{semigroups}.
  Together with the regularity estimate~\eqref{eq:regularity_estimate} we get the rate
  \begin{align*}
    \norm{\fdiv{\bs z^{h}} - \fdiv\bs z^{h,k}}_{\hH}
    &\lesssim T^2 k^{\min(q+\mu,p)+\alpha -1} \triplenorm{(\gamma \uinc,\partial_\nu \uinc)}_{p+5,T,\mu}.      
  \end{align*}
  Define $r_0:=\min(q+\mu+\alpha,p)$ and $r_1:=\min(q+\mu,p)+\alpha -1$ for the
  convergence rates of $\norm{\bs z^{h} - \bs z^{h,k}}_{\hH}$ and $\norm{\fdiv \bs z^{h} - \fdiv \bs z^{h,k}}_{\hH}$, respectively.

  Setting $e_\ell:=z^{h}_\ell - z^{h,k}_\ell$ for the $\ell$-th component of the error,
  we can calculate for
  $\eta \in H^{1/2}(\partial \Omega_{\ell})$ and $v\in H^1(\R^d)$ with
  $\traceint{\ell} v =\traceext{\ell} v=\eta$:
  \begin{align*}    
    \dualproduct[\partial \Omega_{\ell}]{\phi^h_{\ell} - \phi^{h,k}}{\eta}     
    &= \big( e_{\ell} ,\nabla v \big)_{L^2(\R^d \setminus \partial \Omega_\ell)}
      + \big(\fdiv{e}_{\ell},v \big)_{L^2(\R^d \setminus \partial \Omega_{\ell})} \\    
    &\lesssim \begin{multlined}[t][11cm]
    \big(\!\norm{e_{\ell}}_{L^2(\R^d \setminus \partial \Omega_\ell)}\! +\! k^{r_0-r_1}
    \!\norm{ \fdiv e_{\ell}}_{L^2(\R^d \setminus \partial \Omega_\ell)}\!\big)  \\
    \times \big(\!\norm{\nabla v}_{L^2(\R^d \setminus \partial \Omega_\ell)}\! +\! k^{r_1-r_0}\!\norm{ v}_{L^2(\R^d \setminus \partial \Omega_\ell)}\!\big)
  \end{multlined}
    \\            
    &= \left( \bigO(k^{r_0}) + \bigO(k^{r_0-r_1} k^{r_1})\right)
    \big(\norm{\nabla v}_{L^2(\R^d \setminus \partial \Omega_\ell)} + k^{r_1-r_0}\norm{ v}_{L^2(\R^d \setminus \partial \Omega_\ell)}\big).
  \end{align*}
  
      We are still free to pick the precise lifting $v$. 
      By \cite[Prop.~{2.5.1}]{sayas_book}, we have for arbitrary $\beta> 0$ 
      \[
      \inf\{ k^{-\beta}\|v\|_{\R^d}+\|\nabla v\|_{L^2(\R^d)}\,:\, v\in H^1(\R^d),\,\traceint{\ell}{v}=\eta\}
      \le C_\beta\max\{1,k^{-\beta/2}\}\| \eta\|_{H^{1/2}(\partial_{\Omega_\ell})}.
      \]
      We thus get the convergence rate:
      $$
      \norm{\bs \phi^h(t_n) - \bs \phi^{h,k}(t_n)}_{\HH^{-1/2}} \lesssim C(\uinc) T^2 k^{r_0+ (r_1-r_0)/2}.
      $$
      The full statement then follows by explicitly checking the different cases and working out the
      dependencies of $C(\uinc)$.
      \end{proof}

  Finally, we have the following convergence result away from the boundary:
  \begin{theorem}
    \label{thm:local_and_pointwise_convergence}
    Consider the setting of Theorem~\ref{thm:rk_convergence}.
    In addition, assume that $w=0$ is the only point in the right-half plane where
    the stabilty function of the RK-method satisfies $r(w)=1$ .
    Let the incident wave satisfy $\gamma \uinc \in \Cpspace[H^{1/2}(\partial \Omega_0)]{2p+5-q}$ as well as  
    $\partial_{\nu} \uinc \in \Cpspace[H^{-1/2}(\partial \Omega_0)]{2p+4-q}$.
    %

  Let $\widetilde{\Omega}\subseteq \R^d$ be an open subset of $\R^d$ such that
  $\overline{\widetilde{\Omega}} \cap \Gamma=\emptyset$, and $\varphi_{\widetilde{\Omega}}$
  a smooth cutoff function with $\varphi_{\widetilde{\Omega}} \equiv 1$ on ${\widetilde{\Omega}}$
  and $\varphi_{\widetilde{\Omega}}=0$ in  a neighborhood of $\Gamma$. 
  Then the following estimates hold for $t_n=n k$ with $t_n \leq T$:  
  \begin{multline*}
    \big\|{\varphi_{\widetilde{\Omega}}(\dot{\bs u}^{h}(t_n) - \partial_k \fdu(t_n))}\big\|_{\hH}
    +
    \big\|{\varphi_{\widetilde{\Omega}}(\nabla\bs u^{h}(t_n) - \nabla \fdu(t_n))}\big\|_{\hH} \\
    \lesssim C(\widetilde{\Omega})    
     (1+T^{p-q+2})  k^{p}   \triplenorm{(\gamma \uinc,\partial_{\nu} \uinc)}_{2p+3-q,T,\mu}
  \end{multline*}
  with a constant $C(\widetilde{\Omega})$ which depends on $\widetilde \Omega$ in
  addition to the usual dependencies.
  
  Let $d\leq 3$. Then, we can derive the following pointwise bound for $x \in \R^d \setminus \Gamma$.
      \begin{multline*}
    \big|{\dot{\bs u}^{h}(x,t_n) - \partial_k \fdu(x,t_n)}\big|
    +
    \big|{\nabla\bs u^{h}(x,t_n) - \nabla \fdu(x,t_n)}\big\| \\    
    \lesssim
     C(x)
     (1+T^{p-q+2})  k^{p}   \triplenorm{(\gamma \uinc,\partial_{\nu} \uinc)}_{2p+4-q,T,\mu}.
  \end{multline*}
\end{theorem}
\begin{proof}
  Follows from Theorem~\ref{thm:local_convergence}. We point out that
  our quantities of interest are the components of $\bs Y^{h,k}$ as
  defined in~\eqref{eq:differentiated_eq}.
  In order to fulfill Definition~\ref{def:cutoff_operators},
  the cutoff operator $\opT$ is given by multiplication
  with the cutoff function $\varphi_{\widetilde{\Omega}}$.
  To see that $\AA^{-m} \opC_{m}$ is bounded, we will use
  Lemma~\ref{lemma:iterated_commutators_2}, and the easier to investigate commutators $\widetilde{\opC}_m$.
  By the product rule it is easy to see
  that $\opT$ is a bounded operator on $\hH$ and analogously for all its iterated commutators
  $\widetilde{\opC}_j$. For example, for $j=1$ we have
  \begin{align*}
    \AAstar \opT \colvec{\bs v \\\bs w} -   \opT \AAstar\colvec{\bs v \\\bs w}
    &=\colvec{\nabla \varphi_{\widetilde{\Omega}} \nabla \cdot \bs w \\
    \nabla \varphi_{\widetilde{\Omega}} \cdot \nabla v}
  \end{align*}
  and so on.
  Thus, we have that for any
  $u \in \dom(\AA)$ it holds that $\widetilde{\opC}_j u \in \dom(\AA)$. Applying
  \eqref{eq:iterated_commutators2}
  we get
  $$
  \AA^{m} \opC_{m} u =\AA^{m} \AA^{-m} \widetilde{\opC}_m u.
  $$
  Since the right-hand side is a bounded operator and $\dom(\AA)$ is dense, we get that $L(\opT,m)=m$.
  
  By the support
  properties of $\varphi_{\widetilde{\Omega}}$ we get that $\opT$ is an admissible cutoff operator
  of arbitrary order (as specified in Definition~\ref{def:cutoff_operators}),
  thus we may use $M=p-q$ to regain full order for the $\hH$-estimate.

  To see the pointwise estimate, we use $\opT_j:=\AA^{j} \varphi_{\widetilde{\Omega}}$ for $j=1,2$.
  This operator has $L(\opT_j,m)=j+m$.
  This gives convergence for $\opT_j(\nabla u^h- \nabla u^{h,k})$ of full classical order. The stated result
  then follows from the standard Sobolev embedding $H^2(\widetilde{\Omega})\subseteq \mathcal{C}^0(\widetilde{\Omega})$
  for $d\leq 3$ (see, e.g., \cite[Sect. 5, Thm. 6]{evans_pde}).  
\end{proof}

  \subsection{Convergence of the fully discrete scheme}
\label{sect:fully_discrete}
In this section, we collect the previous convergence results for the space and time discretization to give
explicit convergence rates for the fully discrete systems.
In order to quantify the convergence rates of the full discretization, we make the following assumption on the spaces $\XX_h$ and $\YY_h$
(see Section~\ref{sect:construction} on how to construct spaces satisfying these assumptions).
  \begin{theorem}
    \label{thm:full_convergence_quasioptimal}
    Let the incident wave satisfy $\gamma \uinc \in \Cpspace[H^{1/2}(\partial \Omega_0)]{p+5}$ as well as  
    $\partial_{\nu} \uinc \in \Cpspace[H^{-1/2+\mu}(\partial \Omega_0)]{p+4}$
    for some $\mu \in [0,1/2]$. 
  
  Let $p$ denote the classical order of the Runge-Kutta method and $q$ its stage order.
  Assume that the method is A-stable and $\rkA$ is invertible.  Set $\alpha:=1$ if the method is strictly A-stable
  (i.e., $\abs{r(z)}<1$ for  $0\neq z \in \ii \R$ and $r(\infty)\neq 1$),
  and set $\alpha:=0$ otherwise. 
  Let $\bs \lambda:=(\bs \phi,\bs \psi)$ be the exact solution of~\eqref{eq:39}, and $\bs x:=(\bs u,\bs w)$ be the corresponding exact solution to~\eqref{eq:probMain}.
  Let $\bs \lambda^{h,k}:=(\fdl,\fdp)$ denote the solutions to~\eqref{eq:4.1},  and $\fdx=(\fdu,\fdw)$ the post-processing using the representation formula
  \eqref{eq:4.1c} and~\eqref{eq:def_postprocessed_fd}.
  
  Then the following estimates hold for $t_n=n k$ with $t_n \leq T$:  
  \begin{subequations}
    \begin{multline*}      
      \norm{ \bs x(t_n) - \fdx(t_n)}_{\vV}
      \lesssim
      T \; \
      \sum_{j=0}^{2} 
      {\Big[\max_{0\leq \tau \leq t_n} \inf_{\substack{\bs \phi_{h,j} \in \XX_h \\ \bs \psi_{h,j} \in \YY_h}}
        {\big\|{\big(\bs \psi^{(j+1)} - \bs \psi_{h,j},\bs \phi^{(j)}-\bs \phi_{h,j}\big)}}\big\|_{\mathbb{B}}
        \Big]}\\
      + T^2 k^{\min(q+\mu+1+\alpha,p)}   \triplenorm{(\gamma \uinc,\partial_{\nu} \uinc)}_{p+4,T,\mu}.
    \end{multline*}   
    If the method is stiffly accurate, we get:
    \begin{multline*}      
      \norm{ \bs \phi(t) - \fdl(t_n)}_{\HH^{-1/2}}
      \lesssim  T \; \
      \sum_{j=0}^{3} 
      {\Big[\max_{0\leq \tau \leq t_n} \inf_{\substack{\bs \phi_{h,j} \in \XX_h \\ \bs \psi_{h,j} \in \YY_h}}
        {\big\|{\big(\bs \psi^{(j+1)} - \bs \psi_{h,j},\bs \phi^{(j)}-\bs \phi_{h,j}\big)}}\big\|_{\mathbb{B}}
        \Big]}\\
      +T^2 k^{r_{\bs \phi}}  \triplenorm{(\gamma \uinc,\partial_{\nu} \uinc)}_{p+5,T,\mu},            
    \end{multline*}
    where the rate is given by
    \begin{align*}
          r_{\bs \phi}:=\begin{cases}
            q+ \mu +\alpha - 1/2 & \text{for } q+ \alpha < p \\
            q+ \alpha + \frac{\mu-1}{2} & \text{for } q+ \alpha = p \\
            p + \frac{\alpha -1}{2} & \text{for } q+ \alpha >  p.
      \end{cases}
  \end{align*}
    \end{subequations}
    The implied constant $C(\uinc)$  
      depends on the geometry and the Runge-Kutta method but is independent of
      the incident wave, $h$, $k$, and $T$.
  \end{theorem}
  \begin{proof}
  We estimate
  $$\big\|{ \bs x(t_n) - \fdx(t_n)}\big\|_{\vV} \lesssim\big\|{ \bs x(t_n) - \bs x^{h}(t_n)}\big\|_{\vV} + \big\|{ \bs x^h(t_n) - \fdx(t_n)}\big\|_{\vV}.$$
  The convergence of the semi-discretization in space is quasi-optimal by Theorem~\ref{thm:approximation_estimate_in_space}. 
  The convergence with respect to time can be estimated by Theorem~\ref{thm:rk_convergence}, where Lemma~\ref{lemma:more_regularity} and
  Corollary~\ref{exact_solution_in_hH_mu} tell us that we may use the value $\mu=1/2$.
  The bounds on $\bs \psi - \fdp$ follows from the continuity of the trace operator.  The trace on $\bs \phi - \fdl$ follows along the same
  lines but using the bounds proved in Theorem~\ref{thm:approximation_estimate_in_space} and Theorem~\ref{thm:rk_convergence_phi}.
\end{proof}

\begin{assumption}
  \label{ass:assumption_on_spaces}
  For a parameter $r \in \N_0$, the discrete spaces $\XX_h$ and $\YY_h$ satisfy the following approximation property
  for all $\bs \phi:=\left(\phi_{\ell}\right)_{\ell=0}^{L} \in \XX$ with $\phi_{\ell}\in \hppw{r+1}\left(\partial \Omega_{\ell}\right)$
  and all $\bs \psi:=\left(\psi_{\ell}\right)_{\ell=0}^{L} \in \YY$, $\psi_{\ell}\in \hppw{r+2}\left(\partial \Omega_{\ell}\right)$
  such that the lifting in~\eqref{eq:def_YY} is a continuous function on $\Gamma$: 
  \begin{subequations}
    \begin{align}
      \inf_{\bs \phi_h \in \XX_h}{\norm{\bs \phi - \bs \phi_h}_{\HH^{-1/2}}}
      &\leq C h^{r+3/2} \sum_{\ell=0}^{L}{\norm{\phi_{\ell}}_{\hppw{r+1}\left(\partial \Omega_{\ell}\right)}}, \\
      \inf_{\bs \psi_h \in \YY_h}{\norm{\bs \psi - \bs \psi_h}_{\HH^{1/2}}}
      &\leq C h^{r+3/2} \sum_{\ell=0}^{L}{\norm{\psi_{\ell}}_{\hppw{r+2}\left(\partial \Omega_{\ell}\right)}},
    \end{align}
  \end{subequations}
  where the constant $C$ may depend on $r$ and the geometry but not on $h$, $\bs \phi$ or $\bs \psi$.
\end{assumption}

Since we have to implement the scheme in practice, we only consider the case $\partial_\nu u^{inc} \in L^2(\partial \Omega_0)$,
i.e., $\mu=1/2$. Then the following theorem holds.
\begin{corollary}
  \label{thm:full_convergence}
  Let the assumptions of Theorem~\ref{thm:full_convergence_quasioptimal} hold.
    Assume that the traces of the exact solution satisfy
  $\phi_{\ell} \in \Cpspace[\hppw{r+1}\left(\partial \Omega_{\ell}\right)]{3}$, $\psi_{\ell} \in \Cpspace[\hppw{r+2}\left(\partial \Omega_{\ell}\right)]{3}$
  for some $r \in \N_0$.
  Also assume that $\bs \psi^{(j)}$ admits a lifting to $H^1(\R^d)$ that is continuous on $\Gamma$ for $j=0,\dots,3$.  
  Assume $d\leq 3$ and let  Assumption~\ref{ass:assumption_on_spaces}
  be satisfied for $\XX_h$ and $\YY_h$ with the same parameter $r$ as in the regularity assumptions.
  Then the following estimates hold for $t_n=n k$ with $t_n \leq T$:  
  \begin{subequations}
    \begin{align}      
      \norm{ \bs x(t_n) - \fdx(t_n)}_{\vV} &\leq C(\uinc) \left(T h^{r+3/2} + T^2 k^{\min(q+\alpha + 1/2,p)}\right),
                                              \label{eq:full_convergence_1}     \\
      \norm{ \bs \psi(t_n) - \fdp(t_n)}_{\HH^{1/2}} &\leq C(\uinc) \left( T h^{r+3/2} + T^2 k^{\min(q+\alpha+1/2,p)}\right).
                                                         \label{eq:full_convergence_psi}
    \end{align}
    If the method is stiffly accurate, we get:
    \begin{align}
      \norm{ \bs \phi(t) - \fdl(t_n)}_{\HH^{-1/2}} &\leq  C(\uinc) \big(T h^{r+3/2} + T^2 k^{r_{\bs \phi}}\big).
                                                     \label{eq:full_convergence_lambda}
    \end{align}
    where the rate given by
    \begin{align*}
    r_{\bs \phi}:=\begin{cases}
      q+ \alpha & \text{for } q+ \alpha < p \\
      q+ \alpha - \frac{1}{4} & \text{for } q+ \alpha = p \\
      p + \frac{\alpha -1}{2} & \text{for } q+ \alpha >  p.
      \end{cases}
  \end{align*}
    \end{subequations}
    The constant $C(\uinc)$ depends on the incident  wave, the geometry, the Runge-Kutta method, and the constants in Assumption~\ref{ass:assumption_on_spaces},
    but is independent of $h$, $k$, and $T$.
\end{corollary}
\begin{proof}
  Follows directly from Theorem~\ref{thm:full_convergence_quasioptimal}, the regularity assumptions and Assumption~\ref{ass:assumption_on_spaces}.
\end{proof}

\section{Particular geometric configurations}
\label{sect:particular_configurations}

In this section we present two simple geometric configurations
that fit into our framework. We show how in these cases, using the spaces presented
in Section~\ref{sect:construction}, the method analyzed in this paper can be considered
equivalent to known methods in the literature.

\subsection{Multiple homogeneous scatterers}

Assume that $\Omega_\ell$, $\ell=1,\ldots, L$, are bounded Lipschitz domains such that each of the boundaries $\partial\Omega_\ell$ is connected and these boundaries are mutually disjoint. In this case
\[
H^{\pm 1/2}(\partial\Omega_0)\equiv \prod_{\ell=1}^L H^{\pm1/2}(\partial\Omega_\ell).
\]
Using this identification we can establish isomorphisms
\begin{alignat*}{6}
\prod_{\ell=1}^L H^{1/2}(\partial\Omega_\ell)
	\ni (\psi_\ell)_{\ell=1}^L & \longmapsto && 
	((\psi_1,\ldots,\psi_L),\psi_1,\ldots,\psi_L) \in \YY,\\
\prod_{\ell=1}^L H^{-1/2}(\partial\Omega_\ell)
	\ni(\phi_\ell)_{\ell=1}^L& \longmapsto &&
	(-(\phi_1,\ldots,\phi_L),\phi_1,\ldots,\phi_L))\in \XX.
\end{alignat*}
In some way, the second isomorphism can be understood as changing the orientation of the normal vector exterior to the unbounded domain $\Omega_0$ to make it point towards it.
The boundary integral formulation that we got before (see (\ref{eq:39})) can be reformulated in the reduced representation. What we obtain is the Costabel-Stephan system of TDBIE for transmission problems of \cite{tianyu_sayas_costabel}, analyzed as a first order system; see also~\cite{qiu_thesis}.
The frequency domain version of this reduced system is the classical formulation in~\cite{costabel_stephan_trasmission}.

\subsection{Layered scatterers}
Another interesting simplified situation, already considered in \cite{qiu_thesis},
is the one of scatterers containing separate inclusions, which can themselves contain inclusions, etc. We can represent the geometric configuration as a tree, whose nodes are the domains $\Omega_\ell$, rooted at $\Omega_0$, and whose edges are the connected components of $\Gamma$, so that an edge connecting two vertices is the common boundary of both domains. We can  number the components of $\Gamma$ as $\Gamma_i$, $i=1,\ldots,L$ and assign a parent node $p(\ell)\in \{0,\ldots,L\}$ to each node $\ell\ge 1$ so that
\[
\partial\Omega_\ell \cap \partial \Omega_{p(\ell)}=\Gamma_\ell, \qquad \ell\in \{1,\ldots,L\}
\]
and
\[
\partial\Omega_\ell=\Gamma_\ell \cup \left( \cup\{ \Gamma_i\,:\, p(i)=\ell\}\right),
\qquad \ell\in \{0,\ldots,L\}
\]
where we denote $\Gamma_0:=\emptyset$ to unify notation;  see Figure~\ref{fig:layered_scatterers} for a schematic representation.
We can thus identify
\[
H^{\pm1/2}(\partial\Omega_\ell)
	\equiv H^{\pm1/2}(\Gamma_\ell) \times \prod_{\{i:p(i)=\ell\}} H^{\pm1/2}(\Gamma_i)
\]
and use this to define isomorphisms
\begin{alignat*}{6}
\prod_{i=1}^L H^{1/2}(\Gamma_i) \ni (\psi_i)_{i=1}^L
	&\longmapsto && 
	\big( (\psi_i)_{p(i)=0}, (\psi_\ell, (\psi_i)_{p(i)=\ell})\big)_{\ell=0}^L \in \YY,\\
\prod_{i=1}^L H^{-1/2}(\Gamma_i)\ni (\phi_i)_{i=1}^L
	&\longmapsto && 
	\big( -(\phi_i)_{p(i)=0}, 
	(\phi_\ell, -(\phi_i)_{p(i)=\ell})\big)_{\ell=0}^L \in \XX,
\end{alignat*}
morally corresponding to fixing all normals so that they point towards the exterior of the closed boundaries $\Gamma_i$.

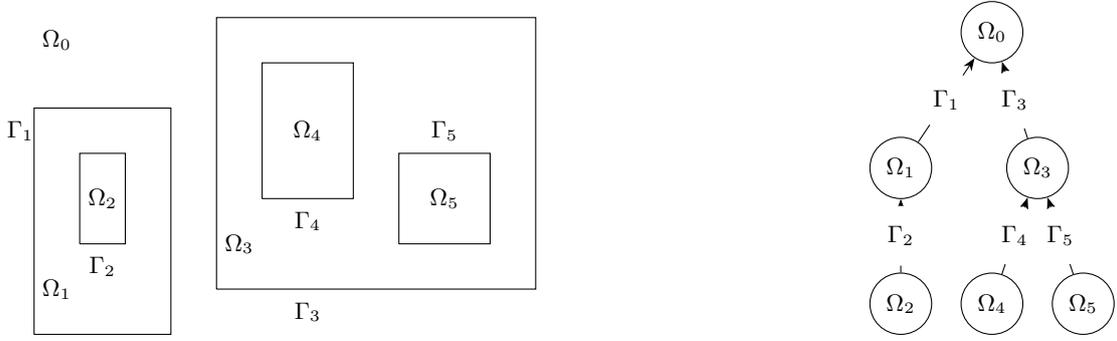
\begin{figure}[htb]
\begin{center}
\begin{tikzpicture}[scale=0.6]
\draw (1,1) -- (4,1) -- (4,6) -- (1,6) -- (1,1);
\draw (2,3) -- (3,3) -- (3,5) -- (2,5) -- (2,3);
\draw (5,2) -- (12,2) -- (12,8) -- (5,8) -- (5,2);
\draw (6,4) -- (8,4) -- (8,7) -- (6,7) -- (6,4);
\draw (9,3) -- (11,3) -- (11,5) -- (9,5) -- (9,3);
\draw (1.5,7.5) node {\footnotesize $\Omega_0$};
\draw (1.5,2) node {\footnotesize $\Omega_1$};
\draw (2.5,4) node {\footnotesize $\Omega_2$};
\draw (5.5,3) node {\footnotesize $\Omega_3$};
\draw (7,5.5) node {\footnotesize $\Omega_4$};
\draw (10,4) node {\footnotesize $\Omega_5$};
\draw (0.7,5.5) node {\footnotesize $\Gamma_1$};
\draw (2.5,2.5) node {\footnotesize $\Gamma_2$};
\draw (7,1.5) node {\footnotesize $\Gamma_3$};
\draw (7,3.5) node {\footnotesize $\Gamma_4$};
\draw (10,5.5) node {\footnotesize $\Gamma_5$};
\end{tikzpicture}
\hfill
\begin{tikzpicture}[scale=0.6]
\begin{scope}[every node/.style={circle,draw}]
    \node (A) at (3,5) {\footnotesize $\Omega_0$};
    \node (B) at (1,2) {\footnotesize $\Omega_1$};
    \node (C) at (1,-1) {\footnotesize $\Omega_2$};
    \node (D) at (4,2) {\footnotesize $\Omega_3$};
    \node (E) at (3,-1) {\footnotesize $\Omega_4$};
    \node (F) at (5,-1) {\footnotesize $\Omega_5$} ;
\end{scope}
\begin{scope}[>={Stealth[black]},
              every node/.style={fill=white,circle},
              every edge/.style={draw}]
    \path [<-] (A) edge node {\footnotesize $\Gamma_1$} (B);
    \path [<-] (B) edge node {\footnotesize $\Gamma_2$} (C);
    \path [<-] (A) edge node {\footnotesize $\Gamma_3$} (D);
    \path [<-] (D) edge node {\footnotesize $\Gamma_4$} (E);
    \path [<-] (D) edge node {\footnotesize $\Gamma_5$} (F);
\end{scope}
\end{tikzpicture}
\end{center}
\caption{A cartoon representing two scatterers with inclusions, and the corresponding tree. The arrows point to the parenting (surrounding) domain and are tagged with the common intersection.}
\label{fig:layered_scatterers}
\end{figure}

\section{Numerical examples}
\label{sect:numerics}
We implemented the proposed method in 2D, using the algorithm for fast solutions of CQ-problems described in \cite{banjai_sauter_rapid_wave}. For the implementation of the
standard BEM operators, we relied on code developed by F.-J.~Sayas and his group at the University of Delaware.
This code has not been published as of yet and differs from the better known deltaBEM~\cite{deltabem}
  package. Namely, it implements a Galerkin scheme instead of a Nystr\"om type method.

Assembling the matrices for Problem~\eqref{eq:4.1} can be done fairly simply using existing boundary element code with the
spaces described in Section~\ref{sect:construction}. In order to do so, we only have to provide the transfer matrices
$\mathcal{R}^{\top}:  \mathcal{Q}_h \times \mathcal{P}_h  \to  \YY_h \times \XX_h$,
which map the degrees of freedom from the boundary element spaces $\mathcal{P}_h$ and $\mathcal{Q}_h$ to the
standard BEM spaces on each subdomain (respecting the orientation of the surfaces in the case of $\mathcal{P}_h$).

The discretization of~\eqref{eq:4.1} is equivalent to:
find ${\Lambda}^{h,k}:=({\Psi}^{h,k},{{\Phi}}^{h,k}) \in \prodSpace{\mathcal{Q}_h \times \mathcal{P}_h } $ such that
\begin{align*}
  \check{\mathcal{R}}\check{ \mathrm Q}_\kappa \mathrm{\bs{C}}(\dd)\check{\mathrm Q}_\kappa^{-1} \check{\mathcal{R}}^\top {\Lambda}^{h,k}
      &= 
      \check{\mathcal{R}} \check{\mathrm Q}_\kappa (\mathrm{\bs{C}}(\dd)-\tfrac12{\mathrm I})
      \check{\mathrm Q}_\kappa^{-1} {\mathrm{\bs J}}(\dd)   \dot{\bs{\Xi}}^k
\end{align*}
and then using the transfer matrices $\mathcal{R}$ to get back
the functions $\bs \Lambda^{h,k}:=(\bs \Psi^{h,k},\bs \Phi^{h,k}):=\check{\mathcal{R}}^\top \Lambda^{h,k}$ in $\prodSpace{\YY_h \times \XX_h}$.

When using the approach from \cite{banjai_sauter_rapid_wave} for solving the convolution system, we need to solve $n=T/k$ problems in the frequency domain.
Since the operator $\mathrm{\bs{C}}(s)$ appears on both the left- and right-hand side, it only has to be assembled once if we combine the
steps for computing the right-hand side and solving.
The computation of $\bs J(\dd)$ does not incur any significant additional cost, as it corresponds to
a multiplication with the matrix
 $k\left(\delta(z)\right)^{-1}$ of the second component of the right-hand side during the CQ-algorithm.

 As the model geometry, we use a simple checkerboard pattern consisting of $2 \times 2$ unit squares
 and the wave speed vector $(\kappa_\ell)_{\ell=0}^{4}:=(2,3,1,5,7)$.
  
In order to be able to quantify the convergence, we prescribe an exact solution in the following way:
On each subdomain $\Omega_{\ell}$, $\ell=1,\dots, L$, the solution $u_{\ell}$ is given as a plane wave with
\begin{align*}  
  u_{\ell}(x,t)&:=G\left( d_{\ell}\cdot  x - \kappa_{\ell} (t-t_{\text{lag}}) \right), \qquad \text{ where } \quad G(z):=e^{{-2z}/{\alpha}}\,\sin(z).
\end{align*}
Here, $ d_{\ell} \in \R^2$ denotes the direction of the wave, and we chose the following parameters:
\begin{alignat*}{3}
  d_{\ell}&:=\begin{cases} \frac{1}{\sqrt{2}}(1,-1)^{\top} &\text{$\ell$ is even}\\  
    \frac{1}{\sqrt{2}}(1,1)^{\top} &\text{$\ell$ is odd}  
    \end{cases},
\end{alignat*}
$t_{\text{lag}}=5/2$, and $\alpha:=1/4$.  In order not to have to concern ourselves with radiation conditions,
we chose $u_0:=0$ for the solution in the exterior. 
For the  boundary traces, we made the following choice, using the function $\chi(t):=t^9 \,e^{-2t}$ to ensure homogeneous initial conditions:
\begin{align*}
\psi(x,t)&:=\chi(t)\,\sin(x_1) \cos(x_2) \qquad \text{ and } \qquad \phi(x,t):=\chi(t) \, \cos(x_1) \sin(x_2).
\end{align*}

(These functions are to be understood as functions on the skeleton. $\bs \phi$ is then built by restricting to the subdomains and multiplying with
a sign function as is done in Section~\ref{sect:construction}, whereas $\bs \psi$ is obtained via the restrictions to the subdomains.)
The boundary data $\bs \xi^0$ and $\bs \xi^1$ were then calculated accordingly
in order to yield these solutions.


\begin{figure}
  \begin{subfigure}{0.5\textwidth}
  \center
  \includeTikzOrEps{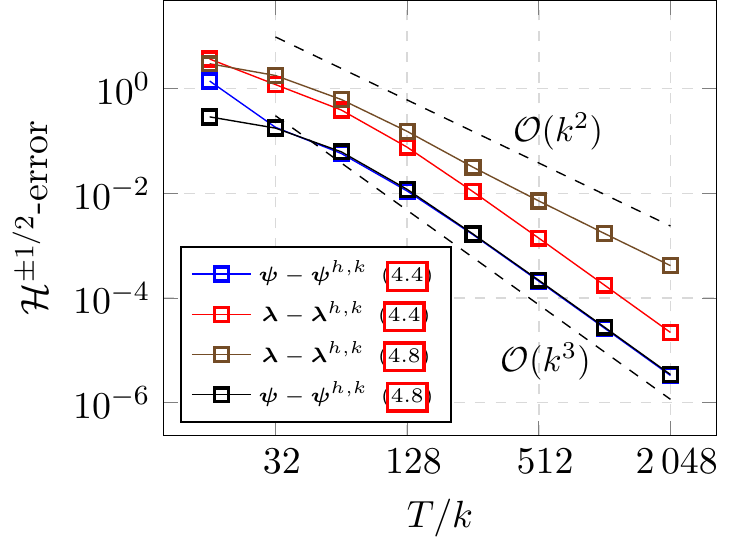}
  \caption{Convergence rates using a 2-stage Radau IIa method.}
  \label{fig:conv_radau2}
\end{subfigure}\quad
\begin{subfigure}{0.5\textwidth}
  \center
  \includeTikzOrEps{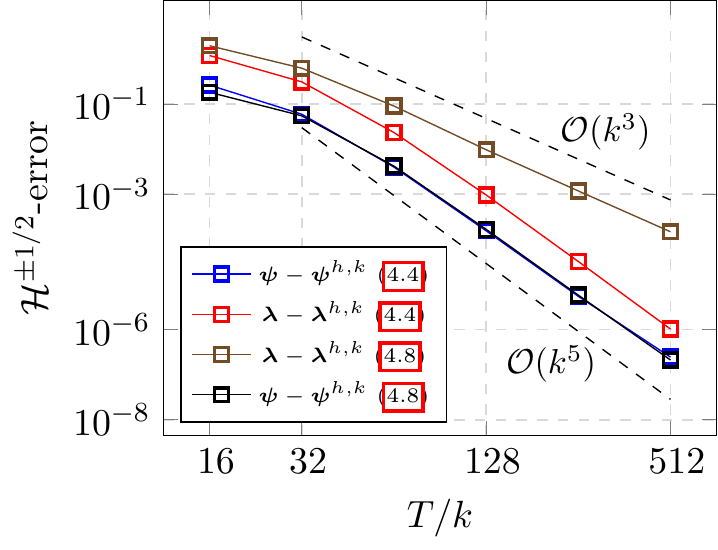}
  \caption{Convergence rates using a 3-stage Radau IIa method. }
  \label{fig:conv_radau3}
\end{subfigure}
\caption{Comparison of discretization schemes}
\end{figure}

\begin{figure}
  \center
  \begin{subfigure}{0.48\textwidth}
    \captionsetup{skip=0mm}
    \includeTikzOrEps{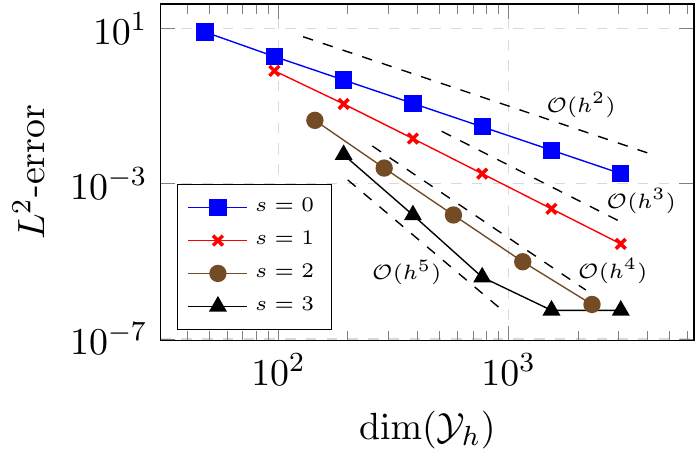}
    \caption{Convergence $\bs \psi -\bs \psi^h$}
  \end{subfigure}
  \begin{subfigure}{0.48\textwidth}
    \captionsetup{skip=0mm}
    \includeTikzOrEps{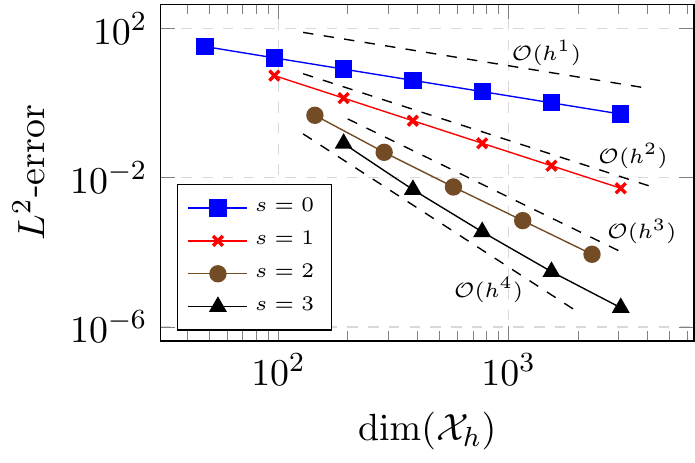}
    \caption{Convergence $\bs \lambda -\bs \lambda^h$}
  \end{subfigure}
  \caption{Convergence rates w.r.t. the spatial discretization}
  \label{fig:conv_space}
\end{figure}

\begin{example}
  \label{ex:conv_radau2}
  In this example, we are interested in the convergence with respect to the time discretization. Therefore, we fix a
  fine uniform mesh with $h \approx 0.03125$ and use $r=4$, i.e., quartic polynomials for the discontinuous space and quintic for the continuous
  splines. We apply a two-stage Radau IIa method, which satisfies $q=2$ and $p=3$. By Theorem~\ref{thm:full_convergence}, we
  expect order $\bigO(k^3)$ for the Dirichlet trace and $\bigO\left(k^{2.75}\right)$ for the Neumann trace when using~\eqref{eq:4.1}.
  As a comparison, we also compute the solutions using~\eqref{eq:the_bad_method}. Figure~\ref{fig:conv_radau2} shows the result.
  Most notably, it shows that when using~\eqref{eq:4.1}, the Neumann trace outperforms our predictions and converges with the full classical order.
  We also see that using~\eqref{eq:the_bad_method} gives a reduced order of $2$ when approximating $\bs \lambda$. 
\eremk
\end{example}

\begin{example}
  \label{ex:conv_radau3}
  We perform the same experiment as in Example~\ref{ex:conv_radau2}, but use a 3-stage Radau IIa method. We expect
  orders $\bigO(k^{4.5})$ and $\bigO(k^{4})$ for the Dirichlet and Neumann traces respectively.
  Again the method~\eqref{eq:4.1} outperforms our expectations, giving the full classical order $5$,
  while using~\eqref{eq:the_bad_method} gives a reduced rate.
\eremk
\end{example}

\begin{remark}
  Examples~\ref{ex:conv_radau2} and~\ref{ex:conv_radau3} showed that the proposed method often outperforms
  the predictions of the theory. While a full theoretical explanation for this effect is still lacking, 
  partial answers can be found in~\cite{cq_superconvergence}  for a simpler model problem.  
\eremk
\end{remark}

\begin{example}
  \label{ex:conv_space}
  We use the same model problem as in Example~\ref{ex:conv_radau2}, but we fix the time discretization at $k \approx 0.015$ using a $3$-stage
  Radau IIa method. We vary the approximation in space by performing successive uniform refinements of the grid, and compare different polynomial
  degrees $s=0,\dots,3$. Since it is easier to compute, we consider the $L^2$-norm of the errors.
  In Figure~\ref{fig:conv_space} we observe the optimal convergence rates until to an error of $\approx 10^{-6}$ is reached, at which point other error
  contributions prohibit further convergence.  
\eremk
\end{example}

  \begin{example}
    \label{ex:complex_geometry}
    We consider a more realistic scattering problem for which no exact solution is available. Errors are estimated 
    by comparing with a reference solution computed to higher accuracy.
    We consider a 3-by-3 checkerboard domain. The wavenumbers are given by
    $$
    \begin{bmatrix}
      5 & 0.2& 5 \\
      0.2 & 5 & 0.2 \\
      5 & 0.2 & 5
    \end{bmatrix}, 
    $$ and in the exterior it is taken to be $1$.
    This obstacle is hit by an incoming wave of the form
    $
    u^{\textrm{inc}}(x,t):=H(x \cdot d -t)
    $
    with $d:=[1,0]^{\top}$ and
   \begin{align*}
     H(x)&:=\sin(10t) \varphi(t/0.1) \varphi((t-0.3)/0.1), \qquad \text{using} \\
    \varphi(x)&:=
      x^5 \Big(1-5(x-1)+15(x-1)^2-35(x-1)^3+70(x-1)^4-
      126(x-1)^5\Big)
   \end{align*}
   for $x \in (0,1)$ and $\varphi(x):=1$ for $x\geq 1$.
   The function $\varphi \in \mathcal{C}^{5}(\R)$ represent a windowing function, smoothly connecting 
   $0$ and $1$. The precise function was taken from the examples of the DeltaBEM package~\cite{deltabem}.

   Figure~\ref{fig:complicated_evolution} depicts the evolution of the solution over time. Once
   the incoming wave hits the scatterer, we observe complicated intersections, especially
   at the triple-points where the wave number changes between domains.
   Figure~\ref{fig:complicated_convergence} presents the convergence of the method,
   where we compared the solution to the one obtained by halving the step size.
   The boundary element grid was taken fixed with mesh size $2^{-6}$ and  polynomials of degree $4$ and $5$ for
   discretizing $\HH^{-1/2}$ and $\HH^{1/2}$ respectively. Due to the non-smooth structure of the
     solution, 
     we observe a large preasymptotic regime, clouding the true asymptotic convergence rate.
       Nevertheless, for small time steps we still observe a high order of convergence. 
\eremk
 \end{example}

 \begin{figure}
   \captionsetup[subfigure]{position=b}
   \caption{Evolution of Example~\ref{ex:complex_geometry}}
   \label{fig:complicated_evolution}
   \subcaptionbox{The incident wave}[0.245\textwidth]{
     \includegraphics[width=4cm]{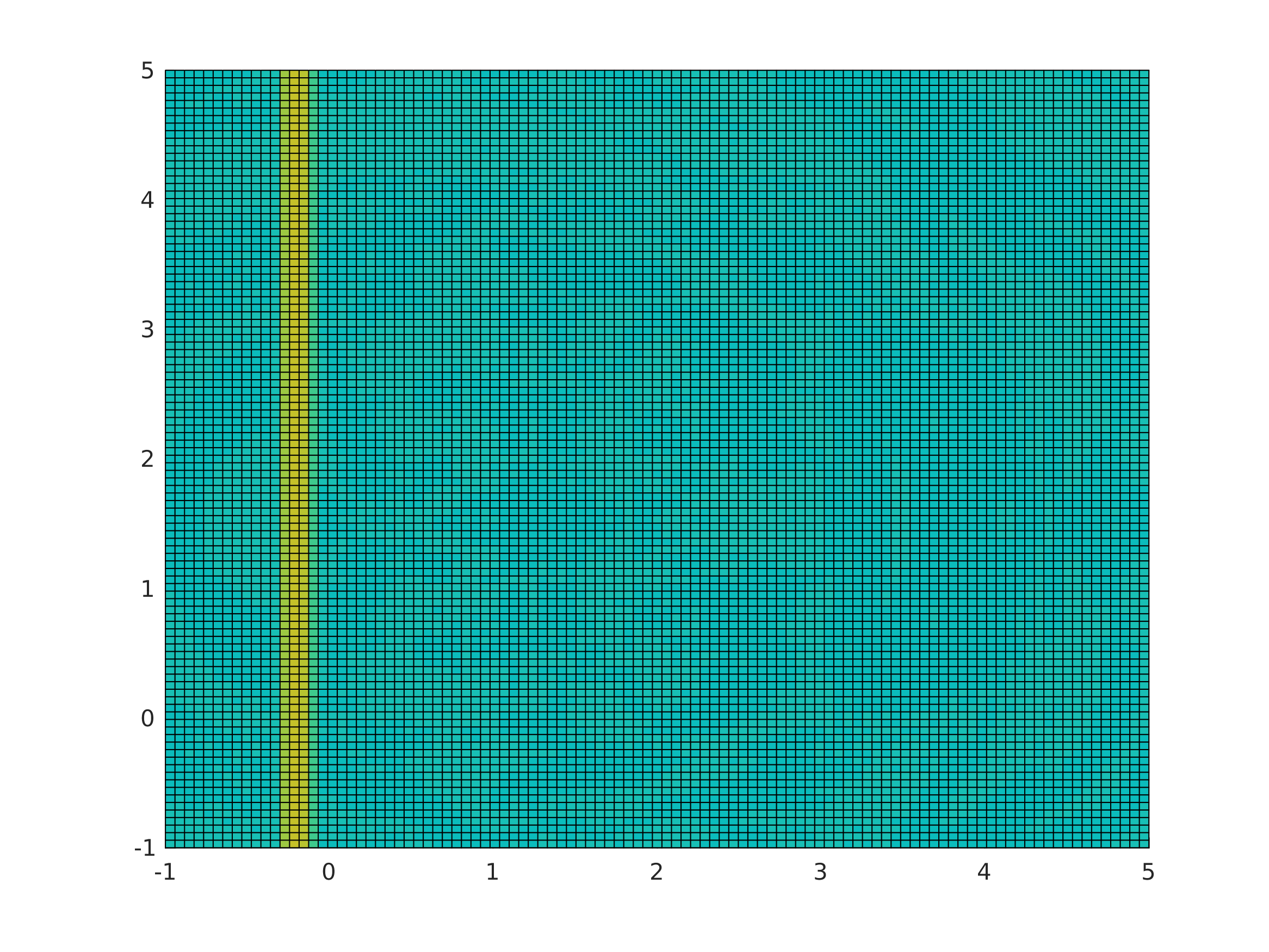}
     \includegraphics[width=4cm]{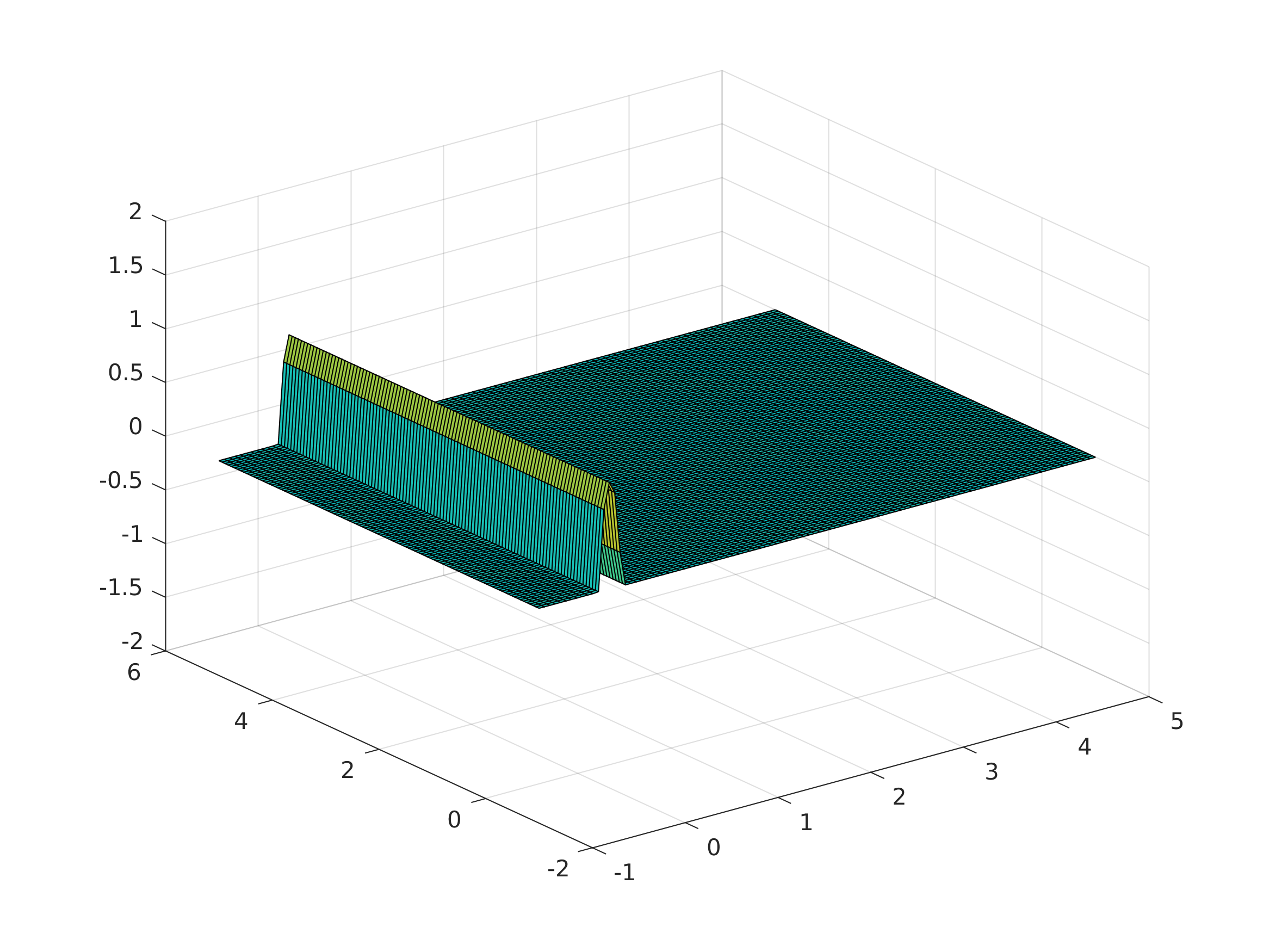}
   }
   \subcaptionbox{The wave hits the checkerboard}[0.245\textwidth]{
     \includegraphics[width=4cm]{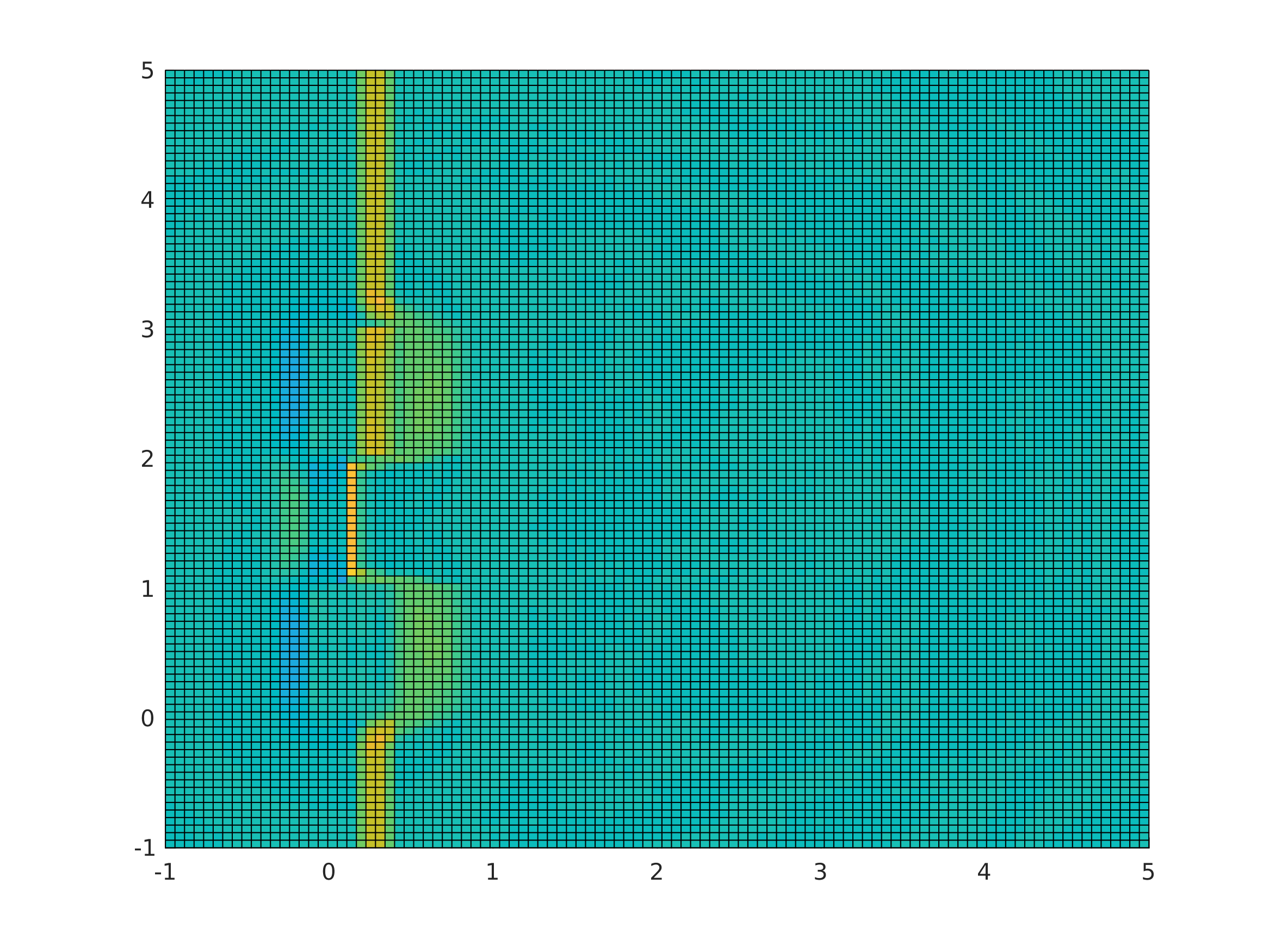}
     \includegraphics[width=4cm]{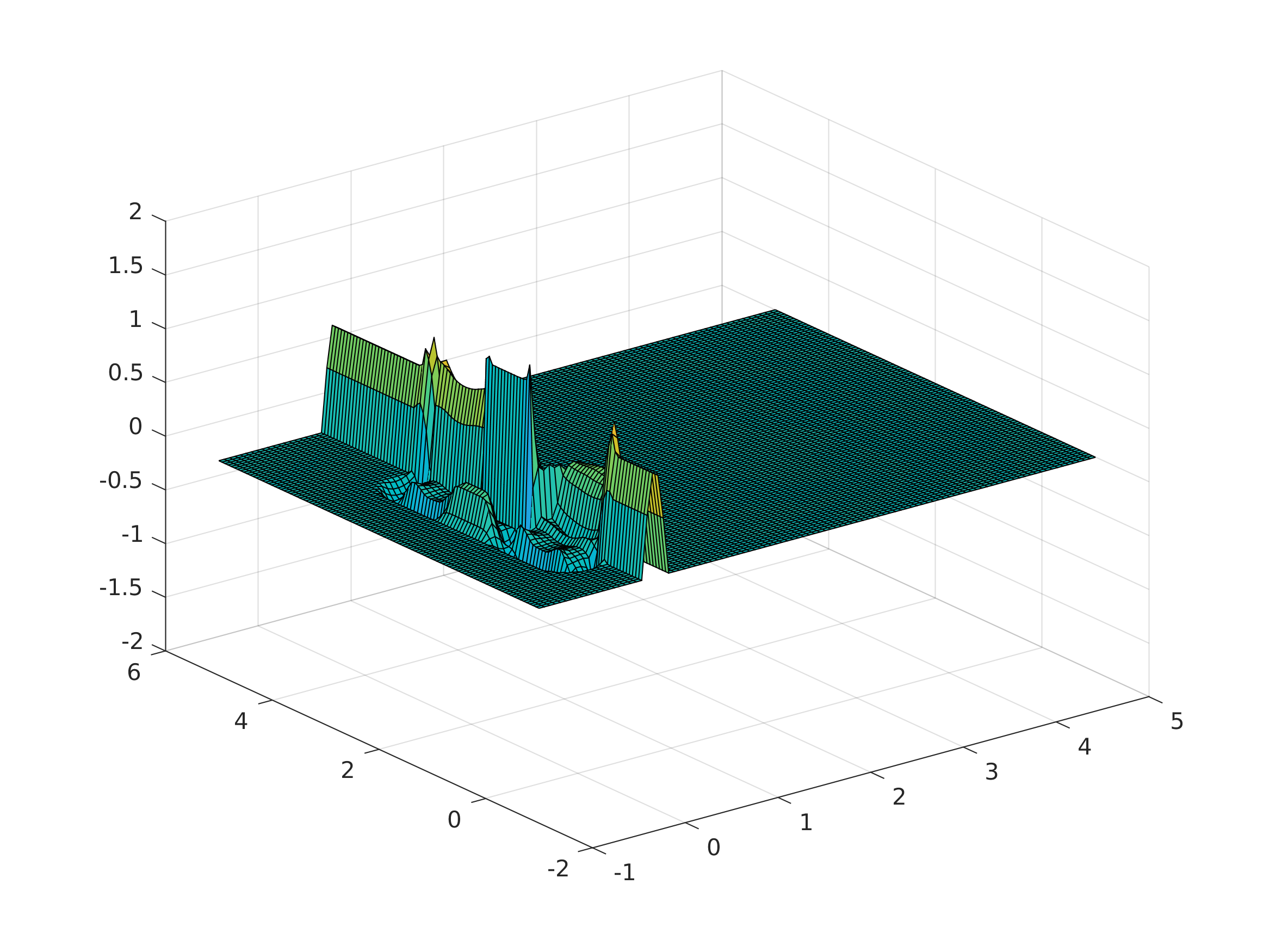}
   }
   \subcaptionbox{Passing through the obstacle}[0.245\textwidth]{
     \includegraphics[width=4cm]{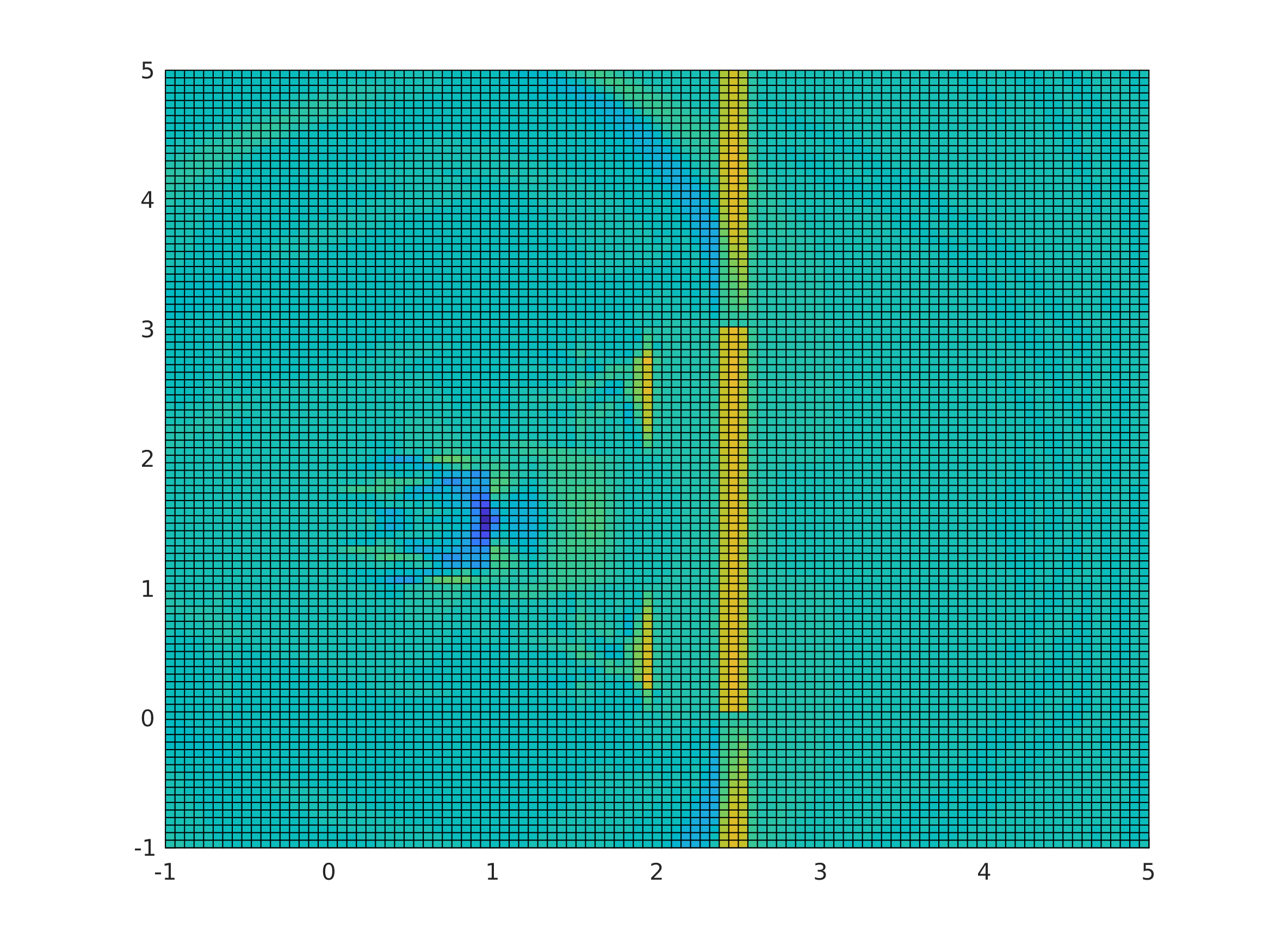}
     \includegraphics[width=4cm]{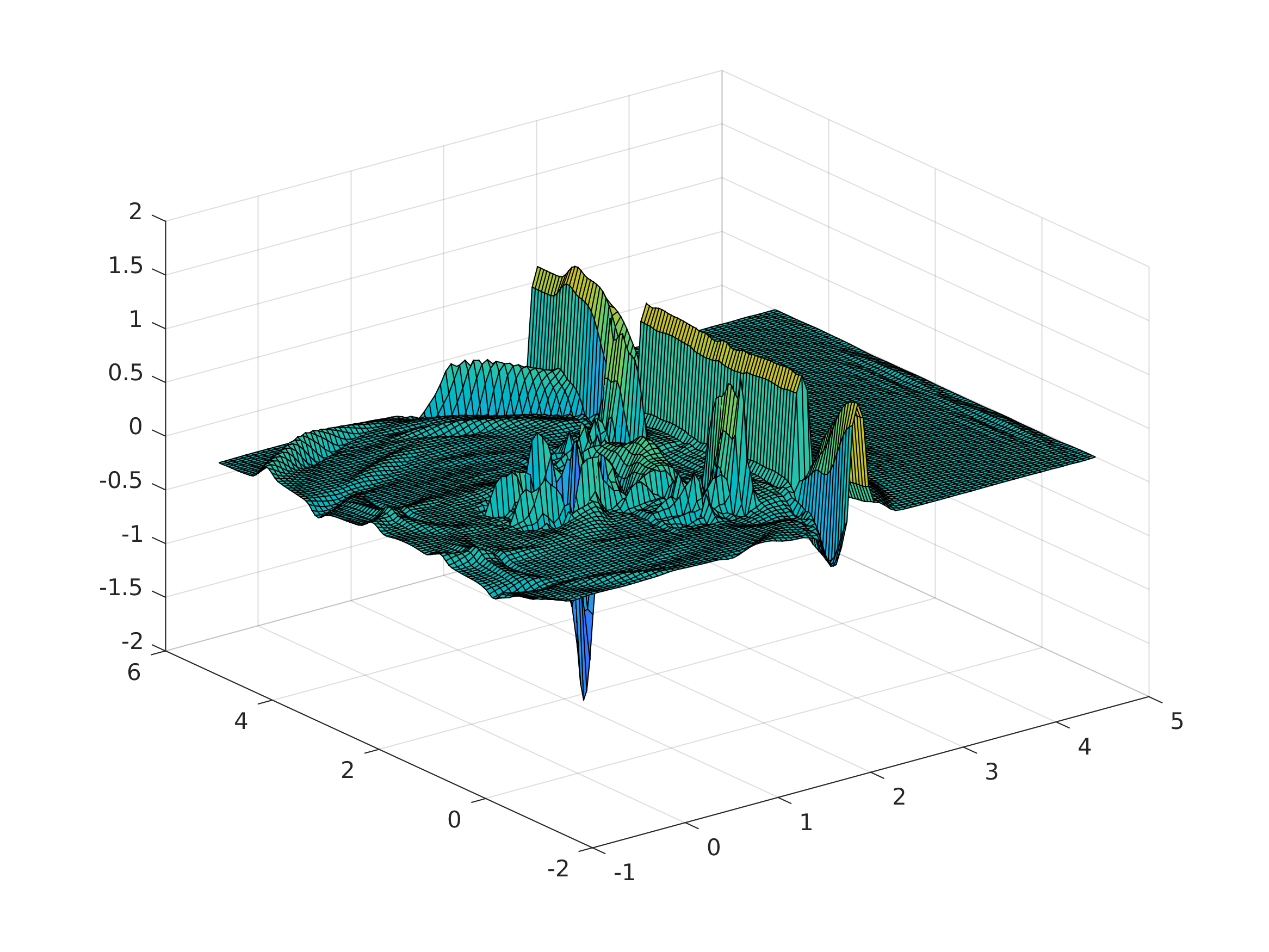}
   }
   \subcaptionbox{The incident wave is past the obstacle}[0.245\textwidth]{
     \includegraphics[width=4cm]{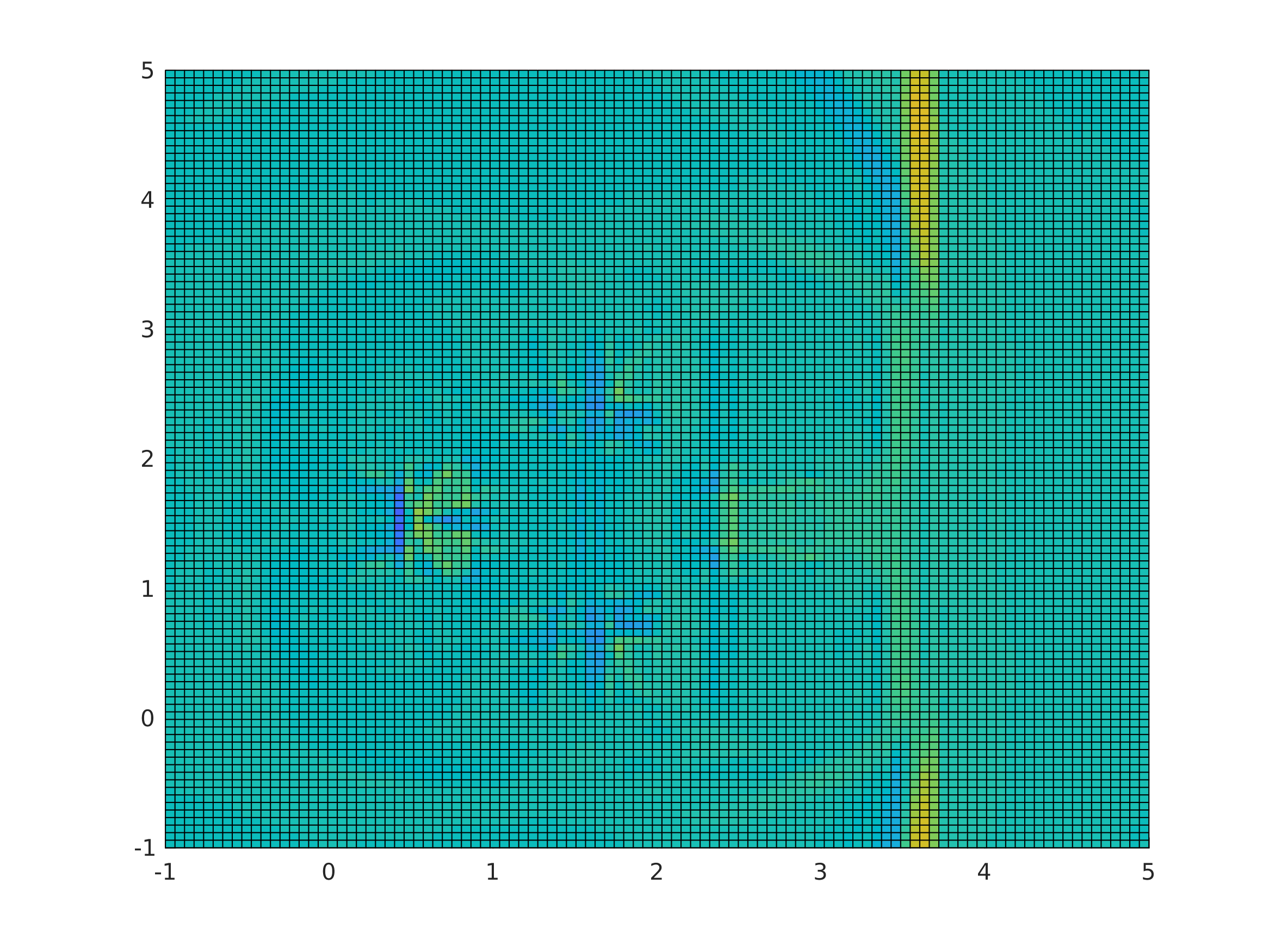}
     \includegraphics[width=4cm]{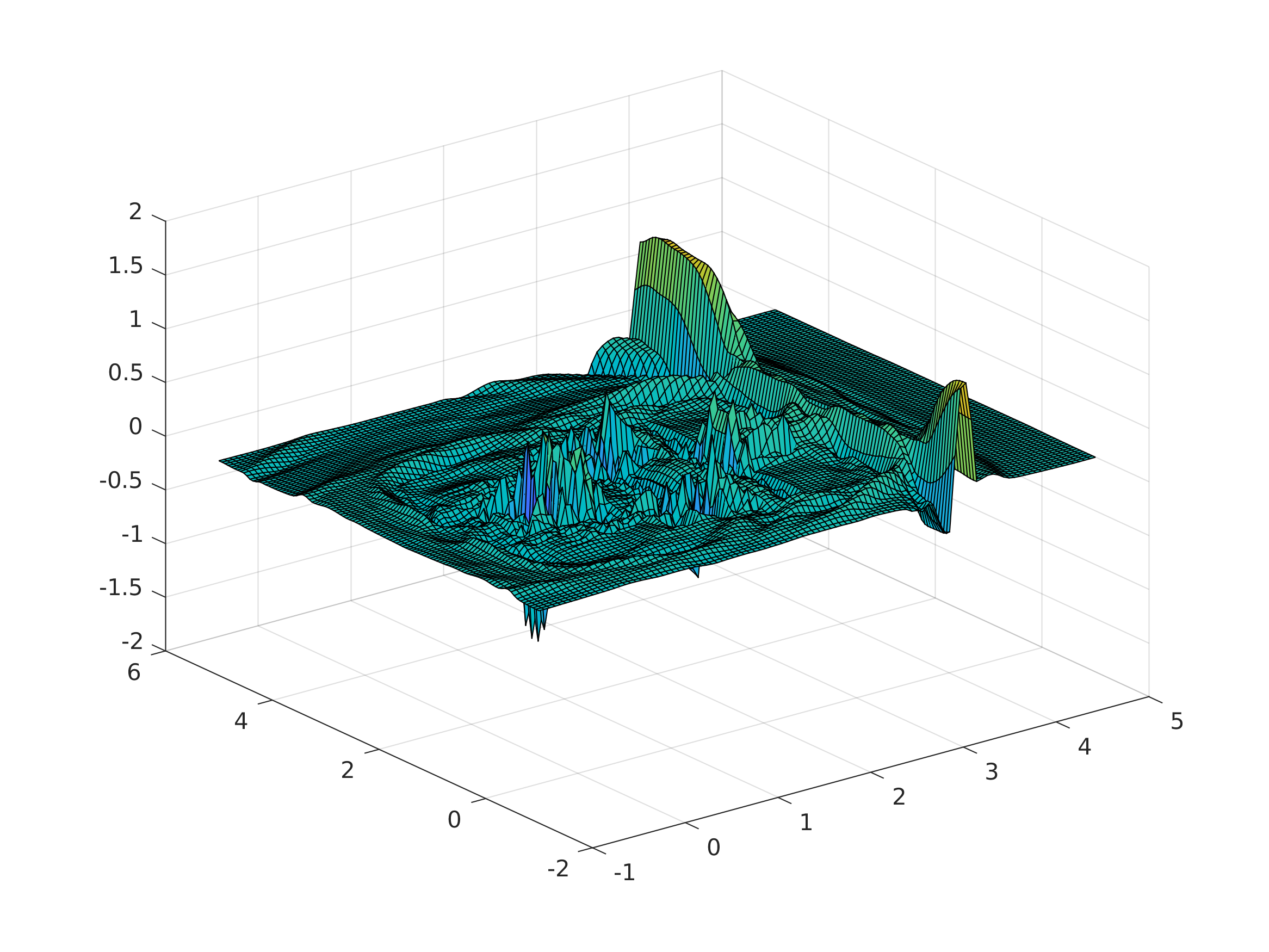}
   }
 \end{figure}

 \begin{figure}
   \centering
   \begin{subfigure}{0.49\textwidth}
   \includeTikzOrEps{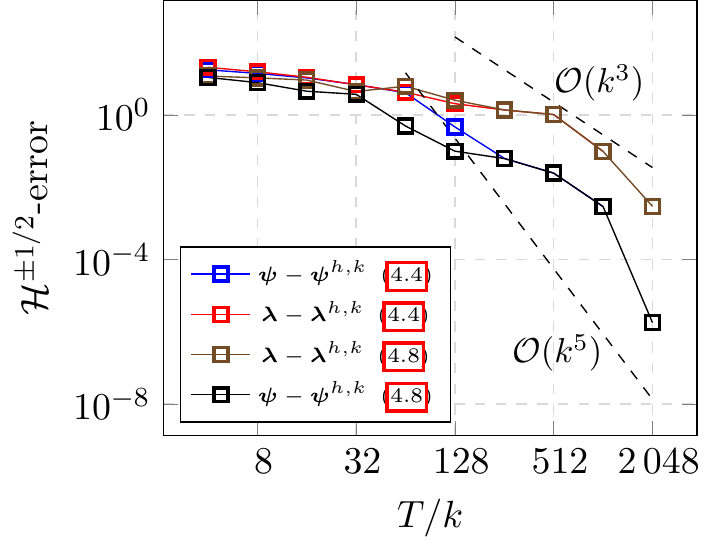}
   \caption{Example~\ref{ex:complex_geometry}}
   \label{fig:complicated_convergence}
 \end{subfigure}
 \begin{subfigure}{0.49\textwidth}
   \centering
   \includeTikzOrEps{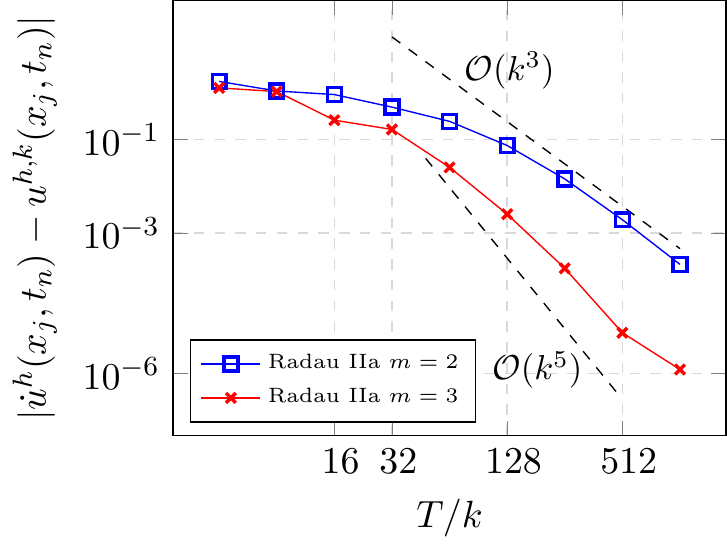}
   \caption{Example~\ref{ex:pointwise}}
   \label{fig:pointwise_convergence}
 \end{subfigure}
 \caption{Convergence history of Examples~\ref{ex:complex_geometry} and \ref{ex:pointwise}.}
\end{figure}

 \begin{example}
   \label{ex:pointwise}
   We look at the pointwise convergence of our method
   as analyzed in Theorem~\ref{thm:local_and_pointwise_convergence}.
   We consider the $2\times2$ checkers grid with wave  numbers $[1,2,0.5,0.5,2]$. The incident wave   
   is given by a slightly smoother version of Example~\ref{ex:complex_geometry}, namely,
    $
    u^{\textrm{inc}}(x,t):=H(x \cdot d -t)
    $
    with $d:=[1,0]^{\top}$ and
   \begin{align*}
     H(x)&:=\sin(t) \varphi(t) \varphi(t-\pi/2), \qquad \text{using} \qquad
    \varphi(x):=\frac{e^{-1/t}}{e^{-\frac{1}{t}}+e^{-\frac{1}{1-t}}} 
   \end{align*}
   for $x \in (0,1)$ and $\varphi(x):=0$ for $x\leq 0$ and $\varphi(x):=1$ for $x\geq 1$.
   We place observation points $x_j$ in the center of each of the squares. We compute the
   difference $|\dot{u}^{h}(x_j,t_n)-\partial_k u^{h,k}(x_j,t_n)|$ and take the maximum over all
   points and timesteps. Since no exact solution was available, we used the numerical
   solution obtained by halving the finest timestep size to obtain an estimate of the
   true error. We compare a 2-stage and a 3-stage Radau IIa method of classical
   orders $3$ and $5$ respectively. We used a fixed mesh of size $h=0.125$ and
   polynomials of order $4$ and $5$ for the discretization spaces.
   
   In Figure~\ref{fig:pointwise_convergence}, we observe that the error goes to zero very rapidly.
   The rates observed are consistent with the full classical orders 
predicted by Theorem~\ref{thm:local_and_pointwise_convergence}, although
   there might still be some preasymptotic effect polluting the $3$-stage computation resulting in
   convergence rates slightly below 5.    \eremk
 \end{example}

\textbf{Acknowledgments:} Financial support by the Austrian Science Fund (FWF) through the 
projects P29197-N32, P33477, W1245 and SFB65~(A.R.)
and project P28367-N35 (J.M.M). FJS is partially supported by NSF-DMS grant 1818867. Part of this work was developed while FJS was a Visiting Professor at the TUW.
\bibliographystyle{alphaabbr}
\bibliography{literature}

\appendix
\section{Local higher order convergence away from the boundary}
\label{sect:local_higher_order}
In this section, we show that Runge-Kutta approximations achieve the full
classical order as long as we stay away from the boundary of the domain. 

We will work in the abstract setting of~\cite{semigroups}. Briefly summarized,
we are given an operator $\AAstar$ on  a Hilbert space $\HH$ and
a second operator $\BB: \dom(\AAstar) \to \mathbb{M}$ such that
$\AA:=\AAstar|_{\ker(\BB)}$ is the generator of a $C_0$-semigroup on $\HH$.
Since we will only work with integer-order norms, we only need
one  Sobolev tower. For $\mu \in \N$ and $u \in \dom(\AAstar^\mu)$ we define the norms
\begin{align*}
  \norm{u}_{\HH_\star^{\mu}}:=\sum_{j=0}^{\mu}{\norm{\AA_\star^{j} u}_{\HH}}.
\end{align*}
We will write $\norm{u}_{\HH^{\mu}}$ for the same norm if we want to emphasize that
$u \in \dom(\AA^\mu)$, i.e., it satisfies the additional side constraints.

The main role will be played by operators correcting the issue of boundary conditions
(or more general side constraints) in the definition of $\AA$ compared to $\AAstar$.
\begin{definition}
  \label{def:cutoff_operators}
  We call an operator $\opT: \HH \to \HH$
  an \emph{admissible cutoff operator} of order $M$,
  if the following holds:
  \begin{enumerate}[(i)]
  \item There exists a constant $n(\opT,M)$ such that
    for $u \in \dom(\AAstar^{n(\opT,M)})$
    \begin{align}
      \label{eq:def_cutoff_operators}
      \opT u  \in \dom(\AA^{M})\, \qquad \text{with}\,\qquad
      \norm{\opT u}_{\HH^{M}}\lesssim \norm{u}_{\HH_{\star}^{n(\opT,M)}}.
    \end{align}
  \item  The commutator $[\opT,\AAstar]:=\opT \AAstar  - \AAstar \opT$ is of lower order than $\AAstar$, and similarly
    for its iterated versions. Namely, using the notation
    \begin{align}
      \label{eq:iterated_commutators}      
      \opC_0:=\opT \qquad \text{and} \qquad \opC_{\ell+1}:=\AA^{-1} [\opC_\ell,\AAstar],
    \end{align}
    we assume that for all $0\leq m$
    there exists $L(\opT,m)>0$ such that $\AA^m \opC_{L(\opT,m)}$ is a bounded  operator
    mapping $\HH \to \HH$.  We write $L(\opT):=L(\opT,0)$.
  \end{enumerate}
\end{definition}

Due to the involvement of $\AA^{-1}$ in the definition of $\opC_{\ell}$, the action of the
commutators is non-trivial to understand. To make it easier to show that
an operator is an admissible cutoff operator, we have the following more straight-forward
representation:
\begin{lemma}
  \label{lemma:iterated_commutators_2}
  Let $\opT: \HH \to \HH$ be a bounded operator. Define a second sequence of
  iterated commutators by
  \begin{align}
    \label{eq:iterated_commutators2}      
    \widetilde{\opC}_0:=\opT \qquad \text{and} \qquad \widetilde{\opC}_{\ell+1}:=[\widetilde{\opC}_\ell,\AAstar].
  \end{align}

  Assume that for $\ell \in \N_0$,
  $\widetilde{\opC}_{j}u \in \dom(\AA)$ for all $j \leq \ell$. Then, the two sequences are related by
  \begin{align}
    \label{eq:iterated_commutators_relations}      
    {\opC_{\ell}} u  &=\AA^{-\ell} \widetilde{\opC}_{\ell} \, u \qquad \forall \ell \in \N_0.
  \end{align}  
\end{lemma}
\begin{proof}
  The base case $\ell=0$ is clear. Assume that that \eqref{eq:iterated_commutators_relations} holds
  for a fixed $\ell$. Then we calculate,
  using the fact that $\widetilde{\opC}_{\ell} u \in \dom(A)$
  and thus $\AA  \AA^{-1} \widetilde{\opC}_{\ell} = \AA^{-1} \AA\widetilde{\opC}_{\ell}$:
  \begin{align*}
    \opC_{\ell+1}
    &=\AA^{-1} [\opC_{\ell},\AAstar]
      =\AA^{-1} \big( \opC_{\ell}\AAstar - \AAstar\opC_{\ell} \big)
      =\AA^{-1} \big(
      \AA^{-\ell} \widetilde{\opC}_{\ell} \AAstar -
      \AA \AA^{-\ell} \widetilde{\opC}_{\ell}
      \big) \\
    &=\AA^{-\ell - 1}\big( \widetilde{\opC}_{\ell} \AA - \AAstar\widetilde{\opC}_{\ell} \big)
      =\AA^{-\ell - 1} [\widetilde{\opC}_{\ell},\AAstar]
      =\AA^{-\ell - 1} \widetilde{\opC}_{\ell+1}. \qedhere
  \end{align*}
\end{proof}

\begin{remark}
  We note that our definition not only allows for ``classical'' multiplication operators
  with a cutoff function, but also for operators involving powers of $\AA$, as
  well as commutators $[\TT,\AAstar]:=\TT \AAstar-\AAstar \TT$. This will be crucial in the later proof.
  
  Similarly, the theory not only covers a multiplication cutoff-operator
  for the operator $\AAstar$ as introduced in Section~\ref{sect:analyzable_form}
  (with $L(\TT,m)=m$ as the commutators are already bounded), 
  but it also applies to, for example, the case $\AAstar=-\laplace$ (with $L(\TT,m)=2m$).
\end{remark}

  Before we begin, we recall the definition $\breve{\opT}:=\operatorname{diag}(\opT,\dots,\opT)$,
    which applies $\opT$ to each stage of the RK-method. We also use the notation of
    Kronecker products, i.e., 
    for a matrix $\mathcal S\in \mathbb R^{m\times m}$ and an operator $C:\mathcal Y\to \mathcal Z$ we write 
    \[
      \mathcal S\otimes C:=
      \left[\begin{array}{ccc} 
              \mathcal S_{11} C & \cdots & \mathcal S_{1m} C \\
              \vdots & & \vdots \\
              \mathcal S_{m1} C & \cdots & \mathcal S_{mm} C
            \end{array}\right] : \mathcal Y^m \to \mathcal Z^m.
        \]
    We will use the vector $\ones:=(1,\dots,1)^{\top}$ as an operator in the sense
    of $\ones u := (u,\dots,u)^{\top}$.

  \begin{lemma}
    Let $w \in \mathcal{C}^1(\R,\HH)\cap \mathcal{C}^0(\R,\dom(\AAstar))$ with $w(t)=0$ for $t\leq 0$.
    We define the continuous RK-error as
    $$
    e_k(w,t):=\sum_{j=0}^{\infty}
    {r(k \AA)^{j}
      \big( v(t-(j-1)k) +
      k\rkb^{\top} \otimes \AA (\id-k\rkA\otimes \AA)^{-1} V(t-jk)\big)}
    $$
    with the consistency error functions
    \begin{subequations}
      \label{eq:def_consitstency_error_functions}
    \begin{align}
      v(t)&:=u(t)- u(t-k) - k \rkb^{\top} \dot{u}(t-k+k\rkc) \qquad\text{and}\\
      V(t)&:=u(t+k\rkc)-u(t)\ones -k \rkA\dot{u}(t+k\rkc).
    \end{align}
    \end{subequations}

    Then, the following identity holds for all $0<t_1<t_2$:
    \begin{align}
    e_k(w,t_1)-e_k(w,t_2)&=
      \int_{t_1}^{t_2}{\dot{e}_k(w,\tau)\,d\tau}
      =\int_{t_1}^{t_2}{e_k(\dot{w},\tau)\,d\tau}.
    \label{eq:telescoping_rk}
    \end{align}

    If $u(t)$  denotes the exact solution to $\dot{u}=\AAstar u$
    and $(u^n)_{n\in \N}$, $(U^n)_{n\in \N}$  is the sequence of RK-approximations
    and stage vectors,
    such that $u(t_n +k \rkc)-U^n \in [\dom(A)]^m$. Then the error can be written as
    \begin{align*}
      u(t_n)-u^n=e_k(u,t_n).
    \end{align*}
  \end{lemma}
  \begin{proof}
    The fact that $e_k$ corresponds to the Runge-Kutta error
    follows from the proof of \cite[Theorem 1]{mallo_palencia_optimal_orders_rk}.  
    \eqref{eq:telescoping_rk} follows from the definition of $e_k$, the linearity
    of all the operators involved and the fundamental theorem of calculus.
  \end{proof}
\begin{lemma}
  \label{lemma:commutators}
  Let $\Rbar:=\big(\id - k \rkA \otimes \AA\big)^{-1}$ and $\opT$ an admissible cutoff operator.
  Define the commutator $[\Rbar,\breve{\opT}]:=\Rbar\breve{\opT} - \breve{\opT} \Rbar$. Then the following identity holds:
  \begin{align*}
    [\Rbar,\breve{\opT}]&=k\Rbar\big(\rkA \otimes [ {\AA}, \breve{\opT}] \big) \Rbar.
  \end{align*}
  We can iterate this identity using the iterated commutators $\opC_j$, as defined in
  \eqref{eq:iterated_commutators}.  For $L \in \N$
  the following expression is valid:
  \begin{align}
    \label{eq:iteratec_commtuators_regularity_version}
    [\Rbar,\breve{\opT}]&=\sum_{j=1}^{L}{\Rbar \big(\Rbar-\id\big)^j \opC_{j}}
                  +  (\Rbar-\id)^{L+1} \opC_{L+1} \Rbar.
  \end{align}
  
\end{lemma}
\begin{proof}
  For $U \in [\dom(\AA)]^m$ we calculate:
  \begin{align*}
    (\id - k \rkA \otimes \AA) \breve{\opT} U
    &= \breve{\opT}(\id - k \rkA \otimes \AA)  U    - [k \rkA\otimes \AA, \breve{\opT}] U.
  \end{align*}
  We apply $\Rbar$ to both sides of the  equation to get:
  \begin{align*}
    \breve{\opT} U
    &= \Rbar \breve{\opT}(\id - k \rkA \otimes \AA)  U    - \Rbar [k \rkA\otimes \AA, \breve{\opT}] U.
  \end{align*}
  Choosing $U:=\Rbar V$ for arbitrary $V \in \HH^m$ then gives the stated result:
    \begin{align*}
      \breve{\opT} \Rbar V
    &= R_k \breve{\opT} V    - \Rbar [k \rkA\otimes \AA, \breve{\opT}] \Rbar V.
    \end{align*}

    To see \eqref{eq:iteratec_commtuators_regularity_version},
    we  use the identity $k\rkA\otimes \AA \Rbar U=\Rbar U - U$.
    This gives
      \begin{align*}
        [\Rbar,\breve{\opT}]
        &=(\Rbar - \id) \breve{\opC}_1 \Rbar 
          = (\Rbar - \id) \Rbar \breve{\opC}_1  +(\Rbar - \id) [\Rbar,\breve{\opC}_1] \\
        &=(\Rbar - \id) \Rbar \breve{\opC}_1  +(\Rbar - \id)  \Rbar 
          \big(k\rkA \otimes
          [\breve{\AA},\breve{\opC}_1] \big) \Rbar \\
        &=(\Rbar - \id) \Rbar \breve{\opC}_1  +(\Rbar - \id)^2 \breve{\opC}_2 \Rbar
      \end{align*}
      We observe that we again have an operator $\Rbar$ at the end of the right hand side.
      Shifting it to the left of $\opC_2$ we can repeat the previous argument. Iterating
      this procedure  we get the stated result.
      
  \end{proof}
  
  Next we need to analyze some rational functions related to the RK-method. 
  \begin{lemma}
    \label{lemma:rk_rational_function_bound}
    Given parameters $0\leq \ell \leq p$ , $0\leq \beta<p$, and  $0\leq j \leq p$,
    define the rational functions:
    \begin{align*}
      r_{\ell,\beta,j}:=z \rkb^{\top} (\id - z \rkA)^{-j} \rkA^{\beta}(\rkc^{\ell} - \ell \rkA \rkc^{\ell-1}).     
    \end{align*}

    Then $r_{\ell,\beta,j}=\bigO(z^{p+1-\ell-\beta})$ as $z\to 0$. The implied constant depends on
    the Runge-Kutta method.
  \end{lemma}
  \begin{proof}
    We expand $(\id - z \rkA)^{-1}$ into its Neumann series and multiply the power series
    $j$-times, collecting the leading terms to get:
    \begin{align*}
      r_{\ell,\beta,j}(z)=
      \sum_{n=0}^{p}{c_n z^{n+1} \rkb^{\top}\rkA^{n+\beta}(\rkc^{\ell} - \ell \rkA \rkc^{\ell-1})}
      + \bigO(z^{p+2})
    \end{align*}
    for some coefficients $c_n$, depending on $j$.
    By the order conditions (see, e.g., \cite[(Eq. 5.1)]{semigroups})
    it holds that
    $$
    \rkb^{\top}\rkA^{n+\beta}(\rkc^{\ell} - \ell \rkA \rkc^{\ell-1}) = 0 \qquad \forall\; 0\leq n+\beta \leq p-\ell -1.
    $$
    Thus, the leading non-vanishing term is of order $\bigO(z^{p+1-\ell -\beta})$.
  \end{proof}

  Next, we show that in a single step, we can achieve full classical order.
  \begin{lemma}
    \label{lemma:cutoff_error_single_step}
      Let $u$ solve $\dot{u}=\AAstar u + f$.
      Fix $t >0$ and let $\widetilde{u}$ be  the
      one-step Runge-Kutta approximation:
      \begin{align*}
        \widetilde{U}&=u(t)\ones + k \rkA\otimes \AAstar \widetilde{U} + k \rkA F(t+k \rkc),
                       \qquad \big( \widetilde{U}-u(t+k\rkc)\big) \in [\dom(\AA)]^m
        \\
        \widetilde{u}&=u(t)+ k \rkb^{\top}\otimes \AAstar \widetilde{U}
                            +  k \rkb^{\top} F(t + k \rkc).
      \end{align*}

    Let $\opT$ be an admissible cutoff operator of order $M\leq p-q$. Then the following estimate holds:
    \begin{align*}
      \norm{\opT(u(t+k)- \widetilde{u})}_{\HH}\lesssim  k^{q+1+M}      
      \sum_{\ell=q+1}^{p+1}{\sum_{j=0}^{M}{\max_{t\leq \tau\leq t+k}{\big\|\opC_{j+1} u^{(\ell)}(\tau)\big\|_{\HH^{L-\ell}}}}}.     
    \end{align*}
  \end{lemma}

  \begin{proof}
    For simplicity of notation, we only consider the ``full regularity'' case $M=p-q$.
    The general case follows along the same lines but replacing $p$ with $M$ in
    some places.
    
    Setting $E:=u(t+k\rkc)-\widetilde{U}$, solves using the commutator notation
    \begin{align}
      \breve{\opT} E&=k \rkA\otimes \AA \breve{\opT} E + [k \rkA \otimes \AA,\breve{\opT}]E
              + \big(\breve{\opT} u(t + k \rkc) - \opT u(t)\ones - \rkA \otimes \opT\dot{u}(t + k \rkc)\big) \nonumber\\
      &=:k \rkA\otimes \AA \breve{\opT} E + [k \rkA \otimes \AA,\breve{\opT}]E 
        + \breve{\opT} V(t) \nonumber\\
      &=\Rbar \big([k \rkA \otimes \AA,\breve{\opT}]E + \breve{\opT} V(t)\big). \label{eq:opTe_via_Rbar}
    \end{align}
    We note that, using the operator $\Rbar$ from Lemma~\ref{lemma:commutators}
    \begin{align*}
      E=u(t+k\rkc)-\widetilde{U}&=\Rbar\big(u(t + k \rkc) - u(t)\ones - \rkA \otimes \dot{U}(t + k \rkc)\big)
      =\Rbar V(t).
    \end{align*}

    For the full step, we get for $e:=u(t+k)-\widetilde{u}$
    \begin{align*}
     \opT e&:=\opT u(t+k)-\opT \widetilde{u}\\
      &= k\rkb^{\top}\otimes \AA \breve{\opT} E + k \rkb^{\top}\otimes[\AA,\opT]E
        + \big(\opT u(t+k)- \opT u(t) - k\rkb^{\top} \otimes \opT\dot{u}(t+k \rkc)\big)  \\
             &\stackrel{\mathclap{\eqref{eq:opTe_via_Rbar}}}{=} \;k\rkb^{\top} \otimes \Rbar [ \breve{\AA},\breve{\opT}] E
               + k(\rkb^{\top}\otimes \AA) \Rbar \breve{\opT} V(t)
               + \opT v(t),
    \end{align*}
    where $v(t)$ is defined in \eqref{eq:def_consitstency_error_functions}.
    The last two terms are standard-consistency terms of order $\bigO(k^{p+1})$,
    see, e.g., \cite{mallo_palencia_optimal_orders_rk} for their treatment
    (where we use that $\opT u(t) \in \dom(\AA^L)$).

    We focus on the following expression
    \begin{align*}      
      w(t):=k \rkb^{T}(\id - k \rkA \otimes \AA)^{-1}[\breve{\AA},\breve{\opT}]\Rbar V(t).
    \end{align*}
    We use Lemma~\ref{lemma:commutators} to rewrite this as:
    \begin{align*}
      w(t)
      &= k \rkb^{T}\otimes \Rbar [\breve{\AA},\breve{\opT}] \Rbar V(t) \\
      &=k \rkb^{T}\otimes \Rbar^2 [\breve{\AA},\breve{\opT}]V(t)
        - k\rkb^{\top}\otimes \Rbar^2 [k \rkA \otimes \AA,[\breve{\AA},\breve{\opT}]]\Rbar V(t) \\
      &= \sum_{j=0}^{p-q-1}{ k \rkb^{T}\otimes \AA\Rbar^{j+2} (k^{j} \rkA^{j})\breve{\AA}^{j}\breve{\opC}_{j+1} V(t)}
      \\
      &\qquad \qquad  + k\rkb^{\top} \otimes \AA \Rbar^{p-q+1} (k^{p-q} \rkA^{p-q})\breve{\AA}^{p-q}\breve{\opC}_{p-q+1} \Rbar V(t).
    \end{align*}
    Expanding $V(t)$ into its Taylor series and using the order conditions    
    $\rkc^{\ell} - \ell \rkA\rkc^{\ell-1}=0$ for $0\leq \ell \leq q$  (see, e.g., \cite[(Eq. 5.1)]{semigroups}),
    we can write
    \begin{align*}
      V(t)=\sum_{\ell=q+1}^{p}{k^{\ell}(c^{\ell} - \ell \rkA c^{\ell-1}) u^{(\ell)}(t_n)} + \varphi(t)
    \end{align*}
    with $\norm{\varphi(t)}_{\HH}\lesssim k^{p+1}$ for $t\in (t_n,t_{n+1})$.
    By using the rational functions $r_{\ell,\beta,j}$ we can rewrite this as
    \begin{align*}
      w(t)
      &= \sum_{j=0}^{p-q-1}{\sum_{\ell=q+1}^{p}{k^{j+\ell} r_{\ell,j,j}(k\AA)\AA^{j}\opC_{j+1} u^{(\ell)}(t)}}
        + \bigO(k^{p+1}).
    \end{align*}
    By Lemma~\ref{lemma:rk_rational_function_bound}, we can
    factor the rational function as $r_{\ell,j,j}(k\AA)=g(k\AA)(k \AA)^{p+1-j-\ell}$ .
    This allows us to bound the term by
    \begin{align*}
      \norm{w(t)}_{\HH}
      &\lesssim k^{p+1} \sum_{\ell=q+1}^{p}{\sum_{j=0}^{p-q}{ 
        \norm{\opC_j u^{(\ell)}(t)}_{\HH^{p-\ell}}}}.
    \end{align*}    
    Thus, combining this with a standard estimate for the original truncation error
    we get the stated result.
\end{proof}

\begin{theorem}
  \label{thm:local_convergence}
    Let $u$ solve $\dot{u}=\AAstar u + f$ and let $u_n$ be  the Runge-Kutta approximation:
    \begin{align*}
      U^n&=u_n\ones + k \rkA\otimes \AAstar U^n + k \rkA  F(t_n+k \rkc),
\quad \text{with}\quad  \big(U^n - u(t_n + k \rkc)\big) \in [\dom(\AA)]^m\\
      u_{n+1}&=u_n + k \rkb^{\top} \otimes \AAstar U^n +  k \rkb^{\top} F(t_n + k \rkc).
    \end{align*}    
    Assume that the spectrum of $k\AA$ is contained in a complex sector
    of opening angle $\omega$ and $\sigma(k\AA)$ is disjoint to the set
    \begin{align*}
      Z_{\omega,\delta}&:=
      \big\{ z \in \C: |\operatorname{Arg}{z}|\leq \omega,
      \quad \operatorname{dist}(z,w)\leq\delta, \; \forall w\neq 0,\text{ with}\; r(w)=1\big\},
    \end{align*}
    i.e., $\sigma(k \AA)\cap Z_{\omega,\delta}=\emptyset$ or, in other words, the spectrum avoids all points
    in a complex sector where $r(w)=1$.

    Let $\opT$ be an admissible cutoff operator of order $M\leq p-q$. Then the following estimate holds:
    \begin{align*}
      \norm{\opT(u(t_n)-u_n)}_{\HH}\lesssim
      (1+\rho_k(T)T)^{L(\opT,M)+1}k^{q+M}
      \sum_{j=0}^{\substack{M+\\L(\opT,M)}} \sum_{\ell=0}^{\substack{p+1+\\L(\opT,M)}}
      {
      \max_{0\leq \tau \leq T}\norm{\opC_{j} u^{(\ell)}(\tau)}_{\HH}},
    \end{align*}
    provided that $u$ is sufficiently smooth such that the right-hand side is finite.
  \end{theorem}
  \begin{proof}
    Throughout this proof, we will use different modifications of
    the operator $\TT$. In order to not get confused, we will denote the
    ``original'' cutoff operator with $\TT_0$ and write $\TT$ for a
    generic cutoff operator which will change role several times throughout.
    
    Since the construction is a bit lengthy and technical, we briefly 
    outline the proof. It consists of four phases:
    \begin{enumerate}
      \item We derive a recurrence formula for the post-processed error, allowing us to bound the error by
        terms involving higher-order commutators and consistency terms. 
     \item 
       Using the fact that commutators are also viable cutoff operators,
       and the fact that high-order commutators are bounded operators,
       we can perform an induction argument to bound
       the error of the postprocessed solution by the error of the original 
       RK-approximation.
    \item 
      Next, we observe that $\AA^{M} \TT_0$ is also a viable cutoff operator,
      thus the previous step shows that the ``differentiated approximation'' has
      the same order as the original.
    \item
      Finally, we consider the integrated semigroup $v(t):=\AA^{-M}u(t)$, which is known to 
        have convergence order $\bigO(k^{\min(q+M,p)})$. This gives the stated result
        since $\TT_0$ can be represented using operators of the
        form $\AA^{M}\opC_{\ell}\AA^{-M}$.
      \end{enumerate}
      
    \emph{Step 1:}
    Set $\theta(t):=\widetilde{u}(t)-u(t)$.
    We use the continuous notation for the Runge-Kutta error.
    $e_k(t)$ solves (note that in order to get continuous time
    we can just start the RK-iteration at points $t \in (-k,0]$ with zero initial condition).
    \begin{align*}
      e_k(t+k)
      &= r(k\AA) \opT e_k(t) +  \theta(t).
    \end{align*}
    Using the representation $r(z)=r(\infty) + k \rkb^{\top} \rkA^{-1} (\id-z\rkA)^{-1}\ones$
    and the commutator representation in Lemma~\ref{lemma:commutators},
    we get
    \begin{align*}
      \opT e_k(t+k)
 &= r(k\AA) \opT e_k(t)
        + [r(k\AA),\opT] e_k(t)
        + \opT \theta(t) \\
      &= r(k\AA) \opT e_k(t)
        + k (\rkb^{\top} \otimes \AA) \Rbar \big(\breve{\AA}^{-1} [\breve{\AA},\breve{\opT}]\big)
        \Rbar e_k(t) \ones
        + \opT \theta(t).
    \end{align*}
    Using the iterated commutator and Lemma~\ref{lemma:commutators}
    this becomes for  $L \in \N_0$ to be fixed later:
    \begin{align*}
      \opT e_k(t+k)
      &= r(k\rkA) \opT e_k(t) +        
        \sum_{j=1}^{L}{q_j(k \rkA) \opC_{j} e_k(t)} \\
      &\quad +  k\rkb^{\top}\otimes \AA \Rbar(\Rbar-I)^{L+1} \breve\opC_{L+1} \Rbar e_k(t)\ones
        + \opT \theta(t), \nonumber
    \end{align*}
    with $q_j(z):=z \rkb^{\top}(\id-z \rkA)^{-1}\big((\id-z\rkA)^{-1}-\id)^j\ones$
    for $j\geq 1$.
   
    Since everything is linear, we can investigate each contribution on its own.
    First, consider the recursion corresponding to one of the factors $q_j(k\AA)$:
    \begin{align*}
      y^j_{n+1}&=r(k\rkA) y^j_{n} + q_j(k\AA) \opC_j e(t_n)
    \end{align*}
    Expansion of the recursion and summation by parts gives
    \begin{align*}
      y^j_{n+1}&=               
               \sum_{n'=0}^{n}{r(k\AA)^{n'} q_j(k \AA) \opC_j  e(t_n -t_{n'})} \\
             &=s_{n+1}(k \AA) e_0  + \sum_{n'=0}^{n}{s_{n-n',j}(k\AA) \opC_j(e_k(t_n)-e_k(t_n-k))}
               + \opC_j e_k(t_n),
    \end{align*}
    using the rational functions $s_{n,j}(z):=q_j(z)\sum_{j=0}^n{r(z)^n}$.    
    By the explicit formula for geometric sums, we can rewrite this as
    \begin{align*}
      s_{n,j}(z)&=(1-r(z)^{n+1})\frac{q_j(z)}{1-r(z)}.
    \end{align*}
    By A-stability, the term $r(z)^{n+1}$ is uniformly bounded on the right-half plane.
    The function $1-r(z)$ has a simple root at $0$. Since 
    $q_j(0)$ also vanishes , we get that $s_{n,j}$ has no pole at $z=0$. For $z=\infty$, 
     we note that $z(\id-z \rkA)^{-1}\to -\rkA^{-1}$. Thus, we get that
      (up to signs),  both $1-r(z)$ and $q_{j}(z)$ converge to $\rkb^{\top} \rkA^{-1} \ones$ for $z\to \infty$ and any $j\geq 1$. 
    Overall, this means that  the only possible poles are in the set $Z_{\omega,\delta}$.
    Since we assumed $\sigma(k \AA) \cap Z_{\omega,\delta}=\emptyset$, we get
    the operator bound $\norm{s_n(k\AA)}_{\HH\to\HH}\lesssim \rho_k(T)$. 

    This gives, since $y^j_0=0$:
    \begin{align*}
      \norm{y^j_{n+1}}_{\HH}
      &
        \lesssim \rho_k(T)
        \Big(\sum_{n'=0}^{n}{\big\|\opC_j  (e(t_{n'})-e(t_{n'}-k))\big\|_{\HH}}
        + \big\|\opC_j e(t_{n'})\|_{\HH}\Big).
    \end{align*}
    Using \eqref{eq:telescoping_rk},
    we can turn the difference into an integral.
    In addition, we observe that we can use $\opC_{\ell}$ instead of $\opT$,
    only shifting the index of the commutators by $\ell$.    
    The overall estimate becomes:    
    \begin{align*}
      \|\opC_{\ell} e_k(t+k)\|_{\HH}
      &\lesssim
        \rho_k(t+k)
        \sum_{j=1}^{L}{\Big(\int_{\tau=0}^{t}\big\|\opC_{\ell+j}  \dot{e}_k(\tau)\big\|_{\HH}\,d\tau
        + \big\|\opC_{\ell+j} e(t)\|_{\HH}\Big)} \\
      & \quad+ \rho_k(T)k \sum_{n'=0}^{n}{\Big(\|\breve{\opC}_{\ell+L+1} (\Rbar-\id)  e_k(t-n' k) \ones\|_{\HH}
        + \|\opC_{\ell} \varphi(t- n' k)\|_{\HH}\Big)} .
    \end{align*}    

    Since $\dot{e}_k(u,t)$  is just $e_k(\dot{u},t)$ (and analogously for higher derivatives),
    we can also apply the same estimate
    for its derivatives and get for arbitrary $\nu \in \N_0$
    (to simplify notation, we stop explicitly tracking the term due to $\varphi$,
    which is small via Lemma~\ref{lemma:cutoff_error_single_step}):    
    \begin{align}
      \label{eq:tmp1}
      \|\opC_{\ell} e^{(\nu)}_k(t+k)\|_{\HH}
      &\lesssim
        \rho_k(t+k)
        \sum_{j=1}^{L}{\Big(\int_{\tau=0}^{t}\big\|\opC_{\ell+j}  {e}^{(\nu+1)}_k(\tau)\big\|_{\HH}\,d\tau
        + \big\|\opC_{\ell+j} e^{(\nu)}(t)\|_{\HH}\Big) } \nonumber \\
      &\quad+ \rho_k(T)k \sum_{m=0}^{n}{\|\breve{\opC}_{\ell+L+1} (\Rbar-\id)  e^{(\nu)}_k(t-m k) \ones\|_{\HH}} 
        + \bigO(k^{q+M}).
    \end{align}
    which shows~\eqref{eq:the_big_induction_formula}.

    \emph{Step 2:}
    We will prove the following statement by  reverse induction with respect to
    $\ell$ and $\nu$:
    If $\nu\leq \ell \leq L(\opT)$, then:
    \begin{align}
      \label{eq:the_big_induction_formula}
      \norm{\opC_{\ell} e^{(\nu)}_k(t)}_{\HH}
      \lesssim C \big (\rho_k(T)(1+T)\big)^{L(\opT)-\ell}
      \sum_{\nu'=\nu}^{L(\opT)}{
      \norm{e^{(\nu')}(\tau)}_{\HH}}
      + \bigO(k^{q+M})
    \end{align}

    For $\ell=L(\opT)$ and $\nu\leq L(\opT)$ the estimate follows trivially,
    because $\opC_{L(\opT)}$ is a bounded linear operator.
    Thus assume that~\eqref{eq:the_big_induction_formula} holds for
    all $\ell'>\ell$ and $\nu \leq \ell'$.
    Then we use \eqref{eq:tmp1} with $L=L(\opT)-\ell-1$ to get:
    \begin{align*}
      \|\opC_{\ell} e^{(\nu)}_k(t+k)\|_{\HH}
      &\lesssim
        \rho_k(t)
        T \Bigg(\sum_{j=1}^{L(\opT)-\ell-1}{\max_{\tau}\big\|\opC_{\ell+j}  {e}^{(\nu+1)}_k(\tau)\big\|_{\HH}\,d\tau}
        + \big\|\opC_{\ell+j} e^{(\nu)}(t)\|_{\HH}\Bigg)  \\
      &\quad+ \rho_k(T)k \sum_{m=0}^{n}{\|\opC_{L(\opT)} (\Rbar-\id)  e^{(\nu)}_k(t-m k) \ones\|_{\HH}}
        + \bigO(k^{q+M}) \\
      &\stackrel{\eqref{eq:the_big_induction_formula}}{\lesssim}
        \big(\rho_k(t)
        (1+T)\big)^{L(\opT)-\ell} \sum_{\nu'=\ell}^{L(\opT)}{ \max_{\tau}{\norm{e^{(\nu')}(\tau)}_{\HH}}}      
         +\bigO(k^{q+M}).
    \end{align*}

    \emph{Step 3:}
    Since the previous estimate was valid for arbitrary orders $L(\TT)$,
    we can apply this to the case $\opT=A^{M} \opC_\ell$
    (where $\opC_\ell$ is defined using the operator $\opC_0:=\opT_0$ as the basis)
    and $\ell \in \N_0$. This gives
    \begin{align}
      \label{eq:diff_of_cutoff_semigroups}
      \|\AA^M \opC_{\ell} e_k(t_n)\|_{\HH}
      &\lesssim
        C (\rho_k(T)(1+T))^{L(\opT_0,M)-\ell+1} \!\!
        \sum_{\nu'=0}^{L(\opT_0,M)-\ell}{ \!\!
        \big\|{e^{(\nu')}(\tau)}\big\|_{\HH}}
        + \bigO(k^{q+M}).
    \end{align}
    By the theory of \cite{mallo_palencia_optimal_orders_rk} we get that (at least)
    $$
    \|e^{(\nu)}_k(t_n)\|_{\HH}\lesssim
    \rho_k(T) T k^{\min(q+\mu,p)}\sum_{j=q+1}^{p}{\max_{0\leq \tau\leq t_n}\|u^{(\nu+1+j)}(\tau)\|_{\HH^{\mu}}}.
    $$
    for any $\mu\geq 0$ for which the right-hand side is finite.      

    \emph{Step 4:}
    In order to get the ``full order estimate'', we note that we can
    consider the semigroup $w(t):=\AA^{-M} u(t)$. Which satisfies
    $w(t) \in \dom(\AA^{M})$ by construction. Then 
    \eqref{eq:diff_of_cutoff_semigroups} gives:
    \begin{multline*}
      \|\AA^M \opC_\ell \AA^{-M} (u(t_n)-u_n)\|_{\HH} \\
      \lesssim 
        \big(\rho_k(T)(1+T)^{L(\opT_0,M)-\ell+1} \big)
        k^{q+M}
        \sum_{j=q+1}^{p}{\|u^{(L(\TT_0,M)+1+j)}\|_{\HH}}
        + \bigO(k^{q+M}).
    \end{multline*}    
    It is an easy proof by induction that one can write
    \begin{align*}
      \opT_0&=\opT_0 \AA^{M} \AA^{-M} = 
             \sum_{\ell=0}^{M}{{M \choose \ell}\AA^{M} \opC_\ell \AA^{-M}.} 
    \end{align*}
    We can then use the previous estimate to bound all the terms on the right hand side
    and get the stated result. The stated regularity assumptions follow from
    more carefully tracking the consistency terms.
  \end{proof}

  \begin{remark}
    In the definition of $Z_{\omega,\delta}$ it is sufficient to only avoid the points
    where $q_j(z)/(1-r(z))$ has a pole. We chose the more standard assumption,
    which is also made in \cite{mallo_palencia_optimal_orders_rk}, because we are not aware
    if there are any methods for which the weakened assumption would offer an advantage.
  \end{remark}
  \begin{remark}
    The previous results all work with the worst-case of $u \notin \dom(\AA)$. If
    $u$ satisfies the boundary conditions up to some order, the powers of $T$ can
    be reduced. Similarly, for strongly A-stable RK-methods we can follow
    \cite{mallo_palencia_optimal_orders_rk} to get one extra order for the
    base-line approximation. This also reduces the power of $T$ by one.
  \end{remark}
\end{document}